\documentclass[reqno,10pt]{article} 
\usepackage{mathrsfs,amsmath,amsthm,amssymb,amsfonts,bbm}
\usepackage{verbatim,pifont,empheq}
\usepackage{esint}
\usepackage{ulem}

\usepackage{color}
\usepackage[hyperfootnotes=false]{hyperref}
\hypersetup{
  colorlinks,
  citecolor=blue,
  linkcolor=red,
  urlcolor=blue}

\usepackage[left=1.4in, right=1.4in, bottom=1.4in]{geometry}
\allowdisplaybreaks[3]

\newtheorem{theorem}{Theorem}[section]
\newtheorem{lemma}[theorem]{Lemma}
\newtheorem{proposition}[theorem]{Proposition}
\newtheorem{corollary}[theorem]{Corollary}
\theoremstyle{definition}

\theoremstyle{remark}
\newtheorem{remark}[theorem]{Remark}

\numberwithin{equation}{section}

\newcommand{\Ag}[1]{\langle#1\rangle}
\newcommand{\cL}{\mathcal{L}}

\newcommand{\U}{\mathcal{U}}
\newcommand{\X}{\mathcal{X}}

\newcommand{\Z}{\mathbb{Z}}

\newcommand{\R}{\mathbb{R}}
\newcommand{\T}{\mathbb{T}}
\newcommand{\N}{\mathbb{N}}
\newcommand{\e}{\varepsilon}

\newcommand{\va}{\varepsilon}

\newcommand{\pa}{\partial}

\newcommand{\lm}{\lambda}
\newcommand{\Lm}{\Lambda}

\newcommand{\ga}{\gamma}
\newcommand{\de}{\delta}

\newcommand{\na}{\nabla}

\newcommand{\Sym}{\textbf{Sym}}
\newcommand{\bfD}{\mathbf{D}}

\newcommand{\txt}[1]{\text{#1}}

\newcommand{\xy}{\mathbf{xy}}
\newcommand{\xx}{\mathbf{xx}}
\newcommand{\yy}{\mathbf{yy}}
\newcommand{\yx}{\mathbf{yx}}
\newcommand{\x}{\mathbf{x}}
\newcommand{\y}{\mathbf{y}}

\DeclareMathOperator*{\esssup}{ess\,sup}

\begin{document}
\title{Optimal convergence rates in multiscale elliptic homogenization}

\author{
Weisheng Niu 
\thanks{School of Mathematical Science, Anhui University,
Hefei,  China. \textit{Email address:} \texttt{niuwsh@ahu.edu.cn}} 
\and
Yao Xu \thanks{School of Mathematical Sciences, University of Chinese Academy of Sciences, Beijing, 100049, China. \textit{Email address:} \texttt{xuyao89@gmail.com}}
\and
Jinping Zhuge \thanks{Morningside Center of Mathematics, Academy of Mathematics and Systems Science, Chinese Academy of Sciences, Beijing 100190, China.
\textit{Email address:} \texttt{jpzhuge@amss.ac.cn}} 
}

\date{}

\maketitle

\begin{abstract}
    This paper is devoted to the quantitative homogenization of multiscale elliptic operator $-\nabla\cdot A_\varepsilon \nabla$, where $A_\varepsilon(x) = A(x/\e_1, x/\e_2,\cdots, x/\e_n)$, $\varepsilon = (\varepsilon_1, \varepsilon_2,\cdots, \varepsilon_n) \in (0,1]^n$ and $\varepsilon_i > \varepsilon_{i+1}$. We assume that $A(y_1,y_2,\cdots, y_n)$ is 1-periodic in each $y_i \in \mathbb{R}^d$ and real analytic. Classically, the method of reiterated homogenization has been applied to study this multiscale elliptic operator, which leads to a convergence rate limited by the ratios $\max \{ \varepsilon_{i+1}/\varepsilon_i: 1\le i\le  n-1\}$. In the present paper, under the assumption of real analytic coefficients, we introduce the so-called multiscale correctors and more accurate effective operators, and improve the ratio part of the convergence rate to $\max \{ e^{-c\varepsilon_{i}/\varepsilon_{i+1}}: 1\le i\le n-1 \}$. This convergence rate is optimal in the sense that $c>0$ cannot be replaced by a larger constant. As a byproduct, the uniform Lipschitz estimate is established under a mild double-log scale-separation condition.
    

\end{abstract}

\tableofcontents

\section{Introduction}

\subsection{Motivations}
The homogenization theory for elliptic equations with one single microscopic scale has been well understood, for either deterministic or random coefficients. However, many hierarchical composite materials (such as biological tissues and nanocomposites) are not characterized by just one microscopic scale but rather by a hierarchy of scales, i.e., a composite material with inclusions at one scale and further fine structures at smaller scales within those inclusions \cite{Lakes93, FW07,LM00}. The multiscale nature also arises in fluid dynamics with a close relation to the evolution of turbulence;  see \cite{FP94,Fr95,MK99,JK99,AV25,BSW23} and references therein.
In this paper, we study the elliptic equations (or systems) with variable coefficients oscillating at multiple microscopic scales. Precisely, we consider
\begin{equation}\label{eq.Le=f}
    -\nabla\cdot A_\e(x) \nabla u_\e = f,
\end{equation}
in a bounded domain $\Omega \subset \R^d$. The coefficient matrix $A_\e$ takes a form of 
\begin{equation*}
    A_\e(x) = A\Big(\frac{x}{\e_1}, \frac{x}{\e_2}, \cdots, \frac{x}{\e_n} \Big),
\end{equation*}
where $\e = (\e_1, \e_2, \cdots, \e_n) \in (0,1]^{n}$. Importantly, we assume that $A(y_1,y_2,\cdots, y_n)$ is 1-periodic in each $y_i = (y_{i}^1, y_{i}^2, \cdots, y_{i}^d) \in \R^d$. Without loss of generality, we assume $1 \gg \e_1 > \e_2 > \cdots > \e_n >0$.  The ellipticity and regularity assumptions of $A$ will be given later precisely.

The homogenization for equation \eqref{eq.Le=f} was first rigorously studied in the classical monograph \cite{BLP78} for the case $\e_i = \e^i$ for $\e \in (0,1)$; also see \cite{Ave87} for linear elasticity. In this case, since $\e$ is small, the scales $\e^i$'s satisfy $\e \gg \e^2 \gg \e^3 \cdots \gg \e^n$, namely, these $n$ distinct scales are all well-separated from each other. Thus a natural strategy can be applied: first homogenize with respect to the smallest microscopic scale $\e_n$ while keeping the remaining $n-1$ larger scales, and then homogenize the second smallest scale $\e_{n-1}$ while keeping the remaining $n-2$ scales, and repeat this process until all the microscopic scales are homogenized. As a consequence, the optimal $O(\e)$ convergence rate was obtained in \cite{BLP78}. This homogenization process is called reiterated homogenization (probably known already on the physical level in the 1930’s \cite{Bru35}),
whose outcome is a homogenized equation with constant coefficient matrix. This homogenized matrix is uniquely determined by $A(y_1, y_2, \cdots, y_n)$ and independent of the microscopic scales $\e$. Later on, Allaire and Briane \cite{AB96} established the qualitative reiterated homogenization under a scale-separation condition, namely,
\begin{equation}\label{cond.separate}
	\lim_{\e\to 0} \e_1 = 0, \quad \lim_{\e \to 0} \frac{\e_{i+1}}{\e_i} = 0, \quad \txt{for all } i=1,2,\cdots, n-1.
\end{equation}
The condition \eqref{cond.separate} is the minimal assumption to guarantee a unique homogenized equation independent of the micro-scales $\e$; see the example in Section \ref{sec.eg1}. Since then, the method of reiterated homogenization has been extended to various kinds of models, such as nonlinear elliptic equations/functionals \cite{BL00,LLPW01,MV05}, parabolic equations \cite{HSW05,Wou10}, etc. The effective numerical algorithms to solve \eqref{eq.Le=f} have also been developed; see e.g. \cite{HS11,KORS22,HJZ24}. 

Recently, by the method of reiterated homogenization, the quantitative convergence rate for \eqref{eq.Le=f} was obtained by Niu, Shen and Xu \cite{NSX20}:
\begin{equation}\label{rate.LipA}
    \| u_\e - u_0 \|_{L^2(\Omega)} \le C\Big(\e_1 + \frac{\e_2}{\e_1} + \cdots + \frac{\e_n}{\e_{n-1}} \Big) \| f \|_{L^2(\Omega)},
\end{equation}
where $u_0$ is the solution of the homogenized equation under \eqref{cond.separate}. It was shown in \cite{NSX20} that the above convergence rate is optimal at least for Lipschitz coefficients (also see \cite{Niu24} for parallel results for parabolic equations with several temporal scales). More recently, the reiterated homogenization method has been used to obtain the anomalous dissipation in scalar turbulence \cite{AV25,BSW23}.

On the other hand, we are also interested in the uniform regularity for the multiscale elliptic equation \eqref{eq.Le=f}. With the convergence rate \eqref{rate.LipA}, the uniform Lipschitz estimate was established in \cite{NSX20} under a quantitative well-separation condition
\begin{equation}\label{cond.well-separate}
	\e_{i+1} \lesssim \e_i^{1+\alpha} \quad \txt{for all } i = 1,2,\cdots, n-1,
\end{equation} 
for some $\alpha>0$. The condition \eqref{cond.well-separate} is often too strong in applications to multiscale self-similar materials. For example, the case $\e_1 = \e$ and $\e_2 = \e/|\log \e|$ is not included in the condition \eqref{cond.well-separate}. More recently, for arbitrary $\e \in (0,1]^n$ without any scale-separation condition, Niu and Zhuge established the uniform $C^\alpha$ regularity for any $\alpha \in (0,1)$ by a compactness method \cite{NZ23} and the $W^{1,p}$ estimates for any $p\in (1,\infty)$ by a quantitative method \cite{NZ24}. However, the uniform Lipschitz estimate for arbitrary $\e \in (0,1]^n$ remains to be an open problem. We mention that in some special cases, such as $\e_{i}/\e_1$ are all constants of order $O(1)$ satisfying certain Diophantine condition, the quantitative convergence rates and uniform Lipschitz estimate are valid; see \cite{Koz78,Shen15,AGK16,SZ18,NZ23}.

In this paper, we reconsider the optimality of the convergence rates for the equation \eqref{eq.Le=f}.
The convergence rate \eqref{rate.LipA} obtained in \cite{NSX20} is limited by the ratios $\e_{i+1}/\e_i$, which could be extremely slow. For example, in the aforementioned case $n=2$, $\e_1 = \e$ and $\e_2 = \e/|\log \e|$, the convergence rate is $O(|\log \e|^{-1})$, which is unsatisfactory in application. We emphasize again that it has been seen in \cite{NSX20} that this convergence rate is in general optimal at least for Lipschitz coefficients by the method of reiterated homogenization.
In view of the above fact, a natural belief in classical homogenization is that, in order to improve the ratio terms of convergence rate, one has to
introduce more correction terms that include the ones of order $O(\e_{i+1}/\e_i)$. Surprisingly, this is not the case for the ratio terms in multiscale homogenization. In the present paper, we will show that the possible slow convergence rates are not due to the missing correction terms, but the regularity of coefficients and the reiterated homogenization method itself, that lead to an inaccurate homogenized equation. The main flaw in the reiterated homogenization method is that it actually only homogenizes one (the smallest) scale at a time and the different coupled scales are not interacting each other. The lack of interaction between distinct scales, particularly for non-well-separated or almost resonant scales, yields inaccuracy in homogenization. To solve this issue, we need to develop a true multiscale method that can handle the interaction between different scales and give the optimal effective equations and convergence rates.

\subsection{Main results}
We assume that $A = A(y_1,y_2,\cdots, y_n)$ satisfies the following assumptions:
\begin{itemize}
    \item Ellipticity and boundedness: there exists $\lambda>0$ such that for any $\xi \in \R^d$,
    \begin{equation}\label{as.ellipticity}
        \xi\cdot A\xi \ge \lambda |\xi|^2, \qquad |A\xi| \le \lambda^{-1} \xi.
    \end{equation}

    \item Periodicity: for any $z_1,\cdots, z_n \in \Z^d$,
    \begin{equation}\label{as.periodicity}
        A(y_1+z_1, y_2+z_2,\cdots, y_n + z_n) = A(y_1,y_2,\cdots, y_n).
    \end{equation}

    \item Real analyticity: there exist $C_0$ and $\Lambda_0 > 0$ such that for any $\ell \ge 0$,
    \begin{equation}\label{as.analyticity}
        |\bfD_n^\ell A| \le C_0 \Lambda_0^{\ell} \ell!, 
    \end{equation}
    where $\bfD_n = (\nabla_{y_1}, \nabla_{y_2},\cdots, \nabla_{y_n})$.
\end{itemize}
The constants $(d, n, \lambda, C_0, \Lambda_0)$ will be called the characters of $A$.

Now we state our main theorem. For the sake of simplicity and convenience, the theorem is stated without seeking generality (for example, the same result holds for elliptic systems, or  different boundary conditions; see \cite{shenbook1}).
\begin{theorem}\label{thm.MainRate}
    Assume $A$ satisfies \eqref{as.ellipticity}, \eqref{as.periodicity} and \eqref{as.analyticity}. Let $\Omega$ be a bounded $C^{1,1}$ domain and $u_\e$ be the weak solution of
    \begin{equation*}
        \left\{ \begin{aligned}
            -\nabla\cdot A_\e \nabla u_\e &= f \quad \text{in } \Omega,\\
            u_\e &= g \quad \text{on } \partial \Omega.
        \end{aligned}
        \right.
    \end{equation*}
    Then there exists a constant (effective) matrix $\overline{A}$ (depending on $A$  and $\e_{i}/\e_{1}$) such that if $u_0$ is the weak solution of
    \begin{equation*}
        \left\{ \begin{aligned}
            -\nabla\cdot \overline{A} \nabla u_0 &= f \quad \text{in } \Omega,\\
            u_0 &= g \quad \text{on } \partial \Omega,
        \end{aligned}
        \right.
    \end{equation*}
    then
    \begin{equation}\label{est.MainRate}
        \| u_\e - u_0 \|_{L^2(\Omega)} \le C\big( \e_1 + \max_{1\le i\le n-1} \{ e^{-c\e_{i}/\e_{i+1}} \} \big) \big\{ \|f\|_{L^2(\Omega)} +  \|g\|_{H^{3/2}(\partial\Omega)} \big\} ,
    \end{equation}
    where $C$ and $c$ are constants depending only on the characters of $A$ and $\Omega$.
\end{theorem}

The above theorem substantially improves the ratio parts of the convergence rate from linear $\e_{i+1}/\e_{i}$ to exponential $e^{-c\e_i/\e_{i+1}}$ when the coefficients are real analytic. This exponential rate is optimal in the sense that it does not hold generally if we replace $c$ by some larger constant $c_1$; see the counterexample in Section \ref{sec.eg2}. An inspect of our proof also shows that if $A \in C^m(\T^{n\times d})$ for some finite $m$, we can also get a rate with a polynomial of $\e_{i+1}/\e_{i}$. Recall that in the classical homogenization theory with one scale, the optimal convergence rate of $O(\e)$ actually is independent of the regularity of the coefficients (precisely, bounded measurable coefficients suffice). Thus $O(\e_1)$ part in \eqref{est.MainRate} is also optimal and cannot be improved in general (unless the higher-order correctors and boundary layers are introduced).

It is important to point out that the effective matrix $\overline{A}$ in Theorem \ref{thm.MainRate} is derived through a new multiscale method involving multiscale correctors, and therefore depends not only on $A$, but also on the ratios $\e_{i}/\e_1$, which indicates that the construction of $\overline{A}$ takes into account the interaction between different scales. It turns out that $\overline{A}$, as an effective matrix, is more accurate (actually optimal in some sense) than the homogenized matrix $A_0$ derived merely by reiterated homogenization under \eqref{cond.separate}, which is independent of $\e$. Moreover, we can precisely measure the difference between $\overline{A}$ and $A_0$ (see Proposition \ref{prop.barA-A0} and Remark \ref{rmk.barA-A0} for details), which perfectly explains the difference between the errors in \eqref{est.MainRate} and \eqref{rate.LipA}.

We now explain briefly why the analyticity/smoothness of the coefficients leads to a better convergence rate. As mentioned earlier, the slow convergence rate is caused by the resonance between close scales and the method of reiterated homogenization itself that ignores the interaction between different scales. The higher order regularity of the coefficients indicates the fast decay of amplitude of high-frequency oscillation that causes resonance between one scale and the next smaller scale. Consequently, the influence of resonance between two close scales can be dramatically reduced for analytic coefficients. This phenomenon seems to be very generic and we give a toy example in Section \ref{sec.eg3} that illustrates how the analyticity/smoothness helps improve the convergence rate in a simple averaging process with two oscillating scales. We recommend that readers look at the example before reading the detailed proof of Theorem \ref{thm.MainRate}. 
Overall, our main results and their proofs, together with the examples in Section \ref{sec.egs}, give a relatively comprehensive picture, including possibilities and limits, on the homogenization of multiscale elliptic equations.

As a general principle in elliptic homogenization, once we obtain an algebraic convergence rate, we can establish the corresponding larger-scale regularity. Particularly, we can show the following (interior) uniform Lipschitz regularity under a mild double-log scale-separation condition.

\begin{theorem}\label{thm.lip.est}
    Assume $A$ satisfies \eqref{as.ellipticity}, \eqref{as.periodicity} and \eqref{as.analyticity}. Let $u_\e \in H^1(B_1)$ be a weak solution of $-\nabla\cdot A_\e \nabla u_\e = f$ in $B_1$ with $f\in L^p(B_1)$ for some $p>d$. Then there exists some constant $M>0$ such that if for all $1\le i\le n-1$,
    \begin{equation}\label{cond.SS4Lip}
    \frac{\e_i}{\e_{i+1}} \ge M \log \log \e_i^{-1},
    \end{equation}
    then
    \begin{equation}\label{mian-re-thm2}
        \| \nabla u_\e \|_{L^\infty(B_{1/2})} \le C\big( \| u_\e \|_{L^2(B_1)} + \| f \|_{L^p(B_1)} \big),
    \end{equation}
    where $M$ and $C$ depend only on $p$ and the characters of $A$.
\end{theorem}

Compared with the uniform Lipschitz estimate in \cite{NSX20},
the scale-separation condition \eqref{cond.SS4Lip} is much weaker than \eqref{cond.well-separate} due to the double logarithm. As an open problem mentioned earlier, it is not known if \eqref{cond.SS4Lip} can be completely removed.

\begin{remark}
    In this paper, we only consider finite number of scales $n$ and the constants in the estimates tend to infinity as $n \to \infty$. Though physically not relevant, mathematically one may also consider infinitely many scales $\e_1, \e_2, \cdots,$ with $\e_k \to 0$ as $k\to \infty$; see e.g. \cite{AB96,JK99}. In this case, we need to identify the particular structure in defining coefficients with infinitely many variables. Typical examples include
    \begin{equation*}
        A_\e(x) = \sum_{k = 1}^\infty \sigma_k A_k(\frac{x}{\e_1},\cdots, \frac{x}{\e_k}),
    \end{equation*}
    and
    \begin{equation*}
        A_\e(x) = \prod_{k = 1}^\infty \Big(I + \sigma_k A_k(\frac{x}{\e_1},\cdots,\frac{x}{\e_k}) \Big),
    \end{equation*}
    where $I$ is the $d\times d$ identity matrix, and each $A_k$ is analytic, periodic and of order $O(1)$. We need $\sum_{k\ge 1} |\sigma_k| < \infty$ to make the coefficients always well-defined as $\e_k$ varies. If $\sigma_k$ decays fast enough, we can truncate the coefficients at $k\le N$ such that $A_\e = A_\e^N + O(\sigma_N)$ and $A_\e^N$ has exactly $N$ microscopic scales. Then we apply our theorem to $A_\e^N$ and get an error consisting of two parts, one from the convergence rate via Theorem \ref{thm.MainRate} (growing as $N$ increases) and the other from the difference $A_\e - A_\e^N$ (decaying as $N$ increases). Since our method is quantitative, it is possible to choose the best $N$ (and the best effective equation) to minimize the error.
\end{remark}

\subsection{Strategy of the proof}

The key of the proof is to construct the so-called multiscale correctors that take into consideration the interactions between different scales. Here we explain how these multiscale correctors are introduced naturally. First, we notice $A(x/\e_1, x/\e_2,\cdots, x/\e_n)$ can be written as $B_\delta(x/\e_1)$, where
\begin{equation}\label{eq.B.str}
    B_{\delta}(x) = A(x, x/\delta_2, \cdots, x/\delta_n), \quad \text{with } \delta_i = \e_i/\e_{1}.
\end{equation}
Then $B_\delta(x/\e_1)$ can be viewed as a coefficient matrix with only one oscillating scale $\e_1$. However, $B_\delta(x)$ depends on $\delta = (\delta_1, \delta_2, \cdots, \delta_n)$ (we set $\delta_1 = 1$), and it is quasi-periodic and still rapidly oscillating in $x$ since $\delta_i$ are typically very small. Nevertheless, we can still try to construct the correctors associated to the matrix $B_\delta$ by solving the following equation:
\begin{equation}\label{eq.Corrector}
    -\nabla\cdot B_\delta \nabla \chi^j_\delta = \nabla \cdot (B_\delta e_j)  \quad \text{ in } \R^d.
\end{equation}
If we can find a bounded solution of \eqref{eq.Corrector}, then we are able to find the effective equation and homogenize all the scales simultaneously in a single step.
Clearly, the corrector equation \eqref{eq.Corrector} has quasi-periodic, multiscale nature and seems difficult to solve in general. We start by guessing the multiscale structure of the solution $\chi_\delta$, if they exist. Actually, since $B_\delta$ has a structure as in \eqref{eq.B.str}, it is natural to guess the solution $\mathcal{X}_\delta = \mathcal{X}_\delta^j$ possessing the same multiscale structure
\begin{equation}\label{eq.Xd}
    \mathcal{X}_\delta(x) = \mathcal{X}(x, x/\delta_2, \cdots, x/\delta_n),
\end{equation}
where $\mathcal{X}(y_1,y_2,\cdots, y_n)$ is periodic in each $y_i \in \R^d$.
Moreover, as $\delta_i$ varies, $\mathcal{X}_\delta$ continuously depends on $\delta_i$. This means that we may need the complete information of $\mathcal{X}(y_1,y_2,\cdots,y_n)$ for $n$ independent variables $y_i \in \R^d$. For this reason, we need to, as usual, lift the system to a degenerate system in $d\times n$ dimensional space. Precisely,
\begin{equation}\label{eq.X.torus}
    -\widehat{\nabla}_n \cdot A(y_1,\cdots, y_n) \widehat{\nabla}_n \mathcal{X}^j = \widehat{\nabla}_n \cdot A(y_1,\cdots, y_n) e_j \quad \text{ in } \T^{d\times n},
\end{equation}
where $\widehat{\nabla}_n$ is certain ``directional gradient'' defined by
\begin{equation*}
    \widehat{\nabla}_n := \sum_{i=1}^n \delta_i^{-1} \nabla_i,
\end{equation*}
and $\nabla_i = \nabla_{y_i}$. Now observe that even though the equation \eqref{eq.X.torus} is periodic, it is extremely degenerate since the ellipticity condition is not satisfied (the dimension of variables is much larger than the dimension of the directional gradient), and the coefficients are singular due to the large factors $\delta_i^{-1}$.

To establish the solvability of the equation \eqref{eq.X.torus}, we take the following strategy.

\textbf{Regularization of the corrector equation.} Instead of working on the equation \eqref{eq.X.torus}, we regularize the equation by adding a positive zero-order term:
\begin{equation}\label{eq.Reg.Y}
    -\widehat{\nabla}_n \cdot A(y_1,\cdots, y_n) \widehat{\nabla}_n Y + \tau^2 Y = \widehat{\nabla}_n \cdot F(y_1,\cdots, y_n) \quad \text{ in } \T^{d\times n},
\end{equation}
where $F$ is a general smooth periodic function. Even though the leading term of equation \eqref{eq.Reg.Y} is still degenerate, we can actually find a unique solution by another regularization, whenever $\tau>0$ and $F$ is smooth enough; see Section \ref{sec.corrector-energy}.

The size of $\tau$ will be crucial for us. On one hand, we have to assume that $\tau$ is small so that the equation \eqref{eq.Reg.Y} with $F = Ae_j$ is indeed a good approximation of \eqref{eq.X.torus}. On the other hand, some crucial estimates of the solution turn out to blow up as $\tau \to 0$. Actually, there are two types of estimates needed for the corrector equation \eqref{eq.Reg.Y}, i.e., the energy estimates and uniform estimates. It is the main challenge of this paper to obtain the uniform estimate for suitably small $\tau$ (depending on $\delta$). We emphasize that the energy estimates of \eqref{eq.Reg.Y} alone are not enough for the uniform boundedness independent of $\tau$ and $\delta$ for the solutions of \eqref{eq.Reg.Y}. The multiscale structure of the equation will play an essential role in overcoming the difficulty. We explain the main idea briefly in the following.

\textbf{Induction on the number of scales.}
In order to get the uniform estimate independent of $\delta = (\delta_1, \cdots, \delta_2)$ for the solution of the equation \eqref{eq.Reg.Y}, we will take advantage of the multiscale structure of the equation and solve it inductively on the number of scales. If $n= 1$, clearly, the equation \eqref{eq.Reg.Y} is periodic and nondegenerate, and thus it is solvable with energy estimate independent of $\tau$. The case $n=2$ will be tricky already but can be solved with uniform estimates by a careful two-scale expansion.  Then we assume that we can solve the equation with $n-1$ scales and  reduce the equation of $n$ scales to the equation with $n-1$ scales. This idea is possible if we apply the following ansatz
\begin{equation}\label{eq.Y.ansatz}
    Y = Y(y_1,\cdots, y_{n-1}, y_n) = \sum_{k=0}^\infty \delta_n^k Y_k(y_1,\cdots, y_{n-1}, y_n),
\end{equation}
and observation
\begin{equation}\label{eq.Dn2Dn-1}
    -\widehat{\nabla}_n\cdot A \widehat{\nabla}_n = -\widehat{\nabla}_{n-1} \cdot A \widehat{\nabla}_{n-1} - \delta_n^{-1} \widehat{\nabla}_{n-1} \cdot A \nabla_n -  \delta_n^{-1} \nabla_n \cdot A \widehat{\nabla}_{n-1} - \delta_n^{-2} \nabla_n \cdot A \nabla_n.
\end{equation}
Taking \eqref{eq.Y.ansatz} and \eqref{eq.Dn2Dn-1} into \eqref{eq.Reg.Y}, we will get a sequence of recursive equations for $Y_k$. To obtain the estimates of $Y_k$, due to the loss of derivatives, we need higher-order derivatives of the coefficients $A$ and $F$. This is exactly where the analyticity of $A$ comes into play. We point out that even for analytic (or trigonometric polynomial) coefficients, the infinite series in \eqref{eq.Y.ansatz} does not converge, as the estimate of $Y_k$ grows as factorial and $\delta_n^k$ decays as exponential in $k$; instead, we have to deal with truncated series and estimate the errors in a proper way. A large part of this paper is devoted to handling these uniform estimates, which lead to considerable complexity in the inductive proofs; see Theorem \ref{thm.2S.Y} and Theorem \ref{thm.Y.nscale} for details.

\textbf{Extra scale-separation condition.}
The above argument to solve the multiscale equation \eqref{eq.Reg.Y} works perfectly, except for an essential flaw due to the degeneracy of equation \eqref{eq.Reg.Y}. To make sure the solution $Y$ of $\eqref{eq.Reg.Y}$ has a uniform bound independent of $\delta$, we require a scale-separation condition: for each $2\le j\le n$, there exists integer $k_j$ such that
    \begin{equation}\label{cond.delta.separation}
        \left\{ \begin{aligned}
            & (\delta_j/\delta_{j-1}) (k_j + \ell_0) \lesssim 1, \\
            & \tau^2 \gtrsim \delta_{j-1}^{-1} e^{-k_j},
        \end{aligned}
        \right.
    \end{equation}
where we have skipped some constants and a precise formulation can be found in Theorem \ref{thm.Y.nscale}. As we also expect  $\tau$ to be small, say $\tau^2 \approx e^{-k_j/2}$, then \eqref{cond.delta.separation} is reduced to
\begin{equation}\label{cond.e.sep}
    \e_j \le \frac{c_j \e_{j-1}}{1+\log(\e_1/\e_{j-1})},
\end{equation}
for some constant $c_j>0$. This extra scale-separation condition seems inevitable in the construction of multiscale correctors and will eventually 
be removed in the main theorem.

In summary, under the scale-separation condition \eqref{cond.e.sep}, we can construct the multiscale correctors from the lifted equation
\begin{equation*}
    -\widehat{\nabla}_n \cdot A(y_1,\cdots, y_n) \widehat{\nabla}_n \mathcal{X}^j + \tau^2 \mathcal{X}^j = \widehat{\nabla}_n \cdot A(y_1,\cdots, y_n) e_j \quad\, \text{ in } \T^{d\times n}
\end{equation*}
as a special case of \eqref{eq.Reg.Y} and obtain their uniform estimates independent of $\delta_i$ and $\tau$. Define $\mathcal{X}_\delta(x)$ as \eqref{eq.Xd}, we see that $\mathcal{X}_\delta$ satisfies the equation
\begin{equation*}
    -\nabla\cdot B_\delta \nabla \mathcal{X}^j_\delta + \tau^2 \mathcal{X}^j_\delta = \nabla \cdot (B_\delta e_j) \quad\, \text{ in } \R^d.
\end{equation*}
Since $\tau$ is very small, $\mathcal{X}_\delta$ will serve a role as multiscale correctors. Moreover, the effective coefficient matrix is given by
\begin{equation*}
    \overline{A} = \Ag{A + A \widehat{\nabla}_n \mathcal{X}},
\end{equation*}
where $\Ag{\cdot}$ means the average taken over the periodic cell $\T^{d\times n}$. Note that our $\overline{A}$ depends essentially on $\delta_i = \e_i/\e_{1}$.

In the construction of flux correctors, we also face degenerate equations and apply a similar argument. However, in this case, we can use essentially the $H^2$ estimate (as the equation for flux correctors is in nondivergence form) to get rid of the factor $\delta_{j-1}^{-1}$ in the second line of \eqref{cond.delta.separation}, which yields a harmless condition $\e_j \le c_j \e_{j-1}$ for some $c_j > 0$.

\textbf{Combination of simultaneous and reiterated homogenization.} 
With the correctors and flux correctors constructed properly, we can establish the convergence rates using the classical approach in elliptic homogenization under the scale-separation condition \eqref{cond.e.sep}. In this case, we actually homogenize all the scales simultaneously in one step, which may be referred as simultaneous homogenization in contrast to reiterated homogenization that homogenizes only one scale in each step. To complete the proof of Theorem \ref{thm.MainRate}, we finally need to remove the extra condition \eqref{cond.e.sep}. But since we do not have multiscale correctors without the condition \eqref{cond.e.sep}, a different idea is needed as follows. Suppose the condition \eqref{cond.e.sep} is not satisfied. Then we can find the largest integer $m \ge 2$ such that the smallest $m$ scales $(\e_{n-m+1}, \e_{n-m+2}, \cdots, \e_n)$ satisfy the condition \eqref{cond.e.sep} (modified with $\e_1$ replaced by $\e_{n-m+1}$), and the smallest $m+1$ scales $(\e_{n-m}, \e_{n-m+1}, \cdots, \e_n)$ do not satisfy the condition. This actually implies that the scale ratio between $\e_{n-m}$ and $\e_{n-m+1}$ is exponentially larger than some scale ratio among $(\e_{n-m+1}, \e_{n-m+2}, \cdots, \e_n)$, namely, 
\begin{equation*}
    \frac{\e_{n-m+1}}{\e_{n-m}} \lesssim e^{-c_p\e_{p-1}/\e_p},
\end{equation*}
for some $p \ge n-m+2$.
In other words, the interaction between the scales $(\e_{n-m+1}, \e_{n-m+2}, \cdots, \e_n)$ and the remaining $n-m$ larger scales is extremely small because they are separated so well. Thus, in the first step we can perform a simultaneous homogenization as above to the smallest $m$ scales and leave the rest $n-m$ scales alone to the next step in the spirit of reiterated homogenization. This first step will give an error as
\begin{equation*}
    \frac{\e_{n-m+1}}{\e_{n-m}} + \max_{n-m+1\le i\le n} \{ e^{-c\e_{i-1}/\e_i} \} \lesssim \max_{n-m+1\le i\le n} \{ e^{-c\e_{i-1}/\e_i} \},
\end{equation*}
where the inequality holds because of the choice of $m$ implies $\frac{\e_{n-m+1}}{\e_{n-m}} \lesssim e^{-c\e_{i-1}/\e_i}$ for some $i \ge n-m+1$; see Section \ref{sec.Remove-Scale-Separation} for details. Finally, we repeat this process until all the scales are homogenized. The key difference between our method sketched above and the classical reiterated homogenization is, for the classical reiterated homogenization, only one scale is homogenized in each step; while in our method we simultaneously homogenize several scales in each step by using multiscale correctors.

The rest of the paper is organized as follows. In Section \ref{sec.notation}, we define notations and provide basics in tensor calculation. In Section \ref{sec.corrector}, we construct multiscale correctors and obtain their energy estimates and uniform estimates. In Section \ref{sec_flux-correctors}, we construct multiscale flux correctors and obtain their energy estimates and uniform estimates. In Section \ref{sec.rate}, we establish the optimal convergence rate in Theorem \ref{thm.MainRate}. In Section \ref{sec.Lip}, we prove the uniform Lipschitz estimate in Theorem \ref{thm.lip.est}. Finally, in Section \ref{sec.egs}, we provide several examples illustrating the optimality of our results.

\textbf{Acknowledgments.}   W. N. is supported by NNSF of China (No.12371106). Y. X. is supported by NNSF of China (No. 12201604, 12371106). J. Z. is supported by NNSF of China (No. 12494541, 12288201, 12471115).

\section{Notations and preliminaries}
\label{sec.notation}

In this section, we introduce some notations and basics that will be used throughout this paper. The readers are recommended to skip this section first and consult it whenever needed. 

\textbf{Tensors.} For rigorous calculation, we need the notation of tensor. Let $1\le l \in \N$. By a tensor of order $k \in \N$, we mean a generalized $k$-dimensional matrix indexed by 
\begin{equation*}
    (j_1,j_2,\cdots, j_k) \in J^k:= \{1,2,\cdots, l \}^k,
\end{equation*}
written as
\begin{equation*}
    T = (T_{j_1j_2\cdots j_k}). 
\end{equation*}
Note that a tensor of order zero is a scalar, a tensor of order 1 is a $l$-dimensional vector and a tensor of order 2 is an $l\times l$ matrix. Denote by $\mathcal{T}^k$ the space of tensors of order $k$. Define the norm of $T \in \mathcal{T}^k$ by
\begin{equation}\label{eq.absTensor}
    | T | = \sum_{(j_1,j_2,\cdots, j_k) \in J^k} |T_{j_1 j_2\cdots j_k}|.
\end{equation}

If $T\in \mathcal{T}^k$ and $S\in \mathcal{T}^m$, then $T\otimes S \in \mathcal{T}^{k+m}$ is defined by
\begin{equation*}
    T\otimes S = (T_{j_1j_2\cdots j_k} S_{j_{k+1}\cdots j_{k+m}}).
\end{equation*}
Obviously $|T\otimes S| = |T| |S|$.

Let $F:\Omega \to \mathcal{T}^k$ be a tensor-valued function defined in $\Omega \subset \R^d$. Then $\nabla F: \Omega \to \mathcal{T}^{k+1}$ is defined as
\begin{equation*}
    \nabla F = (\partial_{x_{j_{k+1}}} F_{j_1j_2\cdots j_k}).
\end{equation*}
By this we can define generally $\nabla^m F: \Omega \to \mathcal{T}^{k+m}$ inductively. 

\textbf{Symmetric tensors and symmetrilization.}
Let $\mathfrak{S}_k$ be the symmetric group of degree $k$ on the symbols $\{ 1,2,\cdots, k \}$. We view an element $\sigma = \{\sigma_1,\sigma_2,\cdots, \sigma_k \}$ of $\mathfrak{S}_k$
as a permutation of $\{ 1,2,\cdots, k \}$.
For $T\in \mathcal{T}^k$, we denote by $\Sym(T)$ the symmetrilization of $T$, given by
\begin{equation*}
    \Sym(T) = \frac{1}{k!} \sum_{\sigma \in \mathfrak{S}_k} (T_{j_{\sigma_1} j_{\sigma_2} \cdots j_{\sigma_k}}).
\end{equation*}
By the triangle inequality, $|\textbf{Sym}(T)| \le |T|$. A tensor $T \in \mathcal{T}^k$ is symmetric if $\Sym(T) = T$. If $f$ is a scalar function, then $\nabla^k f$ is a symmetric tensor of order $k$.

\textbf{Generalized Leibniz rule.} Given two scalar-valued functions $f$ and $g$, it holds
\begin{equation}\label{eq.Leibniz}
    \nabla^m (fg) = \sum_{k=0}^m \binom{m}{k} \Sym (\nabla^k f \otimes \nabla^{m-k}g),
\end{equation}
where the binomial coefficient is given by
\begin{equation*}
    \binom{m}{k} = \frac{m!}{k!(m-k)!}.
\end{equation*}
If $f$ and $g$ are vectors or matrices, we will also use \eqref{eq.Leibniz} by a slight abuse of notation. For example if $A = (a_{ij})$ is a matrix and $w = (w_j)$ is a vector, then $Aw = (Aw)_i = \sum_{j=1}^d a_{ij} w_j$ and \eqref{eq.Leibniz} should be interpreted as
\begin{equation*}
    \nabla^m (Aw)_i = \sum_{k=0}^m \binom{m}{k} \Sym \Big( \sum_{j=1}^d \nabla^k a_{ij} \otimes \nabla^{m-k} w_j \Big).
\end{equation*}


\textbf{Notations of derivatives.}
In order to get quantitative estimates for the degenerate corrector equations \eqref{eq.Reg.Y}, we need higher-order regularity of $A$ and $F$, which are defined on $\R^{d\times n} = \{(y_1,y_2,\cdots, y_n): y_i \in \R^d \}$ or $\T^{d\times n} \simeq [0,1]^{d\times n}$. Let $\bfD_k = (\nabla_1, \nabla_2,\cdots, \nabla_k)$ with $1\le k\le n$ and $\nabla_i = \nabla_{y_i} = (\partial_{y_{i}^1}, \cdots, \partial_{y_{i}^d} )$. For a scalar function $f = f(y_1,y_2,\cdots, y_n)$, $\bfD_k f$ is a $d\times k$ dimensional vector-valued function and $\bfD_k^j f$ is a tensor of order $j$ with dimension $d\times k$. This notation has been used in the assumption \eqref{as.analyticity}.

Note that the generalized Leibniz rule also applies to the gradient $\bfD_k$. That is, \eqref{eq.Leibniz} holds with functions defined on $\R^{d\times n}$ or $\T^{d\times n}$ and $\nabla$ replaced by $\bfD_k$.

Recall that $\widehat{\nabla}_n = \sum_{i=1}^n \delta_i^{-1} \nabla_i$, where $\delta_i$'s are in decreasing order.
Given a tensor-valued function $F$, we observe a simple fact:
\begin{gather}\label{eq.hatD-D}
    |\widehat{\nabla}_n F| \le \sum_{k=1}^n \delta_k^{-1} |\nabla_k F| \le \delta_n^{-1} \sum_{k=1}^n |\nabla_k F| = \delta_n^{-1} |\bfD_n F|.
\end{gather}

\textbf{Function spaces and mixed norms.}
We consider a measurable scalar-valued function $F = F(y_1, y_2, \cdots, y_n)$, 1-periodic in each $y_i \in \R^d$. Since the function is 1-periodic, we will view it as defined in a periodic cell $\T^{d\times n}$.
For defining function spaces with mixed norms, we will use $\x$ and $\y$ to indicate certain norms on the first $n-1$ variable $(y_1,\cdots, y_{n-1})$ and the last variable $y_n$, respectively. Note that the number $n$ may vary according to the context.

We say a scalar-valued function $F \in L^2(\T^{d\times n}) = L^2_{\xy}(\T^{d\times n})$ if $F$ is locally $L^2$ integrable and
\begin{equation*}
    \|F\|_{L^2}^2 = \|F\|_{L^2_{\xy}}^2 := \int_{\T^d} \cdots \int_{\T^d} |F(y_1,\cdots, y_n)|^2 dy_1 \cdots dy_n < \infty.
\end{equation*}
To simplify our notations,  we have used $\| \cdot \|_{L^2}$ as the norm notation instead of $\|\cdot \|_{L^2(\T^{d\times n})}$. This will cause no ambiguity for periodic functions defined on $\T^{d\times n}$.

We say $F \in L^\infty_{\x} L^2_{\y}(\T^{d\times n})$ if
\begin{equation*}
    \|F\|_{L^\infty_{\x} L^2_{\y}}^2: = \esssup_{(y_1,\cdots, y_{n-1}) \in \T^{d\times (n-1)}} \int_{\T^d} |F(y_1,\cdots, y_{n-1}, y_n)|^2 dy_n < \infty.
\end{equation*}
We say $F \in L^2_{\x} H^1_\y(\T^{d\times n})$ if 
\begin{equation*}
    \|F\|_{L^2_{\x}H^1_{\y}}^2 := \int_{\T^d} \cdots \int_{\T^d} |F(y_1,\cdots, y_n)|^2 + |\nabla_n F(y_1,\cdots, y_n)|^2 dy_1 \cdots dy_n < \infty.
\end{equation*}
We say $F \in H^1(\T^{d\times n}) = H^1_{\xy}(\T^{d\times n})$ if
\begin{equation*}
    \|F\|_{H^1}^2 = \|F\|_{H^1_{\xy}}^2 := \int_{\T^d} \cdots \int_{\T^d} |F(y_1,\cdots, y_n)|^2 + |\bfD_n F(y_1,\cdots, y_n)|^2 dy_1 \cdots dy_n < \infty.
\end{equation*}

Similarly, we can also define the spaces $L^\infty(\T^{d\times n}) = L^\infty_{\xy}(\T^{d\times n})$, $L^\infty_{\x} H^1_{\y}(\T^{d\times n})$, $L^2_\x H^2_\y(\T^{d\times n})$, $L^\infty_\x H^2_\y(\T^{d\times n})$, etc. If $F$ is independent of $y_n$, then we will also say
$F = F(y_1,\cdots, y_{n-1}) \in L^2_{\x}(\T^{d\times (n-1)})$ if
\begin{equation*}
    \| F \|_{L^2_{\x}}^2:= \int_{\T^d} \cdots \int_{\T^d} |F(y_1,\cdots, y_{n-1})|^2 dy_1\cdots d y_{n-1} < \infty,
\end{equation*}
and say
$F = F(y_1,\cdots, y_{n-1}) \in H^1_{\x}(\T^{d\times (n-1)})$ if
\begin{equation*}
    \| F \|_{H^1_{\x}}^2:= \int_{\T^d} \cdots \int_{\T^d} |F(y_1,\cdots, y_{n-1})|^2 + |\bfD_{n-1} F(y_1,\cdots, y_{n-1})|^2 dy_1\cdots d y_{n-1} < \infty.
\end{equation*}

\textbf{Norms of tensor-valued functions.}
Let $\| \cdot \|_X$ be a norm for the scalar functions defined on $\T^{d\times n}$. In the same spirit of \eqref{eq.absTensor}, we extend this norm to tensor-valued functions by
\begin{equation*}
    \| F \|_X =  \sum_{(j_1,j_2,\cdots, j_k) \in J^k} \| F_{j_1j_2\cdots j_k}\|_X.
\end{equation*}
Note that $\| \Sym(F) \|_X \le \| F \|_X$.
Moreover, if for scalar functions, we have $\| fg\|_X \le M \| f\|_Y \| g\|_Z$ (with possibly different function spaces $Y$ and $Z$), then for tensor-valued functions $F$ and $G$ we also have
\begin{equation*}
    \| F\otimes G \|_X \le M \| F \|_Y \| G \|_Z.
\end{equation*}
For example, by the H\"{o}lder's inequality for scalar functions, we have $\| F\otimes G \|_{L^\infty_\x L^2_\y} \le \| F\|_{L^\infty_{\xy}} \| G \|_{L^\infty_\x L^2_\y}$.

\begin{remark}
For a multi-index $\beta \in \N^d$ with $|\beta| = m$, we have
\begin{equation*}
    \|\nabla^m f\|_{X} = \sum_{|\beta| = m} \frac{m!}{\beta!} \| \partial^\beta f \|_X.
\end{equation*}
This is because $\nabla^m f$ is symmetric and has $m!/\beta!$ components equal to $\partial^\beta f$. However, $\partial^\beta$ will not be used in this paper.
\end{remark}

\section{Multiscale correctors}
\label{sec.corrector}

\subsection{Multiscale ansatz}\label{sec.corr.ansatz}

In this subsection, we analyze the structure of the solutions of the corrector equation \eqref{eq.Reg.Y}.
Recall the definition of $\widehat{\nabla}_n$ and write
\begin{equation}\label{eq.hatDn}
    \widehat{\nabla}_n = \widehat{\nabla}_{n-1} + \delta_n^{-1} \nabla_n.
\end{equation}
Recall that $\delta_1 = 1, \delta_2,\cdots, \delta_{n-1}$ are in decreasing order. We may assume $\delta_n \ll \delta_{n-1}$ and the solution of \eqref{eq.Reg.Y} takes a form of
\begin{equation}\label{eq.Y2scaleExp}
    Y = Y(y_1,\cdots, y_{n-1}, y_n) = \sum_{k=0}^\infty \delta_n^k Y_k(y_1,\cdots, y_{n-1}, y_n),
\end{equation}
where each $Y_k$ is 1-periodic in each $y_i$. In the above formal expansion, we actually view $\delta_n$ as the only microscopic scale and the rest scales $( \delta_1, \cdots, \delta_{n-1})$ as the same macroscopic scales implicitly enclosed in $Y_k$. Thus both \eqref{eq.hatDn} and \eqref{eq.Y2scaleExp} are essentially two-scale expansions.

Using \eqref{eq.hatDn}, we have
\begin{equation}\label{eq.LLLL}
    - \widehat{\nabla}_n \cdot A \widehat{\nabla}_n = \cL_{\yy} + \delta_n^{-1}\cL_{\xy} + \delta_n^{-1}\cL_{\yx} + \delta_n^{-2}\cL_{\xx},
\end{equation}
where
\begin{equation*}
\begin{aligned}
    & \cL_{\yy} := -\nabla_n\cdot A \nabla_n, \\
    & \cL_{\xy} := -\widehat{\nabla}_{n-1} \cdot A \nabla_n,\\
    & \cL_{\yx} := -\nabla_n \cdot A \widehat{\nabla}_{n-1},\\
    & \cL_{\xx} := -\widehat{\nabla}_{n-1} \cdot A \widehat{\nabla}_{n-1}.
\end{aligned}
\end{equation*}
Here the subscripts $\x$ and $\y$ only for notation purpose refer to the first $n-1$ variables and the last variable, respectively, if there are exactly $n$ independent variables. We will only use them in subscripts for operators, spaces or norms (the exact meaning or definition will be given at its first appearance), and will not use them to represent variables of functions.

Applying \eqref{eq.LLLL} to \eqref{eq.Y2scaleExp}, we have
\begin{equation*}
\begin{aligned}
    -\widehat{\nabla}_n \cdot A\widehat{\nabla}_n Y + \tau^2 Y & = \sum_{k=0}^\infty (\delta_n^{-2} \cL_{\yy} + \delta_n^{-1} (L_{\xy} + \cL_{\yx}) + \cL_{\xx} ) \delta_n^k Y_k \\
    & = \delta_n^{-2} \cL_{\yy} Y_0 + \delta_n^{-1} \Big\{ \cL_{\yy} Y_1 + (\cL_{\xy} + \cL_{\yx}) Y_0 \Big\}\\
    & \qquad + \sum_{k=2}^\infty \delta_n^{k-2} \Big\{ \cL_{\yy} Y_{k} + (\cL_{\xy} + \cL_{\yx}) Y_{k-1} + \cL_{\xx} Y_{k-2} + \tau^2 Y_{k-2}  \Big\}.
\end{aligned}
\end{equation*}
Consequently, if $Y$ is a solution of the  corrector equation \eqref{eq.Reg.Y}, we get a system of recursive equations for $Y_k$:
\begin{subequations}\label{eq.liftphi}
    \begin{empheq}[left=\empheqlbrace]{align}
        & \cL_{\yy} Y_0 = 0, \label{Y0}\\
        & \cL_{\yy} Y_1 + (\cL_{\xy} + \cL_{\yx}) Y_0 = \nabla_n \cdot F, \label{Y1}\\
        & \cL_{\yy} Y_{2} + (\cL_{\xy} + \cL_{\yx}) Y_{1} + \cL_{\xx} Y_{0} + \tau^2 Y_0 = \widehat{\nabla}_{n-1} \cdot F ,\label{Y2}\\
        & \cL_{\yy} Y_{k} + (\cL_{\xy} + \cL_{\yx}) Y_{k-1} + \cL_{\xx} Y_{k-2} + \tau^2 Y_{k-2} = 0, \quad k \ge 3. \label{Yk}
    \end{empheq}
\end{subequations}

In the following, we will show how to formally solve these equations recursively. The following notation for averages of periodic functions will be used:
\begin{equation}\label{def.average}
    \Ag{F}_{y_n} = \int_{\T^d} F(\cdot, y_n) dy_n \quad \text{and} \quad \Ag{F} = \int_{\T^d} \cdots \int_{\T^d} F(y_1, \cdots, y_n) dy_1 \cdots dy_n.
\end{equation}

\textbf{Find $Y_0$:}  Note that $\cL_{\yy}$ is a nondegenerate elliptic operator on the last variable $y_n$ in $\T^d$. Hence, from \eqref{Y0}, we see that $Y_0$ is independent of $y_n$. Thus $\cL_{\xy}Y_0 = 0$ in \eqref{Y1}, which can be written as
\begin{equation}\label{eq.Y1}
    \cL_{\yy} Y_1 = \nabla_n \cdot F + \nabla_n\cdot (A \widehat{\nabla}_{n-1} Y_0).
\end{equation}
Starting from $Y_1$, we will always decompose $Y_k$ as 
\begin{equation*}
    Y_k(y_1,\cdots,y_{n-1},y_n) = Y_k^o(y_1,\cdots,y_{n-1},y_n) + Y_k^r(y_1,\cdots,y_{n-1}),
\end{equation*}
where $Y_k^o$ will be called an ``oscillating'' component (depending on $y_n$) and $Y_k^r$ called a ``regular'' component (independent of $y_n$). Thus $Y_k^r$ cannot be seen by the operator $\cL_{\yy}$ and $\cL_{\xy}$. For example, from \eqref{eq.Y1}, we can only solve for $Y_1^{o}$. Note that the right-hand side of \eqref{eq.Y1} has mean value zero over $y_n$. Generally, we let $Y_k^o$ be the unique solution satisfying  $\Ag{Y_k^o}_{y_n} = 0$. 

We introduce the one-scale corrector (with respect to $y_n$) $\chi(y_1,\cdots, y_n) = (\chi_j(y_1,\cdots, y_n))$, which solves
\begin{equation*}
    \cL_{\yy} \chi_j = \nabla_n\cdot (A e_j).
\end{equation*}
Then it is not difficult to see that
\begin{equation}\label{eq.Y1o.Sol}
    Y_1^o  = \cL_{\yy}^{-1}(\nabla_n\cdot F) + \chi\cdot \widehat{\nabla}_{n-1} Y_0.
\end{equation}
In general if for some 1-periodic functions $Z = Z(y_1,\cdots,y_n)$ and $F = F(y_1,\cdots, y_{n-1})$ satisfying the equation
\begin{equation}\label{eq.Lyy=Lyx}
    \cL_{\yy} Z = \nabla_n\cdot A \widehat{\nabla}_{n-1} F = -\cL_{\yx} F,
\end{equation}
then $Z = -\cL_{\yy}^{-1} \cL_{\yx} F = \chi\cdot \widehat{\nabla}_{n-1} F$.

Now with $Y_1^o$ given by \eqref{eq.Y1o.Sol}, we can use \eqref{Y2} to find the equation for $Y_0$ in $(y_1,\cdots, y_{n-1})$. First taking average in $y_n\in \T^d$ to \eqref{Y2}, we get
\begin{equation*}
    -\widehat{\nabla}_{n-1}\cdot \Ag{A \nabla_n Y_1^o}_{y_n} -\widehat{\nabla}_{n-1}\cdot (\Ag{A}_{y_n} \widehat{\nabla}_{n-1} Y_0) + \tau^2 Y_0= \widehat{\nabla}_{n-1}\cdot (\Ag{F}_{y_n} ).
\end{equation*}
By \eqref{eq.Y1o.Sol}, we have
\begin{equation*}
    \Ag{A \nabla_n Y_1^o}_{y_n} = \Ag{A \nabla_n \cL_{\yy}^{-1}(\nabla_n\cdot F)}_{y_n} + \Ag{A \nabla_n \chi}_{y_n} \widehat{\nabla}_{n-1} Y_0.
\end{equation*}
Combining the last two equations, we arrive at the equation for $Y_0$
\begin{equation}\label{eq.Y0-0}
    -\widehat{\nabla}_{n-1} \cdot \big\{ \Ag{A + A \nabla_n \chi}_{y_n} \big\} \widehat{\nabla}_{n-1} Y_0 +\tau^2 Y_0= \widehat{\nabla}_{n-1}\cdot \big\{ \Ag{F + A \nabla_n \cL_{\yy}^{-1}(\nabla_n\cdot F)}_{y_n}  \big\}.
\end{equation}
Define
\begin{equation*}
    \widehat{A}(y_1,\cdots,y_{n-1}) = \Ag{A + A \nabla_n \chi}_{y_n},
\end{equation*}
and
\begin{equation*}
    \widehat{F}(y_1,\cdots,y_{n-1}) = \Ag{F + A \nabla_n \cL_{yy}^{-1}(\nabla_n\cdot F)}_{y_n}.
\end{equation*}
Then \eqref{eq.Y0-0} is reduced to
\begin{equation}\label{eq.Y0}
    -\widehat{\nabla}_{n-1} \cdot \widehat{A} \widehat{\nabla}_{n-1} Y_0 + \tau^2 Y_0= \widehat{\nabla}_{n-1}\cdot \widehat{F}.
\end{equation}
This equation has exactly the same structure as \eqref{eq.Reg.Y} with one less scale, which suggests that the original multiscale corrector equation may be solved inductively on the number of scales. We point out that the equation \eqref{eq.Y0} guarantees that the mean value over $y_n$ is zero for \eqref{Y2} and therefore we can solve for $Y_2^o$ from that equation in the next step.

\textbf{ Find $Y_1$:} Recall that we write $Y_1 = Y_1^o + Y_1^r$. In the previous step, we have found $Y_1^o$ in a form of \eqref{eq.Y1o.Sol}.
It remains to find $Y_1^r$.

We write \eqref{Y2} as
\begin{equation*}
    \cL_{\yy} Y_2^o = \widehat{\nabla}_{n-1}\cdot F - (\cL_{\xy} + \cL_{\yx}) Y_1^o - \cL_{\yx} Y_1^r - \cL_{\xx} Y_0 - \tau^2 Y_0,
\end{equation*}
where we have used the simple fact $\cL_{\xy} Y_1^r = 0$ as $Y_1^r$ is independent of $y_n$. Using the solution of the general equation \eqref{eq.Lyy=Lyx}, we notice that the solution of the above equation can be formally expressed as
\begin{equation*}
    Y_2^o = \widetilde{Y}_2^o + \chi\cdot \widehat{\nabla}_{n-1} Y_1^r,
\end{equation*}
where $\widetilde{Y}_2^o$ satisfies the equation
\begin{equation*}
    \cL_{\yy} \widetilde{Y}_2^o = \widehat{\nabla}_{n-1}\cdot F - (\cL_{\xy} + \cL_{\yx}) Y_1^o - \cL_{\xx} Y_0 - \tau^2 Y_0.
\end{equation*}
Note that since $Y_0$ and $Y_1^o$ have been found, $\widetilde{Y}_2^o$ can be solved from the above equation.

Next, we consider the equation \eqref{Yk} with $k = 3$
\begin{equation*}
    \cL_{\yy} Y_{3} + (\cL_{\xy} + \cL_{\yx}) Y_{2} + \cL_{\xx} Y_{1} + \tau^2 Y_1= 0.
\end{equation*}
Integrating in $y_n$, we get
\begin{equation*}
\begin{aligned}
    & -\widehat{\nabla}_{n-1}\cdot \Ag{ A \nabla_n Y_2^o}_{y_n} - \widehat{\nabla}_{n-1} \cdot \Ag{ A \widehat{\nabla}_{n-1} Y_1^o}_{y_n} - \widehat{\nabla}_{n-1}\cdot ( \Ag{A}_{y_n} \widehat{\nabla}_{n-1} Y_1^r) + \tau^2 Y_1^r \\
    & = -\widehat{\nabla}_{n-1} \cdot ( \Ag{ A + A \nabla_n \chi }_{y_n} \widehat{\nabla}_{n-1} Y_1^r ) - \widehat{\nabla}_{n-1}\cdot \Ag{ A \nabla_n \widetilde{Y}_2^o}_{y_n}\\
    &\quad- \widehat{\nabla}_{n-1} \cdot \Ag{ A \widehat{\nabla}_{n-1} Y_1^o}_{y_n} + \tau^2 Y_1^r \\
    & = 0.
\end{aligned}
\end{equation*}
As a result, we have
\begin{equation*}
    -\widehat{\nabla}_{n-1} \cdot \widehat{A} \widehat{\nabla}_{n-1} Y_1^r + \tau^2 Y_1^r = \widehat{\nabla}_{n-1}\cdot \Ag{ A \nabla_n \widetilde{Y}_2^o}_{y_n} + \widehat{\nabla}_{n-1} \cdot \Ag{ A \widehat{\nabla}_{n-1} Y_1^o}_{y_n}.
\end{equation*}
This is again the same type of equation as \eqref{eq.Reg.Y} with one less scale and therefore can be solved by induction. Note that the solution $Y_1^r$ satisfies $\Ag{Y_1^r} = 0$. Together with $\Ag{Y_1^o}_{y_n} = 0$, we then have $\Ag{Y_1} = 0$.

\textbf{ Find $Y_k$ for $k\ge 2$:} This can be solved inductively. In the previous calculation, we have completely found $Y_1$ and $Y_2^o$. Now suppose we have completely found $Y_{k-1}$ (including $Y_{k-1}^o$ and $Y_{k-1}^r$) and $Y_k^o$. We then only need to find $Y_{k+1}^o$ and $Y_k^r$.

To this end, we consider the equation \eqref{Yk} with $k$ replaced by $k+1$
\begin{equation*}
    \cL_{\yy} Y_{k+1} + (\cL_{\xy} + \cL_{\yx}) Y_{k} + \cL_{\xx} Y_{k-1} + \tau^2 Y_{k-1}= 0.
\end{equation*}
Using $Y_k = Y_k^o + Y_k^r$, we get
\begin{equation*}
    \cL_{\yy} Y_{k+1}  = - (\cL_{\xy} + \cL_{\yx}) Y_{k}^o - \cL_{\yx} Y_k^r  - \cL_{\xx} Y_{k-1} -\tau^2 Y_{k-1}.
\end{equation*}
Using the solution formula of \eqref{eq.Lyy=Lyx}, we have
\begin{equation}\label{eq.Yok+1A}
    Y_{k+1}^o =  \widetilde{Y}_{k+1}^o + \chi \cdot \widehat{\nabla}_{n-1} Y_k^r,
\end{equation}
where
\begin{equation}\label{eq.tYok+1A}
    \cL_{\yy} \widetilde{Y}_{k+1}^o = - (\cL_{\xy} + \cL_{\yx}) Y_{k}^o - \cL_{\xx} Y_{k-1} - \tau^2 Y_{k-1}.
\end{equation}
Inserting \eqref{eq.Yok+1A} into
\begin{equation*}
    \cL_{\yy} Y_{k+2} + (\cL_{\xy} + \cL_{\yx}) Y_{k+1} + \cL_{\xx} Y_{k} + \tau^2 Y_k = 0
\end{equation*}
and integrating in $y_n$ (using $\Ag{Y_k^o}_{y_n} = 0$), we arrive at
\begin{equation}\label{eq.YrkA}
    -\widehat{\nabla}_{n-1} \cdot \widehat{A} \widehat{\nabla}_{n-1} Y_k^r + \tau^2 Y_k^r = \widehat{\nabla}_{n-1}\cdot \Ag{ A \nabla_n \widetilde{Y}_{k+1}^o}_{y_n} + \widehat{\nabla}_{n-1} \cdot \Ag{ A \widehat{\nabla}_{n-1} Y_k^o}_{y_n}.
\end{equation}
Thus provided that we can solve \eqref{eq.Reg.Y} with $n-1$ scales, we can find $Y_k^r$ from \eqref{eq.tYok+1A} and \eqref{eq.YrkA}, and then $Y_{k+1}^o$ is given by \eqref{eq.Yok+1A}.

We summarize the above procedure in proper order and list the necessary recursive equations as follows:

\begin{itemize}
    \item Step 1: Find $\chi$ and $\widehat{A}$:
    \begin{align}
        & -\nabla_n\cdot A \nabla_n \chi_j  = \nabla_n \cdot (Ae_j), \label{sum-s11}\\
        & \widehat{A}(y_1, y_2,\cdots, y_{n-1}) = \Ag{A + A\nabla_n \chi }_{y_n}. \label{sum.s12}
\end{align}

\item Step 2: Find $Y_0$ by the following three equations in order:
\begin{align}
        & -\nabla_n\cdot A \nabla_n \widetilde{F}  = \nabla_n \cdot F, \label{sum.s21}\\
    & \widehat{F}(y_1,y_2,\cdots, y_{n-1})  = \Ag{ F + A \nabla_n \widetilde{F} }_{y_n}, \label{sum.s22}\\
    & -\widehat{\nabla}_{n-1} \cdot \widehat{A} \widehat{\nabla}_{n-1} Y_0 + \tau^2 Y_0  = \widehat{\nabla}_{n-1} \cdot \widehat{F}. \label{sum.s23}
\end{align}

\item Step 3: Find $Y_1 = Y_1^o + Y_1^r$ by the following three equations in order:
\begin{align}
    &Y_1^o = \widetilde{F} + \chi\cdot \widehat{\nabla}_{n-1} Y_0, \label{sum.s31} \\
    &-\nabla_n \cdot A \nabla_n \widetilde{Y}_2^o = \widehat{\nabla}_{n-1}\cdot F + \widehat{\nabla}_{n-1} \cdot (A \nabla_n Y_{1}^o) + \nabla_n\cdot (A \widehat{\nabla}_{n-1} Y_{1}^o) \label{sum.s32}\\
    &\qquad\qquad\qquad\qquad+ \widehat{\nabla}_{n-1}\cdot ( A \widehat{\nabla}_{n-1} Y_0) - \tau^2 Y_0, \nonumber\\
    &-\widehat{\nabla}_{n-1} \cdot \widehat{A} \widehat{\nabla}_{n-1} Y_1^r + \tau^2 Y_1^r  = \widehat{\nabla}_{n-1}\cdot \big( \Ag{ A \nabla_n \widetilde{Y}_2^o}_{y_n} + \Ag{ A \widehat{\nabla}_{n-1} Y_1^o}_{y_n} \big). \label{sum.s33}
\end{align}

\item Step 4: Find $Y_{k+1} = Y_{k+1}^o+Y_{k+1}^r$ for $k\ge 1$ recursively by the following three equations in order:
\begin{align}
                &Y_{k+1}^o  = \widetilde{Y}_{k+1}^o + \chi\cdot \widehat{\nabla}_{n-1} Y_{k}^r, \label{sum.s41}\\
            & -\nabla_n \cdot A \nabla_n \widetilde{Y}_{k+2}^o  = \widehat{\nabla}_{n-1} \cdot (A \nabla_n Y_{k+1}^o) + \nabla_n\cdot (A \widehat{\nabla}_{n-1} Y_{k+1}^o)  \label{sum.s42}\\ 
            &\qquad\qquad\qquad \qquad\quad+ \widehat{\nabla}_{n-1}\cdot ( A \widehat{\nabla}_{n-1} Y_k) - \tau^2 Y_k, \nonumber\\
            &\hspace{-1em} -\widehat{\nabla}_{n-1}\cdot ( \widehat{A} \widehat{\nabla}_{n-1} Y_{k+1}^r) + \tau^2 Y_{k+1}^r = \widehat{\nabla}_{n-1}\cdot \Ag{ A \nabla_n \widetilde{Y}_{k+2}^o}_y + \widehat{\nabla}_{n-1}\cdot \Ag{ A \widehat{\nabla}_{n-1} Y_{k+1}^o}_y. \label{sum.s43}
\end{align}
\end{itemize}

In the above procedure, except for simple algebraic equations (including \eqref{sum-s11}, \eqref{sum.s22}, \eqref{sum.s31}, \eqref{sum.s41}), we need to solve two types of elliptic PDEs. One type is associated with the operator $\cL_{\yy} = -\nabla_n\cdot A \nabla_n$ (including \eqref{sum-s11}, \eqref{sum.s21}, \eqref{sum.s32}, \eqref{sum.s42}), which is easy to solve as it is a nondegenerate elliptic operator with respect to $y_n$. The other type of equations is associated with $\cL_{\xx} = - \widehat{\nabla}_{n-1} \cdot A \widehat{\nabla}_{n-1}$ (including \eqref{sum.s23}, \eqref{sum.s33}, \eqref{sum.s43}), which is one scale less than the original operator $- \widehat{\nabla}_n \cdot A \widehat{\nabla}_n$. Hence, we expect to solve these equations in an inductive argument on the number of scales.

\subsection{Energy estimates}
\label{sec.corrector-energy}

In this section, we will study the energy estimates of corrector equations in the form of
\begin{equation}\label{eq.RegEq.FG}
    -\widehat{\nabla}_n \cdot A \widehat{\nabla}_n Y + \tau^2 Y = s_1 \widehat{\nabla}_n \cdot F + s_2 G  \quad\text{ in }  \mathbb{T}^{d\times n},
\end{equation}
where $s_i = 0$ or $1$, $\tau>0$ is supposed to be small,  and $F=F(y_1,\cdots,y_n), G=G(y_1,\cdots,y_n)$ are $1$-periodic in each $y_i.$  We add the extra source term $G$ (compared to \eqref{eq.Reg.Y}) for later applications. The equation is very degenerate and its solvability is not obvious. However, we may add another regularization term $-\theta^2 \Delta Y$, where $\Delta = \sum_{k=1}^n \Delta_k = \sum_{k=1}^n \nabla_k \cdot \nabla_k$, to the equation and then consider
\begin{equation}\label{eq.Reg+}
    -\widehat{\nabla}_n \cdot A \widehat{\nabla}_n Y - \theta^2 \Delta Y + \tau^2 Y = s_1 \widehat{\nabla}_n \cdot F + s_2 G  \quad\text{ in }  \T^{d\times n}.
\end{equation}
The equation \eqref{eq.Reg+} is solvable in the energy space $H^1_{\xy}$.
If $F$ and $G$ are smooth enough, due to the term $\tau^2 Y$, we can estimate $\bfD_n^\ell Y$ depending on $\tau$, but independent of $\theta$. Then letting $\theta \to 0$, we can extract a regular periodic solution of \eqref{eq.RegEq.FG}. Some details can be found in the following Proposition \ref{prop.Ytau.F}.

We assume for all $\ell \ge 0$
\begin{equation}\label{cond.A2}
    \|\bfD_n^\ell A\|_{L^\infty_{\xy}} \le C_0 \Lambda_0^\ell \ell!,
\end{equation}
and for some $q\ge 0$ and all $\ell \ge 0$,
\begin{equation}\label{cond.F2}
    \|\bfD_n^\ell F\|_{L^2_{\xy}} \le  \Lambda_1^{\ell} (\ell+q)! \quad \text{ and } \quad \|\bfD_n^\ell G\|_{L^2_{\xy}} \le  \Lambda_1^{\ell} (\ell+q)!.
\end{equation}

\begin{proposition} \label{prop.Ytau.F}
Assume that $A$ satisfies \eqref{as.ellipticity}, \eqref{as.periodicity} and \eqref{cond.A2}. 
Let $Y$ be a weak solution of \eqref{eq.RegEq.FG} with $F,G$ satisfying \eqref{cond.F2}. Then there exists a constant $C_*$ depending only on $C_0, \lambda$ and $d$ such that for $\ell \ge 0$ and $\Lambda = \max \{ \Lambda_1, C_* \Lambda_0 \}$ we have
    \begin{equation*}
        \| \bfD_n^\ell \widehat{\nabla}_n Y \|_{L^2_{\xy}} + \tau \| \bfD_n^\ell  Y \|_{L^2_{\xy}} \le (s_1+\tau^{-1}s_2) C_* \Lambda^{\ell} (\ell+q)!.
    \end{equation*}
\end{proposition}
\begin{proof}
    Let $\Lambda = \max\{ \Lambda_1, C_* \Lambda_0 \}$ for some $C_* > 1$ to be chosen later.
    The energy estimate (multiply the equation against $Y$ and apply the integration by parts) of \eqref{eq.RegEq.FG} yields
    \begin{equation*}
        \| \widehat{\nabla}_n Y \|_{L^2_{\xy}} + \tau \|  Y \|_{L^2_{\xy}} \le C_d \| s_1 F \|_{L^2_{\xy}} + C_d \tau^{-1} \| s_2 G \|_{L^2_{\xy}} \le C_d (s_1+\tau^{-1}s_2 ) q!.
    \end{equation*}
    This is the base case $\ell = 0$ and we need $C_* \ge C_d$.
    Suppose the desired estimate holds for $\ell = 0, 1, \cdots, k-1$, we show it holds for $\ell = k$. Applying $\bfD_n^k$ to the equation \eqref{eq.RegEq.FG}, we obtain a tensor-valued equation for $\bfD_n^k Y$,
    \begin{equation*}
        -\widehat{\nabla}_n \cdot (A \widehat{\nabla}_n \bfD_n^k Y) + \tau^2 \bfD_n^k Y = s_1\widehat{\nabla}_n \cdot \bfD_n^k F + s_2\bfD_n^k G + \widehat{\nabla}_n \cdot \bigg\{  \sum_{m=1}^{k} \binom{k}{m} \Sym( \bfD_n^m A \otimes \bfD_n^{k-m} \widehat{\nabla}_n Y ) \bigg\},
    \end{equation*}
    where we have used the generalized Leibniz rule
    and $\Sym(\cdot)$ denotes the symmetrilization of a tensor; see Section \ref{sec.notation} for the detailed interpretation. 
    By the energy estimate and inductive assumption on $\bfD_n^\ell \widehat{\nabla}_n Y$ for $\ell \le k-1$,
    \begin{equation}\label{est.EnergyY.FG}
        \begin{aligned}
            & \| \bfD_n^k \widehat{\nabla}_n Y \|_{L^2_{\xy}} + \tau \| \bfD_n^k  Y \|_{L^2_{\xy}} \\
            & \le C_d \| s_1 \bfD_n^k F \|_{L^2_{\xy}} + C_d \tau^{-1} \|s_2 \bfD_n^k G \|_{L^2_{\xy}} + C_d \sum_{m=1}^k \binom{k}{m} \| \bfD_n^m A\|_{L^\infty_{\xy}} \| \bfD_n^{k-m} \widehat{\nabla}_n Y \|_{L^2_{\xy}} \\
            & \le C_d (s_1+\tau^{-1}s_2)  \Lambda_1^{k}(k+q)! + C_d (s_1+\tau^{-1}s_2) \sum_{m=1}^k \binom{k}{m} C_0 \Lambda_0^m m! C_* \Lambda^{k-m} (k-m+q)! \\
            & \le C_d (s_1+\tau^{-1}s_2)  \Lambda^{k} (k+q)! + C_d C_0 (s_1+\tau^{-1}s_2) C_*  \Lambda^{k} (k+q)! \sum_{m=1}^k (\Lambda_0/\Lambda)^{m} \\
            & \le \Big(\frac{C_d}{C_*} + \frac{C_d C_0 }{C_*-1}) (s_1+\tau^{-1}s_2) C_* \Lambda^{k} (k+q)!,
        \end{aligned}
    \end{equation}
    where we have used $\Lambda \ge C_* \Lambda_0$ in the last inequality. In the third inequality of \eqref{est.EnergyY.FG}, we have used the simple fact
    \begin{equation}\label{eq.binom-1}
        \binom{k}{m} m! (k-m+q)! \le (k+q)!.
    \end{equation}
    In the last inequality of \eqref{est.EnergyY.FG}, we have used
    \begin{equation}\label{est.GeoSeries}
        \sum_{m=1}^k C_*^{-m} \le \frac{1}{C_* - 1},
    \end{equation}
    whenever $C_*>1$.
    
    We choose $C_*$ large depending only on $C_d$ and $C_0$ such that
    \begin{equation}\label{cond.C*}
        \frac{C_d}{C_*} + \frac{C_d C_0 }{C_*-1} \le 1.
    \end{equation}
    Thus \eqref{est.EnergyY.FG} implies the desired estimate for $\ell = k$ and completes the inductive argument.
\end{proof}

Our second energy estimate is for the equation with operator $\cL_\yy = -\nabla_n \cdot A \nabla_n$ that appears in the recursive equations in the previous subsection.
\begin{proposition}\label{prop.energy2}
Let $A$, $F$ and $G$ satisfy the assumptions in Proposition \ref{prop.Ytau.F}, and $\Ag{G}_{y_n} = 0$. Let $\tau \ge 0$, and $Y$ be the solution of
    \begin{equation}\label{eq.Lyy=FG}
        -\nabla_n \cdot A \nabla_n Y + \tau^2 Y = \nabla_n\cdot F + G\quad\text{ in }\T^d.
    \end{equation}
    There exists $C_*$ depending only on $\lambda, C_0$ and $d$ such that for all $\ell \ge 0$ and $\Lambda = \max\{ \Lambda_1, C_* \Lambda_0 \}$,
    \begin{equation}\label{est.E2}
        \| \bfD_n^\ell Y \|_{L_\x^2 H^1_\y} \le C_* \Lambda^{\ell} (\ell + q)!.
    \end{equation}
    Furthermore, if the assumptions on $F$ and $G$ are replaced by 
    \begin{equation}\label{cond.FG.LinftyL2}
        \| \bfD^\ell_n F \|_{L^\infty_{\x} L^2_\y} \le \Lambda_1^\ell (\ell+q) !\quad \text{and} \quad \| \bfD^\ell_n G \|_{L^\infty_{\x} L^2_\y} \le \Lambda_1^\ell (\ell+q)!,
    \end{equation}
    then the norm of $Y$ in \eqref{est.E2} can also be replaced by $\| \bfD_n^\ell Y \|_{L_\x^\infty H^1_\y}$.
\end{proposition}

\begin{proof}
    The elliptic equation \eqref{eq.Lyy=FG} is with respect to $y_n$, and the rest variables $y_1,\cdots, y_{n-1}$ are viewed as parameters not involved in the differential operators. Due to $\Ag{G}_{y_n} = 0$, the existence of the solution follows from the standard Lax-Milgram theorem. The solution is unique with $\Ag{Y}_{y_n} = 0$.
    Let $\Lambda = \max\{ \Lambda_1, C_* \Lambda_0 \}$ for some $C_* > 1$ to be chosen later.
    The energy estimate of \eqref{eq.Lyy=FG} and Poincar\'{e}'s inequality imply
    \begin{equation*}
        \| Y(y_1,\cdots, y_{n-1}, \cdot) \|_{H^1_{\y}} \le C_d \big(\| F(y_1,\cdots, y_{n-1},\cdot) \|_{L^2_\y} + \| G(y_1,\cdots, y_{n-1},\cdot) \|_{L^2_\y} \big).
    \end{equation*}
    Squaring both sides and integrating over $(y_1, \cdots, y_{n-1})$, we have
    \begin{equation*}
        \| Y \|_{L^2_\x H^1_\y} \le C_d(\| F \|_{L^2_{\xy}} + \| G \|_{L^2_{\xy}}) \le C_d q!.
    \end{equation*}
    This is the base case $\ell = 0$ of \eqref{est.E2} as we choose $C_* \ge C_d$. Suppose the desired estimate holds for $\ell = 0,1,\cdots, k-1$, we show it holds for $\ell = k$. Applying $\bfD_n^k$ to \eqref{eq.Lyy=FG}, we obtain a tensor-valued equation for $\bfD_n^k Y$,
    \begin{equation*}
        -\nabla_n \cdot (A \nabla_n \bfD_n^k Y) + \tau^2 \bfD_n^k Y = \nabla_n \cdot \bfD_n^k F + \bfD_n^k G + \nabla_n \cdot \bigg\{  \sum_{m=1}^{k} \binom{k}{m} \Sym( \bfD_n^m A \otimes \bfD_n^{k-m} \nabla_n Y ) \bigg\}.
    \end{equation*}
    Applying the energy estimate to this equation and then integrating over $(y_1,\cdots, y_{n-1})$ (as the case $\ell = 0$), we obtain
    \begin{equation*}
        \begin{aligned}
            & \| \bfD_n^k Y \|_{L^2_\x H^1_\y} \\
            & \le C_d \| \bfD_n^k F \|_{L^2_{\xy}} + C_d \| \bfD_n^k G \|_{L^2_{\xy}} + C_d \sum_{m=1}^k \binom{k}{m} \| \bfD_n^m A\|_{L^\infty_{\xy}} \| \bfD_n^{k-m} \nabla_n Y \|_{L^2_{\xy}} \\
            & \le C_d  \Lambda_1^{k}(k+q)! + C_d  \sum_{m=1}^k \binom{k}{m} C_0 \Lambda_0^m m! C_* \Lambda^{k-m} (k-m+q)! \\
            & \le C_d   \Lambda^{k} (k+q)! + C_d C_0  C_*  \Lambda^{k} (k+q)! \sum_{m=1}^k (\Lambda_0/\Lambda)^{m} \\
            & \le \Big(\frac{C_d}{C_*} + \frac{C_d C_0 }{C_*-1}) C_* \Lambda^{k} (k+q)!,
        \end{aligned}
    \end{equation*}
    where again we have used \eqref{eq.binom-1}, \eqref{est.GeoSeries} and the assumption $\Lambda \ge C_* \Lambda_0$. This implies the desired estimate with $C_*$ satisfying \eqref{cond.C*}. Finally, under \eqref{cond.FG.LinftyL2}, we only need to replace integrating over $(y_1,\cdots, y_{n-1})$ by taking essential supremum over $(y_1,\cdots, y_{n-1})$, and the corresponding estimate of $\| \bfD_n^\ell Y \|_{L_\x^\infty H^1_\y}$ follows similarly.
\end{proof}

\begin{remark}
    Unlike Proposition \ref{prop.Ytau.F}, the above energy estimate is independent of $\tau$ since the equation only with respect to the variable $y_n$ is not degenerate. We should also point out that both the energy estimates in Proposition \ref{prop.Ytau.F} and Proposition \ref{prop.energy2} hold for all $\ell \ge 0$.
\end{remark}

\subsection{Uniform estimates for two scales}
The estimate of $\| \bfD_n^\ell Y\|_{L^2}$ in Proposition \ref{prop.Ytau.F} depends essentially on $\tau$. In order to get estimates independent of $\tau$, we need to take advantage of the structure of the equation. To achieve this, we begin with the case $n=2$, for which $A = A(y_1, y_2)$ and $\widehat{\nabla}_2 = \nabla_1 + \delta_2^{-1} \nabla_2$.
We will show that the uniform estimate is possible if $\delta_2$ is sufficiently small and $\tau$ is not too small (depending on $\delta_2$).
    
As in \eqref{eq.Y2scaleExp}, we expect to have a solution of \eqref{eq.Reg.Y} with $n=2$ in a form of
\begin{equation}\label{eq.Nk}
    Y(y_1, y_2) = \sum_{j=0}^\infty \delta_2^j Y_j(y_1, y_2).
\end{equation}
The recursive system for $Y_j$ is given by \eqref{eq.liftphi} and can be solved through a sequence of equations \eqref{sum-s11}-\eqref{sum.s43}.
However, even though we can solve all these $Y_j$'s recursively, the convergence of the above series cannot be guaranteed by the energy estimates. Thus we have to look at a partial sum truncated at $j = k$, denoted by $N_k$, and show that $N_k$ satisfies
\begin{equation*}
    -\widehat{\nabla}_2 \cdot A \widehat{\nabla}_2 N_k + \tau^2 N_k = \widehat{\nabla}_2 \cdot F + E_k,
\end{equation*}
with some error $E_k$. With careful quantitative analysis, we can show $E_k$ is small for $k$ suitably larger and thus the error induced by $E_k$ can be controlled by the energy estimate for suitably small $\tau$. The precise statement is the following.

\begin{theorem}\label{thm.2S.Y}
    Let $n=2$. Suppose that $A$ satisfies \eqref{as.ellipticity}, \eqref{as.periodicity}, and \eqref{cond.A2}, and $F$ satisfies \eqref{cond.F2}.  Let $Y$ be the solution of
    \begin{equation*}
        -\widehat{\nabla}_2 \cdot A \widehat{\nabla}_2 Y + \tau^2 Y = \widehat{\nabla}_2\cdot F  \quad\text{ in }  \mathbb{T}^{d\times 2}.
    \end{equation*} Then there exist $C, C_{**}$ (depending on $d, \lambda, C_0$) and $\Lambda =\max \{\Lambda_1, C_{**} \Lambda_0 \}$ such that whenever $\delta_2$ and $\tau$ satisfy 
    $\delta_2 C^3 \Lambda (k+\ell + q) \le e^{-1}$ and $\tau^2 \ge e^{-k+1} C^3 \Lambda$ for some $k \ge 1$ and $\ell \ge 0$, then we have
    \begin{equation*}
     \| \bfD_2^\ell Y \|_{L^2_{\xy}} \le   (C+2) \Lambda^{\ell} (\ell+1+q)!.
\end{equation*}
\end{theorem}

\begin{remark}
    When applying to the corrector equation in two scales, we take $F = Ae_j$ with $C_0 = 1, \Lambda_1 = \Lambda_0$ and $q = 0$. We use the Sobolev embedding theorem to find bounded correctors with bounded gradients, which only requires $\ell > d + 1$. Thus the two constraints in Theorem \ref{thm.2S.Y} can be written as
    \begin{equation}\label{cond.n=2}
        \left\{
        \begin{aligned}
            & \delta_2 C^3 \Lambda (k+d+1) \le e^{-1}, \\
            & \tau^2 \ge e^{-k+1} C^3 \Lambda.
        \end{aligned}
        \right.
    \end{equation} 
    To guarantee the existence for some $k \ge 1$ satisfying both of the above constraints, we need $\delta_2 \le c$ for sufficient  small $c>0$. This scale-separation condition can be assumed without loss of generality for the case $n=2$ since otherwise Theorem \ref{thm.MainRate} is trivial.
    On the other hand, in order to have more accurate correctors, we need $\tau$ as small as possible and thus we would like to choose $k$ as large as possible due to the second constraint in \eqref{cond.n=2}. This implies that the largest possible $k$ is comparable to $\delta_2^{-1}$ and $\tau$ is comparable to $e^{-c/\delta_2} = e^{-c\e_1/\e_2}$ for some $c>0$. 
\end{remark}

\begin{proof}[Proof of Theorem \ref{thm.2S.Y}]
Before the proof, we emphasize that the constant $C$ (depending only on $\lambda, C_0, $ and $d$) will be chosen large enough independent of $\Lambda, \delta_2, k, \ell$ and $\tau$. Other local constants depending on $\lambda, C_0$ and $d$ will be explained during the proof.

In view of \eqref{eq.Nk}, we need to estimate $Y_k$ through a system of equations ordered in \eqref{sum-s11}-\eqref{sum.s43}, divided into 4 steps. The estimates of $Y_k$ for $n=2$ are only based on the energy estimate in Proposition \ref{prop.energy2}.

    \textbf{Step 1:} Estimate of $\chi$ and $\widehat{A}$. 
    
    The equation of $\chi = (\chi_j(y_1, y_2))$ is given by \eqref{sum-s11}, i.e.,
    \begin{equation*}
        -\nabla_2\cdot A \nabla_2 \chi_j  = \nabla_2 \cdot (Ae_j).
    \end{equation*}
    By Proposition \ref{prop.energy2} with $\tau = 0$ and $G=0$,  for all $\ell \ge 0$, we have
    \begin{equation}\label{est.chi.main}
        \| \bfD_2^\ell \chi \|_{L^\infty_{\x} H^1_\y } \le C_* \widehat{\Lambda}_0^\ell \ell!,
    \end{equation}
    where $\widehat{\Lambda}_0 = C_* \Lambda_0$.

    Next, we show the estimate of $\widehat{A}$: for some $\widehat{C}_0$,
    \begin{equation*}
        \| \nabla_1^\ell \widehat{A} \|_{L^\infty_\x} \le \widehat{C}_0 \widehat{\Lambda}_0^\ell \ell!.
    \end{equation*}
    Recall the equation for $\widehat{A} = (\widehat{a}_{ij}(y_1))$ is given by the algebraic equation \eqref{sum.s12}, i.e.,
    \begin{equation*}
        \widehat{A}(y_1) = \Ag{A + A\nabla_2 \chi }_{y_2}.
    \end{equation*}
    Applying $\nabla^\ell_1$ to the equation and using the generalized Leibniz rule, we have
    \begin{equation*}
        \nabla_1^\ell \widehat{A}(y_1) = \Ag{\nabla^\ell_1 A }_{y_2} + \Ag{ \sum_{m=0}^{\ell} \binom{\ell}{m} \Sym( \nabla^m_1 A \otimes \nabla^{\ell-m}_1 \nabla_2 \chi_k ) }_{y_2}.
    \end{equation*}
    It follows from the triangle and H\"{o}lder's inequalities that
    \begin{equation*}
    \begin{aligned}
        \| \nabla^\ell_1 \widehat{A} \|_{L^\infty_\x} & \le \| \nabla^\ell_1 A \|_{L^\infty_{\xy}} + \sum_{m=0}^{\ell} \binom{\ell}{m} \| \nabla^m_1 A \|_{L^\infty_{\xy}} \| \nabla^{\ell-m}_1 \chi \|_{L^\infty_\x H^1_\y} \\
        & \le C_0 \Lambda_0^{\ell} \ell! + \sum_{m=0}^{\ell} \binom{\ell}{m} C_0 \Lambda_0^{m}m! C_* \widehat{\Lambda}_0^{\ell-m} (\ell-m)!\\
        & \le C_0 \Lambda_0^{\ell} \ell! + C_0 C_* \widehat{\Lambda}_0^\ell \ell! \sum_{m=0}^\ell (\Lambda_0 /\widehat{\Lambda}_0)^m \\
        & \le (C_0 + \frac{C_0 C_*^2}{C_*-1}) \widehat{\Lambda}_0^\ell \ell!,
    \end{aligned}
    \end{equation*}
    where we have used the fact $\Lambda_0 /\widehat{\Lambda}_0 \le C_*^{-1}$, the identity
    \begin{equation*}
        \binom{\ell}{m} m!(\ell-m)! = \ell!,
    \end{equation*}
    and
    \begin{equation*}
        \sum_{m=0}^\ell C_*^{-m} \le \frac{C_*}{C_* - 1}.
    \end{equation*}
    We rewrite the constant $C_0 + \frac{C_0 C_*^2}{C_*-1}$ as $\widehat{C}_0$ and obtain the desired estimate.

    \textbf{Step 2:} Estimate of $Y_0$: for some $C$ and $\Lambda=\max\{ \Lambda_1, C_*^2 \Lambda_0 \}$, and any $\ell \ge 0$,
    \begin{equation}\label{est.Y0.l>0}
        \| \nabla^\ell_1 Y_0 \|_{H^1_\x} \le C \Lambda^{\ell} (\ell+q)!.
    \end{equation}
    
    Recall that $Y_0$ is calculated through three equations \eqref{sum.s21}, \eqref{sum.s22} and \eqref{sum.s23}, also listed below:
    \begin{align}
        -\nabla_2 \cdot A \nabla_2 \widetilde{F}  & = \nabla_2 \cdot F, \label{eq.tildeF}\\
        \widehat{F}(y_1) & = \Ag{F + A \nabla_2 \widetilde{F}}_{y_2}, \label{eq.hatF}\\
        -\nabla_1 \cdot \widehat{A} \nabla_1 Y_0 + \tau^2 Y_0 & = \nabla_1 \cdot \widehat{F}.\label{eq.Y0.ik}
    \end{align}
        We first estimate $\widetilde{F}$.
    By Proposition \ref{prop.energy2} applied to \eqref{eq.tildeF}, we have
    \begin{equation}\label{est.DltF}
        \| \bfD_2^\ell \widetilde{F}\|_{L^2_\x H^1_\y} \le C_* \widehat{\Lambda}_1^{\ell} (\ell+q)!,
    \end{equation}
    where $\widehat{\Lambda}_1 = \max\{ \Lambda_1, C_* \Lambda_0 \}$.

    Next, we estimate $\widehat{F}$ using \eqref{eq.hatF}. It follows from the generalized Leibniz rule that for $\ell \ge 0$,
    \begin{equation}\label{est.hatF.2S}
    \begin{aligned}
        \| \bfD_2^\ell \widehat{F} \|_{L^2_{\xy}} & \le \| \bfD_2^\ell F \|_{L^2_{\xy}} + \sum_{m=0}^\ell \binom{\ell}{m} \| \bfD_2^m \chi \|_{L^\infty_{\xy}} \| \bfD_2^{\ell-m} \widetilde{F} \|_{L^2_\x H^1_\y} \\ & \le \Lambda_1^{\ell} (\ell+q)! +  \sum_{m=0}^\ell \binom{\ell}{m} C_* \widehat{\Lambda}_0^m m!  C_* \widehat{\Lambda}_1^{\ell-m} (\ell-m+q)! \\
        & \le \Lambda^{\ell} (\ell+q)! + C_*^2 \Lambda^\ell (\ell+q)! \sum_{m=0}^\ell C_*^{-m} \\
        & \le \Big( 1 + \frac{ C_*^3}{C_*-1} \Big) \Lambda^{\ell} (\ell + q)! \\
        & \le C_*^3 \Lambda^{\ell} (\ell + q)!,
    \end{aligned}   
    \end{equation}
    where we need $\Lambda \ge \max\{ \widehat{\Lambda}_1, C_* \widehat{\Lambda}_0 \} = \max\{ \Lambda_1, C_*^2 \Lambda_0 \}$, and we have used \eqref{eq.binom-1} and chosen $C_* \ge 3$ in the last inequality, without loss of generality. Observe that $\widehat{F}$ is independent of $y_2$ and thus $|\bfD_2^\ell \widehat{F}| = |\nabla_1^\ell  \widehat{F}|$.

    Now we are ready to estimate $Y_0$ by applying Proposition \ref{prop.energy2} to the nondegenerate equation \eqref{eq.Y0.ik} with $n=1$ and $A$ replaced by $\widehat{A}$. The constant $C_*$ from Proposition \ref{prop.energy2} should be updated (depending on $d, \lambda$ and $\widehat{C}_0$), denoted by $\widehat{C}_* \ge C_*$. Define $\Lambda := \max\{ \widehat{\Lambda}_1, \widehat{C}_* \widehat{\Lambda}_0 \} = \max\{ \Lambda_1, \widehat{C}_* C_* \Lambda_0 \}$. Let $C_{**} = \widehat{C}_* C_*$. Hence we now have
    \begin{equation}\label{eq.Lambda.relation}
        \Lambda/ \widehat{\Lambda}_0 \ge \widehat{C}_*, \qquad \Lambda/\Lambda_0 \ge C_{**}.
    \end{equation}
    These two relations will be used frequently through the proof. With the above setting,
    it follows from Proposition \ref{prop.energy2} applied to \eqref{eq.Y0.ik} and \eqref{est.hatF.2S} that
    \begin{align*}
        \|\nabla_1^\ell Y_0\|_{H^1_\x}\leq \widehat{C}_*C_*^3\Lambda^\ell(\ell+q)!\leq C\Lambda^\ell(\ell+q)!,
    \end{align*}
    where we choose $C\geq \widehat{C}_*C_*^3$.

    \textbf{Step 3:} Estimate of $Y_1$.

    The procedure to compute $Y_1$ from $Y_0$ follows the order $Y_0 \to Y_1^o \to \widetilde{Y}_2^o \to Y_1^r$, using the equations \eqref{sum.s31}, \eqref{sum.s32} and \eqref{sum.s33}, also listed below:
    \begin{align}
        Y_1^o & = \widetilde{F} + \chi\cdot  \nabla_1 Y_0, \label{eq.Y1o.n=2}\\
        \begin{split}
            -\nabla_2 \cdot A \nabla_2 \widetilde{Y}_2^o & = \nabla_1 \cdot F + \nabla_1 \cdot (A \nabla_2 Y_{1}^o) + \nabla_2\cdot (A \nabla_1 Y_{1}^o)  \\&\qquad\qquad+ \nabla_1 \cdot ( A \nabla_1 Y_0) - \tau^2 Y_0, \label{eq.tY2o.n=2}
        \end{split}
        \\
        -\nabla_1 \cdot \widehat{A} \nabla_1 Y_1^r + \tau^2 Y_1^r  &= \nabla_1 \cdot \big( \Ag{ A \nabla_2 \widetilde{Y}_2^o}_{y_2} + \Ag{ A \nabla_1 Y_1^o}_{y_2} \big). \label{eq.Y1r.n=2}
    \end{align}
    The first equation \eqref{eq.Y1o.n=2} is an algebraic equation, which could be handled directly by the Leibniz rule. Actually, it follows from \eqref{est.DltF}, \eqref{est.chi.main} and \eqref{est.Y0.l>0} that for $\ell \ge 0$
    \begin{equation}\label{est.Dl2Y1o}
    \begin{aligned}
        \| \bfD_2^\ell Y_1^o \|_{L^2_\x H^1_\y} & \le \| \bfD_2^\ell \widetilde{F} \|_{L^2_\x H^1_\y } + \sum_{m=0}^\ell \binom{\ell}{m} \| \Sym( \bfD_2^m \chi \otimes \bfD_2^{\ell-m} \nabla_1 Y_0 ) \|_{L^2_\x H^1_\y} \\
        & \le C_* \widehat{\Lambda}_1^{\ell} (\ell+q)! + \sum_{m=0}^\ell \binom{\ell}{m} \| \bfD_2^m \chi \|_{L^\infty_\x H^1_\y} \| \nabla^{\ell-m}_1 Y_0 \|_{H^1_\x} \\
        & \le C_*  \Lambda^{\ell} (\ell+q)! + \sum_{m=0}^\ell \binom{\ell}{m} C_* \widehat{\Lambda}_0^m m! C \Lambda^{\ell-m} (\ell-m+q)! \\
        & \le C_*  \Lambda^{\ell} (\ell+q)! + C_* C \Lambda^{\ell} (\ell+q)! \sum_{m=0}^\ell (\widehat{\Lambda}_0 /\Lambda)^{m} \\
        & \le \Big( \frac{C_*}{C} + \frac{C_* \widehat{C}_* }{\widehat{C}_* - 1} \Big) C \Lambda^\ell (\ell+q)! \\
        & \le C^{2} \Lambda^{\ell} (\ell+q)!,
    \end{aligned}
    \end{equation}
    where in the last inequality, we need $C^*/C+C_*\widehat{C}_*/(\widehat{C}_*-1)\le C$.
    
    Now, we use the second equation \eqref{eq.tY2o.n=2} to estimate $\widetilde{Y}_2^o$ by Proposition \ref{prop.energy2}. Let
    \begin{equation*}
        F_2 = A \nabla_1 Y_1^o, \quad \text{and} \quad G_2 = \nabla_1\cdot F + \nabla_1\cdot (A \nabla_2 Y_1^o) + \nabla_1\cdot (A \nabla_1 Y_0) - \tau^2 Y_0.
    \end{equation*}
    Then \eqref{eq.tY2o.n=2} can be written as
    \begin{equation}\label{eq.tY2o.n=2+}
        -\nabla_2 \cdot A \nabla_2 \widetilde{Y}_2^o = \nabla_2 \cdot F + G_2.
    \end{equation}
    Recall that $\Ag{G_2}_{y_2} = 0$, due to \eqref{eq.Y0.ik}. 
    It follows from the generalized Leibniz rule and \eqref{est.Dl2Y1o} that
    \begin{equation}\label{est.F2.n=2}
        \begin{aligned}
            \| \bfD^\ell_2 F_2 \|_{L^2_{\xy}} & \le \sum_{m = 0}^\ell \binom{\ell}{m} \| \bfD^m_2 A \|_{L^\infty_{\xy}} \| \bfD_2^{\ell-m} \nabla_1 Y_1^o \|_{L^2_{\xy}} \\ & \le C_0 \sum_{m=0}^\ell \binom{\ell}{m} \Lambda_0^m m! C^2 \Lambda^{\ell-m+1} (\ell-m+1+q)! \\
            & \le \frac{C_0 C_{**}}{C_{**} - 1} C^2 \Lambda^{\ell+1} (\ell+1+q)!,
        \end{aligned}
    \end{equation}
    where we have used $|\bfD_2^{\ell-m} \nabla_1 Y_1^o| \le |\bfD_2^{\ell-m+1} Y_1^o|$ which contributes to the increasing of factorial by 1 in \eqref{est.F2.n=2}.
    Similarly,
    \begin{equation}\label{est.G2.n=2}
        \begin{aligned}
            \| \bfD_2^\ell G_2 \|_{L^2_{\xy}}  & \le  \| \bfD_2^\ell \nabla_1 F \|_{L^2_{\xy}} + \tau^2 \| \bfD_2^\ell Y_0 \|_{L^2_\x}+ \sum_{m = 0}^\ell \binom{\ell}{m} \| \bfD_2^m A \|_{L^\infty_{\xy}} \| \bfD_2^{\ell-m} \nabla_1 Y_1^o \|_{L^2_\x H^1_\y} \\
            & \qquad + \sum_{m = 0}^\ell \binom{\ell}{m} \| \bfD_2^m \nabla_1 A \|_{L^\infty_{\xy}} \| \bfD_2^{\ell-m} Y_1^o \|_{L^2_\x H^1_\y} \\
            & \qquad + \sum_{m = 0}^\ell \binom{\ell}{m} \| \bfD_2^m \nabla_1 A \|_{L^\infty_{\xy}} \| \bfD_2^{\ell-m} Y_0 \|_{H^1_\x} \\
            & \qquad + \sum_{m = 0}^\ell \binom{\ell}{m} \| \bfD_2^m A \|_{L^\infty_{\xy}} \| \bfD_2^{\ell-m} \nabla_1 Y_0 \|_{H^1_\x} \\
            & \le \Big( 1+\tau^2 + \frac{4C_0 C_{**}}{C_{**}-1} \Big) C^2 \Lambda^{\ell+1} (\ell+1+q)!.
        \end{aligned}
    \end{equation}
    Then by Proposition \ref{prop.energy2} applied to \eqref{eq.tY2o.n=2+} and the linearity of the equation, we have
    \begin{equation}\label{est.tY2o.n=2}
        \| \bfD_2^\ell \widetilde{Y}_2^o \|_{L^2_\x H^1_\y} \le C_* \Big( 1+\tau^2 + \frac{4C_0 C_{**}}{C_{**}-1} \Big) C^2 \Lambda^{\ell+1} (\ell+1+q)! \le C^3 \Lambda^{\ell+1} (\ell+1+q)!.
    \end{equation}
    Here we choose $C \ge C_*( 1+\tau^2 + \frac{4C_0 C_{**}}{C_{**}-1} )$.

    Next, we consider the equation of $Y_1^r$. Let
    \begin{equation*}
        F_1^r = \Ag{ A \nabla_2 \widetilde{Y}_2^o }_{y_2} + \Ag{ A \nabla_1 Y_1^o}_{y_2}.
    \end{equation*}
    Then \eqref{eq.Y1r.n=2} is reduced to
    \begin{equation}\label{eq.Y1r.n=2+}
        -\nabla_1\cdot \widehat{A} \nabla_1 Y_1^r + \tau^2 Y_1^r = \nabla_1 \cdot F_1^r.
    \end{equation}
    We have by \eqref{est.tY2o.n=2} and \eqref{est.Dl2Y1o},
    \begin{equation} \label{est.Fr12.n=2}
        \begin{aligned}
            \| \nabla_1^\ell F_1^r  \|_{L^2_\x} & \le \sum_{m=0}^\ell \binom{\ell}{m} \| \nabla_1^m A \|_{L^\infty_{\xy}} \| \nabla_1^{\ell-m} \widetilde{Y}_2^o \|_{L^2_\x H^1_\y} \\
            & \qquad + \sum_{m=0}^\ell \binom{\ell}{m} \| \nabla_1^m A \|_{L^\infty_{\xy}} \| \nabla_1^{\ell-m} \nabla_1 Y_1^o \|_{L^2_{\xy}} \\
            & \le \frac{2C_0 C_{**}}{C_{**} -1} C^3  \Lambda^{\ell+1 } (\ell+1+q)!.
        \end{aligned}
    \end{equation}
    By Proposition \ref{prop.energy2} applied to the equation \eqref{eq.Y1r.n=2+}, we have
    \begin{equation*}
        \|  \nabla_1^\ell Y_1^r \|_{H^1_\x} \le \widehat{C}_* \frac{2C_0 C_{**}}{C_{**} -1} C^3 \Lambda^{\ell+1} (\ell+1+q)! \le \frac12 C^4 \Lambda^{\ell+1} (\ell+1+q)!,
    \end{equation*}
    where we need $C \ge 2\widehat{C}_* \frac{2C_0 C_{**}}{C_{**} -1}$. Here again the constant $\widehat{C}_*$ might be larger than $C_*$ as we apply Proposition \ref{prop.energy2} with $A$ replaced by $\widehat{A}$. 

    \textbf{Step 4:} Estimate of $Y_k$ for $k\ge 2$. 
    
    From the recursive equations in Section \ref{sec.corr.ansatz}, we see that the estimates of $(Y_k^o, \widetilde{Y}_{k+1}^o, Y_k^r)$ yield the estimates of $(Y_{k+1}^o, \widetilde{Y}_{k+2}^o, Y_{k+1}^r)$.
    We have proved our base case for $k = 1$ as follows:
    \begin{subequations}\label{est.Y1base.n=2}
        \begin{empheq}[left=\empheqlbrace]{align}
            & \| \bfD_2^\ell Y_1^o \|_{L^2_\x H^1_\y} \le C^{2} \Lambda^{\ell} (\ell+q)!, \\
            & \| \bfD_2^\ell \widetilde{Y}_2^o \|_{L^2_\x H^1_\y} \le C^{3} \Lambda^{\ell+1} (\ell+1+q)!, \\
             & \| \nabla^{\ell}_1 Y_1^r \|_{H^1_\x} \le \frac12 C^{4} \Lambda^{\ell+1} (\ell+1+q)!.
        \end{empheq}
    \end{subequations}

    Suppose now that for some $k\ge 1 $ and $\ell \ge 0$ we have
    \begin{subequations}\label{assum.Yk.n=2}
        \begin{empheq}[left=\empheqlbrace]{align}
            & \| \bfD_2^\ell Y_{k}^o \|_{L^2_\x H^1_\y} \le C^{3k-1} \Lambda^{\ell+k-1} (\ell+k-1+q)!, \label{assum.Yk-a}\\
            & \| \bfD_2^\ell \widetilde{Y}_{k+1}^o\|_{L^2_\x H^1_\y} \le C^{ 3k} \Lambda^{\ell+k} (\ell+k+q)!, \label{assum.Yk-b}\\
            & \| \nabla^{\ell}_1 Y_k^r \|_{H^1_\x} \le \frac12 C^{3k+1} \Lambda^{\ell+k} (\ell+k+q)!. \label{assum.Yk-c}
\end{empheq} \end{subequations}
        We would like to show for all $\ell \ge 0$,
\begin{subequations}\label{est.Yk.induction.n=2}
        \begin{empheq}[left=\empheqlbrace]{align}
        & \| \bfD_2^\ell Y_{k+1}^o\|_{L^2_\x H^1_\y} \le C^{3(k+1) -1} \Lambda^{\ell+k } (\ell+k+q)!, \label{est.Yok+1.n=2}\\
        & \| \bfD_2^\ell \widetilde{Y}_{k+2}^o\|_{L^2_\x H^1_\y} \le C^{3(k+1)} \Lambda^{\ell+k+1} (\ell+k+1+q)!,\label{est.tYok+1.n=2}\\
            & \| \nabla^{\ell}_1 Y_{k+1}^r \|_{H^1_\x} \le \frac{1}{2}C^{3(k+1)+1} \Lambda^{\ell+k+1 } (\ell+k+1+q)!. \label{est.Yrk+1.n=2}
\end{empheq}
\end{subequations}

    The proof of \eqref{est.Yk.induction.n=2} is very similar to the estimates in \eqref{est.Y1base.n=2}, based on the recursive equations \eqref{sum.s41}, \eqref{sum.s42} and \eqref{sum.s43}, restated as follows for $n=2$, 
    \begin{align}
            Y_{k+1}^o & = \widetilde{Y}_{k+1}^o + \chi\cdot \nabla_1 Y_{k}^r,\label{eq.Yko.n=2}\\
            \begin{split}
                -\nabla_2 \cdot A \nabla_2 \widetilde{Y}_{k+2}^o & = \nabla_1 \cdot (A \nabla_2 Y_{k+1}^o) + \nabla_2 \cdot (A \nabla_1 Y_{k+1}^o)  \\&\qquad\qquad+ \nabla_1 \cdot ( A \nabla_1 Y_k) - \tau^2 Y_k, \label{eq.tYko.n=2}
            \end{split}
            \\
            -\nabla_1 \cdot ( \widehat{A} \nabla_1 Y_{k+1}^r) + \tau^2 Y_{k+1}^r & = \nabla_1 \cdot \Ag{ A \nabla_2 \widetilde{Y}_{k+2}^o}_{y_2} +\nabla_1 \cdot \Ag{ A \nabla_1 Y_{k+1}^o}_{y_2}.\label{eq.Ykr.n=2}
    \end{align}

    We first use the algebraic equation \eqref{eq.Yko.n=2} to estimate $Y_{k+1}^o$. In fact, by the inductive assumption \eqref{assum.Yk.n=2}, we have
    \begin{equation*}
        \begin{aligned}
            \| \bfD_2^\ell Y_{k+1}^o\|_{L^2_\x H^1_\y} & \le \| \bfD_2^\ell \widetilde{Y}_{k+1}^o \|_{L^2_\x H^1_\y} + \sum_{m=0}^\ell \binom{\ell}{m} \| \bfD_2^m \chi \|_{L^\infty_\x H^1_\y} \| \bfD_2^{\ell-m} Y_{k}^r \|_{H^1_\x} \\
            & \le C^{3k} \Lambda^{\ell+k}(\ell+k+q)! \\
            & \quad + \sum_{m=0}^\ell \binom{\ell}{m} C_* \widehat{\Lambda}_0^m m! \frac12 C^{3k+1} \Lambda^{\ell-m+k} (\ell-m+k+q)! \\
            & \le C^{3k+2} \Lambda^{\ell+k} (\ell+k+q)! \\
            & = C^{3(k+1) - 1} \Lambda^{\ell+k} (\ell+k+q)!.
        \end{aligned}
    \end{equation*}
    This proves the first inequality of \eqref{est.Yok+1.n=2}.

    The estimate of $\widetilde{Y}_{k+2}^o$ is similar to $\widetilde{Y}_2^o$. Let
    \begin{equation*}
        F_{k+2} = A \nabla_1 Y_{k+1}^o \qquad \text{and} \qquad G_{k+2} = \nabla_1\cdot (A \nabla_2 Y_{k+1}^o) + \nabla_1\cdot (A\nabla_1 Y_k) - \tau^2 Y_k.
    \end{equation*}
    Note that by the triangle inequality and the inductive assumption \eqref{assum.Yk.n=2}, 
    \begin{equation}\label{est.Yk.o+r}
        \| \bfD_2^\ell Y_k \|_{L^2_\x H^1_\y} \le \| \bfD_2^\ell Y_k^o \|_{L^2_\x H^1_\y} + \| \nabla_1^\ell Y_k^r \|_{L^2_\x} \le C^{3k+1} \Lambda^{\ell-1+k} (\ell-1+k+q)!,
    \end{equation}
    for $\ell \ge 1$ (the estimate for $\ell = 0$ is the same as $\ell = 1$ by the Poincar\'{e} inequality),
    and for all $\ell \ge 0$,
        \begin{equation*}
        \| \bfD_2^\ell \nabla_1 Y_k \|_{L^2_{\xy}} \le \| \bfD_2^{\ell+1} Y_k^o \|_{L^2_\x H^1_\y} + \| \nabla_1^\ell Y_k^r \|_{H^1_\x} \le C^{3k+1} \Lambda^{\ell+k} (\ell+k+q)!.
    \end{equation*}
    Thus, by these estimates, \eqref{est.Yok+1.n=2}, and a familiar calculation as in \eqref{est.F2.n=2}  and \eqref{est.G2.n=2}, we have
    \begin{equation*}
        \| \bfD_2^\ell F_{k+2} \|_{L^2_{\xy}} \le \frac{C_0 C_{**}}{C_{**} -1} C^{3(k+1) - 1} \Lambda^{\ell+k+1} (\ell+k+1+q)!,
    \end{equation*}
    and
    \begin{equation*}
        \| \bfD_2^\ell G_{k+2} \|_{L^2_{\xy}} \le \big( \frac{4C_0 C_{**}}{C_{**} -1} + \tau^2 \big) C^{3(k+1) - 1} \Lambda^{\ell+k+1} (\ell+k+1+q)!.
    \end{equation*}
    Then applying the Proposition \ref{prop.energy2} to the equation
    \begin{equation*}
        -\nabla_2 \cdot A\nabla_2 \widetilde{Y}_{k+1}^o = \nabla_2\cdot F_{k+2} + G_{k+2},
    \end{equation*}
    we get 
    \begin{align*}
    \begin{split}
       \| \bfD_2^\ell \widetilde{Y}_{k+2}^o\|_{L^2_\x H^1_\y} &\le  C_* \big( \frac{4C_0 C_{**}}{C_{**} -1} + \tau^2 \big) C^{3(k+1)-1} \Lambda^{\ell+k+1} (\ell+k+1+q)! \\
        &\le C^{3(k+1)} \Lambda^{\ell+k+1} (\ell+k+1+q)!,
        \end{split}
    \end{align*}
    where we have assumed $C \ge C_*\frac{4C_0 C_{**}}{C_{**} -1} + \tau^2 $ previously. This proves \eqref{est.tYok+1.n=2}.

    Finally, we estimate $Y_{k+1}^r$. Let
    \begin{equation*}
        F_{k+1}^r = \Ag{ A\nabla_2 \widetilde{Y}_{k+2}^o }_y + \Ag{ A \nabla_1 Y_{k+1}^o}_y.
    \end{equation*}
    By a familiar calculation as in \eqref{est.Fr12.n=2}, 
    \begin{equation*}
        \| \nabla_1^\ell F_{k+1}^r \|_{L^2_\x} \le \frac{2C_0 C_{**}}{C_{**}-1} C^{3(k+1)} \Lambda^{\ell+k+1} (\ell+k+1+q)!.
    \end{equation*}
    Then by Proposition \ref{prop.energy2} applied to \eqref{eq.Ykr.n=2}, we have
    \begin{align*}
    \begin{split}
        \| \nabla_1^\ell Y_{k+1}^r \|_{H^1_\x} &\le \frac{2C_0  \widehat{C}_* C_{**} }{C_{**}-1} C^{3(k+1)} \Lambda^{\ell+k+1} (\ell+k+1+q)!\\
        &\le \frac12 C^{3(k+1)+1} \Lambda^{\ell+k+1} (\ell+k+1+q)!,
        \end{split}
    \end{align*}
    where we require $C \ge \frac{4C_0 \widehat{C}_* C_{**} }{C_{**}-1}$. Note that this requirement (as well as the previous several requirements on $C$) is independent of $k$ and thus $C$ can be fixed globally.

    \textbf{Step 5:} Complete the proof.

    Consider the partial sum of $Y_k$ truncated at $j = k$ (with a modification):
    \begin{equation*}
        N_k(x,y) = \sum_{j=0}^{k-1} \delta_2^j Y_j(x,y) + \delta_2^k Y_k^o.
    \end{equation*}
    Then for $\ell \ge 0$, by the estimate of $Y_k$ in \eqref{est.Yk.o+r} and the estimate of $Y_k^o$ in \eqref{assum.Yk-a},
    \begin{equation*}
    \begin{aligned}
        \| \bfD^\ell_2 N_k \|_{L^2_\xy}& \le 
        \sum_{j=0}^{k-1} \delta_2^j C^{3j+1} \Lambda^{\ell+j} (\ell+j+q)! +  \delta_2^k C^{3k-1} \Lambda^{\ell-1+k} (\ell-1+k+q)! \\
        & \le C \Lambda^{\ell} (\ell+q)!,
    \end{aligned}
    \end{equation*}
    provided that $\delta_2 C^3 \Lambda (\ell+k+q) \le 1/e$.

    Now by \eqref{eq.liftphi}, we see that $N_k$ satisfies
    \begin{equation*}
    \begin{aligned}
        & -\widehat{\nabla}_2 \cdot A\widehat{\nabla}_2 N_k + \tau^2 N_k - \widehat{\nabla}_2 \cdot F \\
        & = E_k := \delta^{k-1}_2 \big( (\cL_{\xy} + \cL_{\yx}) Y_k^o + \cL_{\xx} Y_{k-1}+\tau^2Y_{k-1}) + \delta_2^{k} (\cL_{\xx} Y_k^o+\tau^2Y_k^o).
    \end{aligned}
    \end{equation*}
    By a simple observation, we have
    \begin{equation*}
        E_k = -\delta_2^k \widehat{\nabla}_2 \cdot A \nabla_1 Y_k^o - \delta_2^{k-1} (\nabla_1 \cdot A \nabla_2 Y_k^o + \nabla_1 \cdot A \nabla_1 Y_{k-1})+\tau^2(\delta_2^k Y_k^o+\delta_2^{k-1} Y_{k-1} ).
    \end{equation*}
    Let $P_k = A\nabla_1 Y_k^o$ and $Q_k =\nabla_1 \cdot A \nabla_2 Y_k^o + \nabla_1 \cdot A \nabla_1 Y_{k-1} - \tau^2(\delta_2 Y_k^o + Y_{k-1})$. Then
    \begin{equation*}
    \begin{aligned}
        \| \bfD_2^\ell P_k \|_{L^2_{\xy}} & \le \sum_{m=0}^\ell \binom{\ell}{m} \| \bfD_2^m A \|_{L^\infty_{\xy}} \| \bfD_2^{\ell-m} \nabla_1 Y_k^o \|_{L^2_{\xy}} \\
        & \le \sum_{m=0}^\ell \binom{\ell}{m} C_0 \Lambda_0^m m! C^{3k-1} \Lambda^{\ell-m+k}
        (\ell-m+k+q)! \\
        & \le \frac{C_0 C_{**}}{C_{**}-1} C^{3k-1} \Lambda^{\ell+k} (\ell+k+q)!,
    \end{aligned}
    \end{equation*}
    and
    \begin{equation*}
        \begin{aligned}
            \| \bfD_2^\ell Q_k\|_{L^2_{\xy}} & \le \sum_{m=0}^\ell \binom{\ell}{m} \Big( \| \bfD_2^m \nabla_1 A \|_{L^\infty_{\xy}} \| \bfD_2^{\ell-m} \nabla_2 Y_k^o \|_{L^2_{\xy}} +  \| \bfD_2^m  A \|_{L^\infty_{\xy}} \| \bfD_2^{\ell-m} \nabla_1 \nabla_2 Y_k^o \|_{L^2_{\xy}} \Big) \\
            & \quad + \sum_{m=0}^\ell \binom{\ell}{m} \Big( \| \bfD_2^m \nabla_1 A \|_{L^\infty_{\xy}} \| \bfD_2^{\ell-m} \nabla_1 Y_{k-1} \|_{L^2_{\xy}} +  \| \bfD_2^m  A \|_{L^\infty_{\xy}} \| \bfD_2^{\ell-m} \nabla_1^2  Y_{k-1} \|_{L^2_{\xy}} \Big) \\
            & \quad + \tau^2 \delta_2 \| \bfD_2^\ell Y_k^o \|_{L^2_{\xy}} + \tau^2 \| \bfD_2^\ell Y_{k-1} \|_{L^2_{\xy}} \\
            & \le \sum_{m=0}^\ell \binom{\ell}{m} \Big( C_0 \Lambda_0^{m+1} (m+1)! C^{3k-1} \Lambda^{\ell-m -1+k} (\ell-m-1+k+q)!  \Big) \\
            & \quad + \sum_{m=0}^\ell \binom{\ell}{m} \Big( C_0 \Lambda_0^{m} m! C^{3k-1} \Lambda^{\ell-m+k } (\ell-m+k+q)!  \Big) \\
            & \quad + \sum_{m=0}^\ell \binom{\ell}{m} \Big( C_0 \Lambda_0^{m+1} (m+1)!  C^{3k-2} \Lambda^{\ell-m-1+k} (\ell-m-1+k+q)! \Big) \\
            & \quad + \sum_{m=0}^\ell \binom{\ell}{m} \Big( C_0 \Lambda_0^{m} m!  C^{3k-2} \Lambda^{\ell-m+k} (\ell-m+k+q)! \Big) \\
            & \quad + 2\tau^2 C^{3k-1} \Lambda^{\ell+k-1} (\ell+k-1+q)! \\
            & \le \Big( \frac{4C_0 C_{**}}{C_{**} -1} + 2\tau^2 \Big) C^{3k-1} \Lambda^{\ell+k} (\ell+k+q)!.
        \end{aligned}
    \end{equation*}

    Now let $Z_k$ be the solution of
    \begin{equation*}
        -\widehat{\nabla}_2 \cdot A\widehat{\nabla}_2 Z_k + \tau^2 Z_k = - E_k = \delta_2^k \widehat{\nabla}_2 \cdot P_k + \delta_2^{k-1} Q_k.
    \end{equation*}
    Using the linearity of the equation and Proposition \ref{prop.Ytau.F}, we obtain
    \begin{equation*}
        \tau \|  \bfD_2^\ell Z_k \|_{L^2_\xy}  \le \frac12 \delta_2^k  C^{3k} \Lambda^{\ell+k} (\ell+k+q)! + \frac12 \delta_2^{k-1} \tau^{-1} C^{3k} \Lambda^{\ell+k} (\ell+k+q)!,
    \end{equation*}
    where we need to assume $C \ge 2C_*(\frac{4C_0 C_{**}}{C_{**} -1} + 2\tau^2 )$.
    Thus, if $\delta_2 C^3 \Lambda (\ell+k+q) \le 1/e$, we have
    \begin{equation*}
        \|  \bfD_2^\ell Z_k \|_{L^2_\xy} \le \tau^{-2} e^{-k+1} C^{3} \Lambda^{\ell+1} (\ell+1+q)!.
    \end{equation*}
    Choose $\tau$ such that $\tau^{-2}e^{-k+1} C^3 \Lambda \le 1$. Then we have
    \begin{equation*}
        \|  \bfD_2^\ell Z_k \|_{L^2_\xy} \le \Lambda^{\ell} (\ell+1+q)!.
    \end{equation*}
    Since $Y = N_k + Z_k$ solves the original equation, from the estimates of $N_k$ and $Z_k$, it follows that
    \begin{equation*}
        \|  \bfD^\ell_2 Y \|_{L^2_{\xy}} \le (C+2) \Lambda^\ell (\ell+1+q)!.
    \end{equation*}
    The assumptions we need are $\delta_2 C^3 \Lambda (\ell+k+q) \le 1/e$ and $\tau^2 \ge e^{-k+1} C^3\Lambda.$ This completes the proof.
\end{proof}

\subsection{Uniform estimates for general {$n$} scales}
\label{sec_n-scales}

\begin{theorem}\label{thm.Y.nscale}
    Let $n \ge 2$. Assume that \eqref{as.ellipticity}, \eqref{as.periodicity}, \eqref{cond.A2} and \eqref{cond.F2} hold. Let $Y$ be the solution of
    \begin{equation*}
        -\widehat{\nabla}_n \cdot A \widehat{\nabla}_n Y + \tau^2 Y = \widehat{\nabla}_n\cdot F  \quad\text{ in }  \mathbb{T}^{d\times n}.
    \end{equation*}
    There exist $\widetilde{C}_j$, $\widetilde{\Lambda}_j, C_n, \Lambda_n$ such that if for some $\ell \ge 0$ and any $2\le j\le n$, there exists $k_j \ge 1$ such that
    \begin{equation}\label{cond.sep.n}
        \left\{ \begin{aligned}
            & (\delta_j/\delta_{j-1}) \widetilde{C}_{j}^3  \widetilde{\Lambda}_j (\ell + k_j+q) \le e^{-1}, \\
            & \tau^2 \ge \widetilde{C}_{j}^3 \widetilde{\Lambda}_j \delta_{j-1}^{-1} e^{-k_j+1},
        \end{aligned}
        \right.
    \end{equation}
    then $Y$ satisfies
    \begin{equation}\label{est.DlY.n}
         \| \bfD_n^\ell  Y \|_{L^2_\xy} \le C_n\Lambda_n^{\ell} (\ell + 1 + q)!.
    \end{equation}
    The constants $\widetilde{C}_j, C_n$ depend on $d,\lambda, j, n, C_0$, while $\widetilde{\Lambda}_j, \Lambda_n$ depend additionally on $\Lambda_0$ and $\Lambda_1$.
\end{theorem}
\begin{remark}
Indeed, if \eqref{cond.sep.n} holds for some $ \ell = \ell_0$, then it holds for all $0\le \ell < \ell_0$. 
\end{remark}

\begin{proof}
    We have proved the case of two scales in the previous subsection. Now we assume that the result holds for $n-1$ scales and prove the case of $n$ scales by an inductive argument. Precisely, we assume that $V$ is the solution of
    \begin{equation*}
        -\widehat{\nabla}_{n-1} \cdot \widehat{A} \widehat{\nabla}_{n-1} V + \tau^2 V  = \widehat{\nabla}_{n-1} \cdot F,
    \end{equation*}
    for $\widehat{A}$ and $F$ being 1-periodic of $n-1$ scales. Moreover, $F$ satisfies the condition \eqref{cond.F2} and $\widehat{A}$ (with the same ellipticity constant $\lambda$) satisfies a similar condition as \eqref{cond.A2} with $C_0$ and $\Lambda_0$ replaced by different constants $\widehat{C}_0$ and $\widehat{\Lambda}_0$, which will be given shortly.
    Then there exist $\Lambda_{n-1}, C_{n-1}$ and $\widetilde{\Lambda}_j, \widetilde{C}_j$ with $2\le j\le n-1$ such that if for each $2\le j\le n-1$, there exists $k_j \ge 1$ such that    
    \begin{equation*}
        \left\{ \begin{aligned}
            & (\delta_j/\delta_{j-1}) \widetilde{C}_{j}^3  \widetilde{\Lambda}_j (\ell + k_j+q) \le e^{-1}, \\
            & \tau^2 \ge \widetilde{C}_{j}^3 \widetilde{\Lambda}_j \delta_{j-1}^{-1} e^{-k_j+1},
        \end{aligned}
        \right.
    \end{equation*}
    then 
    \begin{equation*}
        \| \bfD_{n-1}^\ell  V \|_{L^2_\x} \le C_{n-1}
        \Lambda_{n-1}^{\ell} (\ell + 1 + q)!,
    \end{equation*}
    where $\widetilde{C}_j$ and $C_{n-1}$ are constants depending only on $d, \lambda, \widehat{C}_0, j$ and $n$, and $\widetilde{\Lambda}_j$ and $\Lambda_{n-1}$ depend additionally on $\widehat{\Lambda}_0$ and $\Lambda_1$.  This inductive assumption will not be used until Step 5. Actually we only need the energy estimates in Proposition \ref{prop.Ytau.F} and \ref{prop.energy2} from Step 1 to Step 4.

    Recall the formal multiscale ansatz $Y = \sum_{j = 0}^\infty \delta_n^j Y_j$ in Section \ref{sec.corr.ansatz} and we need to solve $Y_k$ recursively by \eqref{sum-s11} - \eqref{sum.s43}. The following proof largely follows the same line as Theorem \ref{thm.2S.Y} in the case of two scales.

    \textbf{Step 1.}  Estimate $\chi$ and $\widehat{A}$. The equations for $\chi$ and $\widehat{A}$ are given by
\begin{equation*}
    \begin{aligned}
        -\nabla_n \cdot A \nabla_n \chi_j & = \nabla_n \cdot (Ae_j),\\
        \widehat{A}(y_1, y_2,\cdots, y_{n-1}) &= \Ag{A + A\nabla_n \chi }_{y_n}. 
    \end{aligned}
\end{equation*}
Applying the similar argument as in Step 1 of Theorem \ref{thm.2S.Y},
we have
\begin{equation*}
    \| \bfD_n^\ell \chi \|_{L^\infty_\x H^1_\y} \le C_* \widehat{\Lambda}_0^\ell  \ell!,
\end{equation*}
and
\begin{equation}\label{cond.hatA}
    \| \bfD_{n-1}^\ell \widehat{A} \|_{L^\infty_\x} \le \widehat{C}_0 \widehat{\Lambda}_0^\ell \ell!.
\end{equation}
Here $\widehat{C}_0$ is a constant generally larger than $C_0$,  $\widehat{\Lambda}_0 = C_* \Lambda_0$ and $C_*$ is a constant depending only on $C_0, d$ and $\lambda$ arising from Proposition \ref{prop.energy2}. 
Moreover, by the classical homogenization theory, the ellipticity constant is preserved from $A$ to its homogenized matrix $\widehat{A}$, namely,
\begin{equation*}
    \lambda |\xi|^2 \le \xi \cdot \widehat{A} \xi, \qquad \text{for all }\xi \in \R^d.
\end{equation*}

\textbf{Step 2.} Estimate of $Y_0$. The involved equations are the following:
\begin{align}
    \cL_{\yy} \widetilde{F} & = \nabla_n \cdot F,\\
    \widehat{F}(y_1,y_2,\cdots, y_{n-1}) & = \Ag{ F + A \nabla_n \widetilde{F} }_{y_n},\\
    -\widehat{\nabla}_{n-1} \cdot \widehat{A} \widehat{\nabla}_{n-1} Y_0 + \tau^2 Y_0 & = \widehat{\nabla}_{n-1} \cdot \widehat{F}. \label{eq.Y0.n}
\end{align}
By the similar argument as in Step 2 of Theorem \ref{thm.2S.Y},
we have for $\widehat{\Lambda}_1 = \max\{\Lambda_1, C_* \Lambda_0 \}$,
\begin{equation}\label{est.tF.n}
    \| \bfD_n^\ell \widetilde{F}\|_{L^2_\x H^1_\y} \le C_* \widehat{\Lambda}_1^{\ell} (\ell+q)!,
\end{equation}
and
\begin{equation*}
    \| \bfD_{n-1}^\ell \widehat{F} \|_{L^2_\x} \le C_{*}^3 \Lambda^{\ell} (\ell + q)!,
\end{equation*}
where we need $\Lambda \ge  \max\{ \widehat{\Lambda}_1, C_* \widehat{\Lambda}_0 \} = \max\{ \Lambda_1, C_*^2 \Lambda_0 \}$.

Now we use the energy estimate in Proposition \ref{prop.Ytau.F} to \eqref{eq.Y0.n} to obtain for all $\ell \ge 0$,
\begin{equation}\label{est.Y0.energy}
    \| \bfD^\ell_{n-1} \widehat{\nabla}_{n-1} Y_0 \|_{L^2_\x} + \tau \| \bfD_{n-1}^\ell Y_0 \|_{L^2_\x} \le \widehat{C}_* C_*^3 \Lambda^\ell (\ell+q)! \le C \Lambda^\ell (\ell+q)!,
\end{equation}
where we set $\Lambda = \max\{ \widehat{\Lambda}_1, \widehat{C}_* \widehat{\Lambda}_0 \} = \max\{ \Lambda_1, \widehat{C}_* C_* \Lambda_0 \}$ and we assume $C \ge \widehat{C}_* C_*^3$.
Now as before, we define
$C_{**} = \widehat{C}_* C_*$. Recall that $\widehat{C}_*$ is the constant from the energy estimates in Proposition \ref{prop.Ytau.F} and Proposition \ref{prop.energy2} with $A$ replaced by $\widehat{A}$. Moreover, we also have \eqref{eq.Lambda.relation}.

\textbf{Step 3.} 
Estimate $Y_1$.
The involved equations are the following:
\begin{align}
    & Y_1^o = \widetilde{F} + \chi\cdot \widehat{\nabla}_{n-1} Y_0,\label{eq.Y1o}\\
    \begin{split}
        & -\nabla_n \cdot A \nabla_n \widetilde{Y}_2^o = \widehat{\nabla}_{n-1}\cdot F + (\widehat{\nabla}_{n-1}\cdot A \nabla_n + \nabla_n \cdot A \widehat{\nabla}_{n-1} ) Y_1^o \\&\qquad\qquad\qquad\qquad+ \widehat{\nabla}_{n-1}\cdot A \widehat{\nabla}_{n-1} Y_0 - \tau^2 Y_0, \label{eq.tY2o}
    \end{split}
    \\
    & -\widehat{\nabla}_{n-1} \cdot \widehat{A} \widehat{\nabla}_{n-1} Y_1^r + \tau^2 Y_1^r  = \widehat{\nabla}_{n-1}\cdot \big( \Ag{ A \nabla_n \widetilde{Y}_2^o}_{y_n} + \Ag{ A \widehat{\nabla}_{n-1} Y_1^o}_{y_n} \big). \label{eq.Y1r}
\end{align}

By \eqref{est.tF.n}, \eqref{est.Y0.energy}, the generalized Leibniz rule and H\"{o}lder's inequality, we obtain from \eqref{eq.Y1o} that
    \begin{equation}\label{est.DlxY1o}
    \begin{aligned}
        \| \bfD_n^\ell Y_1^o \|_{L^2_\x H^1_\y} & \le \| \bfD_n^\ell \widetilde{F} \|_{L^2_\x H^1_\y }  + \sum_{m=0}^\ell \binom{\ell}{m} \| \bfD_n^m \chi \|_{L^\infty_\x H^1_\y} \| \bfD_{n-1}^{\ell-m} \widehat{\nabla}_{n-1} Y_0 \|_{L^2_\x} \\
        & \le C_*  \widehat{\Lambda}_1^{\ell} (\ell+q)! + \sum_{m=0}^\ell \binom{\ell}{m} C_* \widehat{\Lambda}_0^m m! C \Lambda^{\ell-m} (\ell-m+q)! \\
        & \le C_*  \Lambda^{\ell} (\ell+q)! + C_* C \Lambda^{\ell} (\ell+q)! \sum_{m=0}^\ell (\widehat{\Lambda}_0 /\Lambda)^{m} \\
        & \le \Big( \frac{C_*}{C} + \frac{C_*^2}{C_*-1} \Big) C \Lambda^{\ell} (\ell+q)! \\
        & \le C^2 \Lambda^\ell (\ell+q)!,
    \end{aligned}
    \end{equation}
    where we need $C$ large enough such that $C_*/C + C_*^2/(C_*-1) \le C$.

Next, we estimate $\widetilde{Y}_2^o$ from \eqref{eq.tY2o}. Let
    \begin{equation*}
        F_2 = A \widehat{\nabla}_{n-1} Y_1^o, \quad \text{and} \quad G_2 = \widehat{\nabla}_{n-1} \cdot F + \widehat{\nabla}_{n-1} \cdot (A \nabla_n Y_1^o) + \widehat{\nabla}_{n-1} \cdot (A \widehat{\nabla}_{n-1} Y_0) - \tau^2 Y_0.
    \end{equation*}
    Then $\widetilde{Y}_2^o$ satisfies
    \begin{equation*}
        -\nabla_n \cdot A \nabla_n \widetilde{Y}_2^o = \nabla_n \cdot F_2 + G_2.
    \end{equation*}
    It follows from the generalized Leibniz rule and \eqref{est.DlxY1o} that
    \begin{equation*}
        \begin{aligned}
            \| \bfD_n^\ell F_2 \|_{L^2_{\xy}} & \le \sum_{m = 0}^\ell \binom{\ell}{m} \| \bfD_n^m A \|_{L^\infty_{\xy}} \| \bfD_n^{\ell-m} \widehat{\nabla}_{n-1} Y_1^o \|_{L^2_\xy} \\ & \le C_0 \sum_{m=0}^\ell \binom{\ell}{m} \Lambda_0^m m! \delta_{n-1}^{-1} C^2 \Lambda^{\ell-m+1} (\ell-m+1+q)! \\
            & \le \frac{C_0 C_{**}}{C_{**} - 1} \delta_{n-1}^{-1} C^2 \Lambda^{\ell+1} (\ell+1+q)!,
        \end{aligned}
    \end{equation*}
    where we have used $\Lambda/\Lambda_0 \ge C_{**}$ and the simple fact (also see \eqref{eq.hatD-D})
    \begin{equation}\label{est.hatd2D}
        |\bfD_n^\ell \widehat{\nabla}_{n-1} f| \le \delta_{n-1}^{-1} |\bfD_n^{\ell+1} f|.
    \end{equation}

Similarly,
    \begin{equation*}
        \begin{aligned}
            \| \bfD_n^\ell G_2 \|_{L^2_{\xy}}  & \le  \| \bfD_n^\ell \widehat{\nabla}_{n-1} F \|_{L^2_{\xy}} + \tau^2 \| \bfD_n^\ell Y_0 \|_{L^2_{\x}}+ \sum_{m = 0}^\ell \binom{\ell}{m} \| \bfD_n^m A \|_{L^\infty_\xy} \| \bfD_n^{\ell-m} \widehat{\nabla}_{n-1} Y_1^o \|_{L^2_\x H^1_\y} \\
            & \qquad + \sum_{m = 0}^\ell \binom{\ell}{m} \| \bfD_n^m \widehat{\nabla}_{n-1} A \|_{L^\infty_\xy} \| \bfD_n^{\ell-m} Y_1^o \|_{L^2_\x H^1_\y} \\
            & \qquad + \sum_{m = 0}^\ell \binom{\ell}{m} \| \bfD_n^m \widehat{\nabla}_{n-1} A \|_{L^\infty_\xy} \| \bfD_n^{\ell-m} \widehat{\nabla}_{n-1} Y_0 \|_{L^2_\x} \\
            & \qquad + \sum_{m = 0}^\ell \binom{\ell}{m} \| \bfD_n^m A \|_{L^\infty_\xy} \| \bfD_n^{\ell-m} \widehat{\nabla}_{n-1}^2 Y_0 \|_{L^2_\x} \\
            & \le \Big( 1+\tau + \frac{4C_0 C_*^2}{C_*^2-1} \Big) \delta_{n-1}^{-1} C^2 \Lambda^{\ell+1} (\ell+1+q)!.
        \end{aligned}
    \end{equation*}
Then by Proposition \ref{prop.energy2} and the linearity of the equation, we have
    \begin{equation}\label{est.tY2o+}
        \| \bfD_n^\ell \widetilde{Y}_2^o \|_{L^2_\x H^1_\y} \le C_* \Big( 1+\tau + \frac{4C_0 C_{**} }{C_{**}-1} \Big) \delta_{n-1}^{-1} C^2 \Lambda^{\ell+1} (\ell+1+q)! \le \delta_{n-1}^{-1} C^3 \Lambda^{\ell+1} (\ell+1+q)!.
    \end{equation}
    Here we choose $C \ge C_*( 1+\tau + \frac{4C_0 C_{**}}{C_{**}-1} )$.

    Next, we estimate $Y_1^r$. Let
    \begin{equation*}
        F_1^r = \Ag{ A \nabla_n \widetilde{Y}_2^o }_{y_n} + \Ag{ A \widehat{\nabla}_{n-1} Y_1^o}_{y_n}.
    \end{equation*}
    Then \eqref{eq.Y1r} is reduced to
    \begin{equation}\label{eq.Y1r+}
        -\widehat{\nabla}_{n-1} \cdot \widehat{A} \widehat{\nabla}_{n-1} Y_1^r + \tau^2 Y_1^r  = \widehat{\nabla}_{n-1} \cdot F_1^r.
    \end{equation}
    By the generalized Leibniz rule, \eqref{est.tY2o+} and \eqref{est.DlxY1o}, we have
    \begin{equation*}
        \begin{aligned}
            \| \bfD_{n-1}^\ell F_1^r  \|_{L^2_\x} & \le \sum_{m=0}^\ell \binom{\ell}{m} \| \bfD_{n-1}^m A \|_{L^\infty_\xy} \| \bfD_{n-1}^{\ell-m} \widetilde{Y}_2^o \|_{L^2_\x H^1_\y} \\
            & \qquad + \sum_{m=0}^\ell \binom{\ell}{m} \| \bfD_{n-1}^m A \|_{L^\infty_\xy} \| \bfD_{n-1}^{\ell-m} \widehat{\nabla}_{n-1} Y_1^o \|_{L^2_{\xy}} \\
            & \le \frac{2C_0 C_{**}}{C_{**} -1} \delta_{n-1}^{-1} C^3  \Lambda^{\ell+1 } (\ell+1+q)!.
        \end{aligned}
    \end{equation*}
    By Proposition \ref{prop.Ytau.F} applied to \eqref{eq.Y1r+}, we have
    \begin{equation*}
    \begin{aligned}
        \|  \bfD_{n-1}^\ell \widehat{\nabla}_{n-1} Y_1^r \|_{L^2_\x} + \tau \| \bfD_{n-1}^\ell Y^r_1 \|_{L^2_\x} & \le \widehat{C}_* \frac{2C_0 C_{**}}{C_{**} -1} \delta_{n-1}^{-1} C^3 \Lambda^{\ell+1} (\ell+1+q)! \\
        & \le \frac12 \delta_{n-1}^{-1} C^4 \Lambda^{\ell+1} (\ell+1+q)!,
    \end{aligned}
    \end{equation*}
    where we need $C \ge 2\widehat{C}_* \frac{2C_0 C_{**}}{C_{**} -1}$.

    \textbf{Step 4:} Estimate of $Y_k$ for $k\ge 2$. From the recursive equations, we see that the estimates of $(Y_k^o, \widetilde{Y}_{k+1}^o, Y_k^r)$ yields the estimates of $(Y_{k+1}^o, \widetilde{Y}_{k+2}^o, Y_{k+1}^r)$.
    We have proved our base case for $k = 1$ as follows:
    \begin{subequations}\label{est.Y1base}
        \begin{empheq}[left=\empheqlbrace]{align}
            & \| \bfD_n^\ell Y_1^o \|_{L^2_\x H^1_\y} \le  C^{2} \Lambda^{\ell} (\ell+q)!, \\
            & \| \bfD_n^\ell \widetilde{Y}_2^o \|_{L^2_\x H^1_\y} \le \delta_{n-1}^{-1} C^{3} \Lambda^{\ell+1} (\ell+1+q)!, \\
             & \| \bfD_{n-1}^{\ell} \widehat{\nabla}_{n-1} Y_1^r \|_{L^2_\x} + \tau \| \bfD_{n-1}^\ell Y^r_1 \|_{L^2_\x} \le \frac12 \delta_{n-1}^{-1} C^{4} \Lambda^{\ell+1} (\ell+1+q)!.
        \end{empheq}
    \end{subequations}

    Suppose now that for some $k\ge 1 $ and all $\ell \ge 0$ we have
    \begin{subequations}\label{assum.Yk}
        \begin{empheq}[left=\empheqlbrace]{align}
            & \| \bfD_n^\ell Y_{k}^o \|_{L^2_\x H^1_\y} \le \delta_{n-1}^{-k+1} C^{3k-1} \Lambda^{\ell+k-1} (\ell+k-1+q)!, \\
            & \| \bfD_n^\ell \widetilde{Y}_{k+1}^o\|_{L^2_\x H^1_\y} \le \delta_{n-1}^{-k} C^{ 3k} \Lambda^{\ell+k} (\ell+k+q)!,\\
            & \| \bfD_{n-1}^{\ell} \widehat{\nabla}_{n-1} Y_k^r \|_{L^2_\x} + \tau \| \bfD_{n-1}^\ell Y^r_k \|_{L^2_\x} \le \frac12 \delta_{n-1}^{-k} C^{3k+1} \Lambda^{\ell+k} (\ell+k+q)!.
        \end{empheq}
    \end{subequations}
    We would like to show that for all $\ell \ge 0$,
    \begin{equation}\label{est.Yk.induction}
        \left\{
        \begin{aligned}
        & \| \bfD_n^\ell Y_{k+1}^o\|_{L^2_\x H^1_\y} \le \delta_{n-1}^{-k} C^{3(k+1) -1} \Lambda^{\ell+k } (\ell+k+q)!, \\
        & \| \bfD_n^\ell \widetilde{Y}_{k+2}^o\|_{L^2_\x H^1_\y} \le \delta_{n-1}^{-k-1} C^{3(k+1)} \Lambda^{\ell+k+1} (\ell+k+1+q)!,\\
            & \| \bfD_{n-1}^{\ell} \widehat{\nabla}_{n-1} Y_{k+1}^r \|_{L^2_\x} + \tau \| \bfD_{n-1}^\ell Y^r_{k+1} \|_{L^2_\x} \le \frac{1}{2} \delta_{n-1}^{-k-1} C^{3(k+1)+1} \Lambda^{\ell+k+1 } (\ell+k+1+q)!. \\
        \end{aligned}
        \right.
    \end{equation}

    The proof of \eqref{est.Yk.induction} is very similar to the estimates of \eqref{est.Y1base}, based on the following recursive equations
    \begin{align}
            &Y_{k+1}^o  = \widetilde{Y}_{k+1}^o + \chi\cdot \widehat{\nabla}_{n-1} Y_{k}^r, \label{eq.Yok+1}\\
            \begin{split}
                & -\nabla_n \cdot A \nabla_n \widetilde{Y}_{k+2}^o  = \widehat{\nabla}_{n-1} \cdot (A \nabla_n Y_{k+1}^o) + \nabla_n\cdot (A \widehat{\nabla}_{n-1} Y_{k+1}^o)\\&\qquad\qquad\qquad\qquad + \widehat{\nabla}_{n-1}\cdot ( A \widehat{\nabla}_{n-1} Y_k) - \tau^2 Y_k, \label{eq.tYok+2}
            \end{split}
            \\
            & -\widehat{\nabla}_{n-1}\cdot ( \widehat{A} \widehat{\nabla}_{n-1} Y_{k+1}^r) + \tau^2 Y_{k+1}^r = \widehat{\nabla}_{n-1}\cdot \Ag{ A \nabla_n \widetilde{Y}_{k+2}^o}_y + \widehat{\nabla}_{n-1}\cdot \Ag{ A \widehat{\nabla}_{n-1} Y_{k+1}^o}_y.\label{eq.Yrk+1}
    \end{align}
   We first use the algebraic equation \eqref{eq.Yok+1} to estimate $Y_{k+1}^o$. In fact, by the inductive assumption \eqref{assum.Yk} and the generalized Leibniz rule, we have
    \begin{equation*}
        \begin{aligned}
            \| \bfD_n^\ell Y_{k+1}^o\|_{L^2_\x H^1_\y} & \le \| \bfD_n^\ell \widetilde{Y}_{k+1}^o \|_{L^2_\x H^1_\y} + \sum_{m=0}^\ell \binom{\ell}{m} \| \bfD_n^m \chi \|_{L^\infty_\x H^1_\y} \| \bfD_n^{\ell-m} \widehat{\nabla}_{n-1} Y_{k}^r \|_{L^2_\x} \\
            & \le \delta_{n-1}^{-k} C^{3k} \Lambda^{\ell+k}(\ell+k+q)! \\
            & \quad + \sum_{m=0}^\ell \binom{\ell}{m} C_* \widehat{\Lambda}_0^m m! \frac12 \delta_{n-1}^{-k} C^{3k+1} \Lambda^{\ell-m+k} (\ell-m+k+q)! \\
            & \le \delta^{-k}_{n-1}C^{3k+2} \Lambda^{\ell+k} (\ell+k+q)! \\
            & = \delta_{n-1}^{-k} C^{3(k+1) - 1} \Lambda^{\ell+k} (\ell+k+q)!.
        \end{aligned}
    \end{equation*}
    This proves the first inequality of \eqref{est.Yk.induction}.

    The estimate of $\widetilde{Y}_{k+2}^o$ is similar to $\widetilde{Y}_2^o$. Let
    \begin{equation*}
        F_{k+2} = A \widehat{\nabla}_{n-1} Y_{k+1}^o \qquad \text{and} \qquad G_{k+2} = \widehat{\nabla}_{n-1}\cdot (A \nabla_n Y_{k+1}^o) + \widehat{\nabla}_{n-1}\cdot (A\widehat{\nabla}_{n-1} Y_k) - \tau^2 Y_k.
    \end{equation*}
    Then $\widetilde{Y}_{k+2}^o$ satisfies
    \begin{equation}\label{eq.tYok+2+}
        -\nabla_n \cdot A\nabla_n \widetilde{Y}_{k+2}^o = \nabla_n \cdot F_{k+2} + G_{k+2}.
    \end{equation}
    Note that by the triangle inequality to $Y_k = Y_k^o + Y_k^r$ and \eqref{est.hatd2D}, we have
    \begin{equation}\label{est.DnhatD.Yk}
        \| \bfD_n^\ell \widehat{\nabla}_{n-1} Y_k \|_{L^2_{\xy}} \le \| \bfD_n^\ell \widehat{\nabla}_{n-1} Y_k^o \|_{L^2_{\xy}} + \| \bfD_n^\ell \widehat{\nabla}_{n-1} Y_k^r \|_{L^2_{\xy}} \le \delta_{n-1}^{-k} C^{3k+1} \Lambda^{\ell+k} (\ell+k+q)!.
    \end{equation}
    Moreover, for $0<\tau\le 1$, 
    \begin{equation}\label{est.tauDn.Yk}
        \tau \| \bfD_n^\ell Y_k \|_{L^2_{\xy}} \le \tau \| \bfD_n^\ell  Y_k^o \|_{L^2_{\xy}} + \tau \| \bfD_n^\ell Y_k^r \|_{L^2_{\xy}} \le \delta_{n-1}^{-k} C^{3k+1} \Lambda^{\ell+k} (\ell+k+q)!.
    \end{equation}
    Thus, by a familiar calculation, we have
    \begin{equation*}
        \| \bfD_n^\ell F_{k+2} \|_{L^2} \le \frac{C_0 C_{**}}{C_{**} -1} \delta_{n-1}^{-k-1} C^{3(k+1) - 1} \Lambda^{\ell+k} (\ell+k+q)!,
    \end{equation*}
    and
    \begin{equation*}
        \| \bfD_n^\ell G_{k+2} \|_{L^2} \le \big( \frac{4C_0 C_{**}}{C_{**} -1} + \tau \big) \delta_{n-1}^{-k-1} C^{3(k+1) - 1} \Lambda^{\ell+k+1} (\ell+k+1+q)!.
    \end{equation*}
    Then applying Proposition \ref{prop.energy2} to the equation \eqref{eq.tYok+2+},
    we get 
    \begin{equation*}
    \begin{aligned}
        \| \bfD^\ell \widetilde{Y}_{k+2}^o\|_{L^2_\x H^1_\y} & \le  C_* \big( \frac{4C_0 C_{**} }{C_{**} -1} + \tau \big) \delta_{n-1}^{-k-1} C^{3(k+1)-1} \Lambda^{\ell+k+1} (\ell+k+1+q)! \\
        & \le \delta_{n-1}^{-k-1} C^{3(k+1)} \Lambda^{\ell+k+1} (\ell+k+1+q)!,
    \end{aligned} 
    \end{equation*}
    where we have assumed $C \ge C_*(\frac{4C_0 C_{**}}{C_{**} -1} + \tau) $ previously. This proves the second inequality in \eqref{est.Yk.induction}.

    Finally, we estimate $Y_{k+1}^r$. Let
    \begin{equation*}
        F_{k+1} = \Ag{ A\nabla_n \widetilde{Y}_{k+2}^o }_{y_n} + \Ag{ A \widehat{\nabla}_{n-1} Y_{k+1}^o}_{y_n}.
    \end{equation*}
    By a familiar calculation, we have
    \begin{equation*}
        \| \bfD_{n-1}^\ell F_{k+1}^r \|_{L^2_\x} \le \frac{2C_0 C_{**} }{C_{**}-1} \delta_{n-1}^{-k-1} C^{3(k+1)} \Lambda^{\ell+k+1} (\ell+k+1+q)!.
    \end{equation*}
    Then by Proposition \ref{prop.energy2} applied to
    \begin{equation*}
        -\widehat{\nabla}_{n-1}\cdot ( \widehat{A} \widehat{\nabla}_{n-1} Y_{k+1}^r) + \tau^2 Y_{k+1}^r = \widehat{\nabla}_{n-1}\cdot F_{k+1}^r,
    \end{equation*}
    we have
    \begin{equation*}
    \begin{aligned}
       & \| \bfD_{n-1}^\ell \widehat{\nabla}_{n-1} Y_{k+1}^r \|_{L^2_\x} + \tau \| \bfD_{n-1}^\ell Y_{k+1}^r \|_{L^2_\x} \\
        & \le \frac{2C_0 \widehat{C}_* C_{**}}{C_{**}-1} \delta_{n-1}^{-k-1} C^{3(k+1)} \Lambda^{\ell+k+1} (\ell+k+1+q)! \\
        & \le \frac12 \delta_{n-1}^{-k-1} C^{3(k+1)+1} \Lambda^{\ell+k+1} (\ell+k+1+q)!,
    \end{aligned}  
    \end{equation*}
    where we require $C \ge \frac{4C_0 \widehat{C}_* C_{**}}{C_{**}-1}$. Note that this requirement is independent of $k$ and $\ell$, and thus $C$ can be fixed globally.

    \textbf{Step 5:} Complete the proof.
    First of all, by \eqref{est.DnhatD.Yk} and \eqref{est.tauDn.Yk}, as well as \eqref{est.Y0.energy}, we already have
    \begin{equation}\label{est.Yk-1}
        \| \bfD_n^\ell \widehat{\nabla}_{n-1} Y_k \|_{L^2_{\xy}} \le  \delta_{n-1}^{-k} C^{3k+1} \Lambda^{\ell+k} (\ell+k+q)!,
    \end{equation}
    and
    \begin{equation*}
        \tau \| \bfD_n^\ell Y_k \|_{L^2_{\xy}}  \le \delta_{n-1}^{-k} C^{3k+1} \Lambda^{\ell+k} (\ell+k+q)!.
    \end{equation*}
    
It is important to point out that the above two estimates hold for all $\ell \ge 0$ and $k\ge 0$. Unfortunately the second estimate also relies on $\tau$ since the estimate of $Y_k^r$ does. While on the other hand, if $n \ge 3$, we cannot get the estimate of $\| \bfD_n^\ell Y_k \|_{L^2_{\xy}}$ from \eqref{est.Yk-1} by the Poincar\'{e} inequality due to the degenerate directional gradient $\widehat{\nabla}_{n-1}$. In order to get the estimate of $\| \bfD_n^\ell Y_k \|_{L^2_{\xy}}$ independent of $\tau$, we have to use the inductive assumption. In fact, we can recall
\begin{equation}\label{eq.Ykr}
    -\widehat{\nabla}_{n-1}\cdot ( \widehat{A} \widehat{\nabla}_{n-1} Y_{k}^r) + \tau^2 Y_{k}^r = \widehat{\nabla}_{n-1}\cdot F_{k}^r,
\end{equation}
and the estimate of $F_k^r$
\begin{equation}\label{est.Dl.Fkr}
    \| \bfD_{n-1}^\ell F_{k}^r \|_{L^2_\x} \le \frac{2C_0 C_{**} }{C_{**}-1} \delta_{n-1}^{-k} C^{3k} \Lambda^{\ell+k} (\ell+k+q)!.
\end{equation}
Also recall that $\widehat{A}$ satisfies the ellipticity condition with the same constant $\lambda$ and the quantitative analyticity condition \eqref{cond.hatA}. So the constant $C_0, \Lambda_0$ and $\Lambda_1$ in the inductive assumption are replaced by $\widehat{C}_0, \widehat{\Lambda}_0$ and $\Lambda$, respectively. Therefore, by applying the inductive assumption on \eqref{eq.Ykr} in this situation, we have the following: there exist $\Lambda_{n}, C_{n-1}$ and $\widetilde{\Lambda}_j, \widetilde{C}_j$ with $2\le j\le n-1$ such that if for each $2\le j\le n-1$, there exists $k_j \ge 1$ such that   
    \begin{equation}\label{cond.delta+tau+l}
        \left\{ \begin{aligned}
            & (\delta_j/\delta_{j-1}) \widetilde{C}_{j}^3  \widetilde{\Lambda}_{j} (\ell + k_j+q) \le e^{-1}, \\
            & \tau^2 \ge \widetilde{C}_{j}^3 \widetilde{\Lambda}_{j} \delta_{j-1}^{-1} e^{-k_j+1},
        \end{aligned}
        \right.
    \end{equation}
    then 
    \begin{equation}\label{est.DlYkr+unif}
        \| \bfD_{n-1}^\ell  Y_k^r \|_{L^2_\x} \le  C_{n-1} \frac{2C_0 C_{**} }{C_{**}-1} \delta_{n-1}^{-k} C^{3k}
        \Lambda_{n}^{\ell} \Lambda^k (\ell + 1 +k + q)!.
    \end{equation}
    In this estimate, the factor $\frac{2C_0 C_{**} }{C_{**}-1} \delta_{n-1}^{-k} C^{3k} \Lambda^k$ is preserved from \eqref{est.Dl.Fkr} by the linearity of the equation. 
    The constants $\widetilde{C}_j$ in \eqref{cond.delta+tau+l} depend only on $d, \lambda, \widehat{C}_0, j$ and $n$.
    The constants $\widetilde{\Lambda}_j$ depends additionally on $\widehat{\Lambda}_0, \Lambda$. The constant $\widetilde{C}_{n-1}$ depends only on $d, \widehat{C}_0 $ and $\lambda$ and $n$, and $\Lambda_n$ depends additionally on $\widehat{\Lambda}_0, \Lambda$. Notice that $\widehat{C}_0$ depends only on $d,\lambda$ and $C_0$, while $(\widehat{\Lambda}_0, \Lambda)$ depends additionally on $(\Lambda_0, \Lambda_1)$. Thus eventually, $C_{n-1}$ and $\widetilde{C}_j$ depend on $d, \lambda, C_0, j$ and $n$, and $\Lambda_n$ and $\widetilde{\Lambda}_j$ depend additionally on $(\Lambda_0, \Lambda_1)$.

As a consequence of \eqref{assum.Yk} and \eqref{est.DlYkr+unif}, by the triangle inequality, for $\ell$ satisfying \eqref{cond.delta+tau+l}, we have
\begin{equation}\label{est.Yk.L2}
    \| \bfD_n^\ell Y_k \|_{L^2_{\xy}} \le \| \bfD_n^\ell Y_k \|_{L^2_{\x} H^1_\y } \le C_{n-1} \delta_{n-1}^{-k} C^{3k+1} \Lambda_{n}^{\ell} \Lambda^k (\ell +1+k+q)!,
\end{equation}
provided $C \ge \frac{4C_0 C_{**} }{C_{**}-1}$.

    Now we consider a partial sum of \eqref{eq.Y2scaleExp} truncated at $j = k$ (with a modification):
    \begin{equation*}
        N_k(x,y) = \sum_{j=0}^{k-1} \delta_n^j Y_j(x,y) + \delta_n^k Y_k^o.
    \end{equation*}
    Then for $\ell \ge 0$ satisfying \eqref{cond.delta+tau+l}, by \eqref{assum.Yk} and \eqref{est.Yk.L2},
    \begin{equation*}
    \begin{aligned}
        \| \bfD_n^\ell N_k \|_{L^2_{\xy}} & \le 
        \sum_{j=0}^{k-1} \delta_n^j \delta_{n-1}^{-j} C_{n-1} C^{3j+1} \Lambda_{n}^{\ell} \Lambda^j (\ell+1+j+q)! \\
        & \qquad +  \delta_n^k \delta_{n}^{-k+1} C^{3k-1} \Lambda^{\ell+k-1} (\ell+k-1+q)! \\
        & \le (CC_{n-1}+1) \Lambda_{n-1}^{\ell} (\ell+1+q)!,
    \end{aligned}
    \end{equation*}
    provided that $(\delta_n/\delta_{n-1}) C^3 \Lambda (\ell+k+q) \le 1/e$.

    Now in view of the two-scale expansion \eqref{eq.liftphi}, we write down the equation for $N_k$,
    \begin{equation*}
    \begin{aligned}
        & -\widehat{\nabla}_n \cdot A\widehat{\nabla}_n N_k + \tau^2 N_k - \widehat{\nabla}_n \cdot F \\
        & = E_k := \delta^{k-1}_n \big( (\cL_{\xy} + \cL_{\yx}) Y_k^o + \cL_{\xx} Y_{k-1}+\tau^2Y_{k-1}) + \delta_n^{k} (\cL_{\xx} Y_k^o+\tau^2Y_k^o).
    \end{aligned}
    \end{equation*}

     By a simple observation, we have
    \begin{equation*}
    \begin{aligned}
        E_k 
        & = -\delta_n^k \widehat{\nabla}_n \cdot P_k - \delta_n^{k-1} Q_k,
    \end{aligned}
    \end{equation*}
    where $P_k = A \widehat{\nabla}_{n-1} Y_k^o$ and $Q_k =\widehat{\nabla}_{n-1} \cdot A \nabla_n Y_k^o + \widehat{\nabla}_{n-1} \cdot A \widehat{\nabla}_{n-1} Y_{k-1} - \tau^2(\delta_n Y_k^o + Y_{k-1})$. Then
    \begin{equation*}
    \begin{aligned}
        \| \bfD_n^\ell P_k \|_{L^2_{\xy}} & \le \sum_{m=0}^\ell \binom{\ell}{m} \| \bfD_n^m A \|_{L^\infty_{\xy}} \| \bfD_n^{\ell-m} \widehat{\nabla}_{n-1} Y_k^o \|_{L^2_{\xy}} \\
        & \le \sum_{m=0}^\ell \binom{\ell}{m} C_0 \Lambda_0^m m! \delta_{n-1}^{-k} C^{3k-1} \Lambda^{\ell-m+k} (\ell-m+k+q)! \\
        & \le \frac{C_0 C_{**}}{C_{**}-1} \delta_{n-1}^{-k} C^{3k-1} \Lambda^{\ell+k} (\ell+k+q)!,
    \end{aligned}
    \end{equation*}
    and
    \begin{equation*}
        \begin{aligned}
            &\| \bfD_n^\ell Q_k\|_{L^2_{\xy}} \\
            & \le \sum_{m=0}^\ell \binom{\ell}{m} \Big( \| \bfD_n^m \widehat{\nabla}_{n-1} A \|_{L^\infty_{\xy}} \| \bfD_n^{\ell-m} \nabla_n Y_k^o \|_{L^2_{\xy}} +  \| \bfD_n^m  A \|_{L^\infty_{\xy}} \| \bfD_n^{\ell-m} \widehat{\nabla}_{n-1} \nabla_n Y_k^o \|_{L^2_{\xy}} \Big) \\
            & \quad + \sum_{m=0}^\ell \binom{\ell}{m} \Big( \| \bfD_n^m \widehat{\nabla}_{n-1} A \|_{L^\infty_{\xy}} \| \bfD_n^{\ell-m} \widehat{\nabla}_{n-1} Y_{k-1} \|_{L^2_{\xy}} +  \| \bfD_n^m  A \|_{L^\infty_{\xy}} \| \bfD_n^{\ell-m} \widehat{\nabla}_{n-1}^2  Y_{k-1} \|_{L^2_{\xy}} \Big) \\
            & \quad + \tau^2 \delta_n \| \bfD_n^\ell Y_k^o \|_{L^2_\xy} + \tau^2 \|\bfD_n^\ell Y_{k-1} \|_{L^2_\xy} \\
            & \le \sum_{m=0}^\ell \binom{\ell}{m} \Big( \delta_{n-1}^{-1} C_0 \Lambda_0^{m+1} (m+1)! \delta_{n-1}^{-k+1} C^{3k-1} \Lambda^{\ell-m -1+k} (\ell-m-1+k+q)!  \Big) \\
            & \quad + \sum_{m=0}^\ell \binom{\ell}{m} \Big( C_0 \Lambda_0^{m} m! \delta_{n-1}^{-k} C^{3k-1} \Lambda^{\ell-m+k } (\ell-m+k+q)!  \Big) \\
            & \quad + \sum_{m=0}^\ell \binom{\ell}{m} \Big( \delta_{n-1}^{-1} C_0 \Lambda_0^{m+1} (m+1)! \delta_{n-1}^{-k+1}  C^{3k-2} \Lambda^{\ell-m-1+k} (\ell-m-1+k+q)! \Big) \\
            & \quad + \sum_{m=0}^\ell \binom{\ell}{m} \Big( C_0 \Lambda_0^{m} m! \delta_{n-1}^{-k} C^{3k-2} \Lambda^{\ell-m+k} (\ell-m+k+q)! \Big) \\
            & \quad + \tau^2 \delta_n\delta_{n-1}^{-k+1} C^{3k-1} \Lambda^{\ell+k-1} (\ell+k-1+q)!+\tau\delta_{n-1}^{-k+1} C^{3k-2} \Lambda^{\ell+k-1} (\ell+k-1+q)! \\
            & \le \Big( \frac{4C_0 C_{**}}{C_{**} -1} + 2\tau \Big) \delta_{n-1}^{-k} C^{3k-1} \Lambda^{\ell+k} (\ell+k+q)!.
        \end{aligned}
    \end{equation*}

    Now let $Z_k$ be the solution of
    \begin{equation}\label{eq.Zk}
        -\widehat{\nabla}_n \cdot A\widehat{\nabla}_n Z_k + \tau^2 Z_k = - E_k = \delta_n^k \widehat{\nabla}_n \cdot P_k + \delta_n^{k-1} Q_k.
    \end{equation}
    Using the linearity of the equation and Proposition \ref{prop.Ytau.F}, we obtain
    \begin{equation*}
        \|\bfD_n^\ell\widehat{\nabla}_nZ_k\|_{L^2_\xy}+\tau \|  \bfD_n^\ell Z_k \|_{L^2_\xy}  \le \frac12 \delta_n^k \delta_{n-1}^{-k}  C^{3k} \Lambda^{\ell+k} (\ell+k+q)! + \frac12 \delta_n^{k-1} \delta_{n-1}^{-k} \tau^{-1} C^{3k} \Lambda^{\ell+k} (\ell+k+q)!,
    \end{equation*}
    where we need to assume $C \ge 2C_*(\frac{4C_0 C_{**}}{C_{**} -1} + 2\tau )$.
    Thus, if $(\delta_n/\delta_{n-1}) C^3 \Lambda (\ell+k+q) \le 1/e$, we have
    \begin{equation*}
        \|  \bfD^\ell_n Z_k \|_{L^2_\xy} \le \tau^{-2} \delta_{n-1}^{-1}e^{-k+1} C^{3} \Lambda^{\ell+1} (\ell+1+q)!.
    \end{equation*}
    Choose $\tau$ such that $\tau^{-2}\delta_{n-1}^{-1}e^{-k+1} C^3 \Lambda \le 1$. Then we have
    \begin{equation*}
        \|  \bfD^\ell_n Z_k \|_{L^2_\xy} \le \Lambda^{\ell} (\ell+1+q)!.
    \end{equation*}
    Since $Y = N_k + Z_k$ solves the original equation, combining the estimates of $N_k$ and $Z_k$, we get
    \begin{equation*}
        \|  \bfD_n^\ell Y \|_{L^2_{\xy}} \le (CC_{n-1}+2) \Lambda_n^\ell (\ell+1+q)! = C_n \Lambda_n^\ell (\ell+1+q)!,
    \end{equation*}
    where we have relabeled $CC_{n-1}+2$ as $C_n$. This is the expected estimate \eqref{est.DlY.n}. Finally, we note that in addition to the assumption required in \eqref{cond.delta+tau+l} for $2\le j\le n-1$, we also require that there exists $k = k_n$ such that
    \begin{equation*}
        \left\{ \begin{aligned}
            & (\delta_n/\delta_{n-1}) \widetilde{C}_{n}^3  \widetilde{\Lambda}_{n} (\ell + k_n+q) \le e^{-1}, \\
            & \tau^2 \ge \widetilde{C}_n^3 \widetilde{\Lambda}_n \delta_{n-1}^{-1} e^{-k_n+1},
        \end{aligned}
        \right.
    \end{equation*}
    holds for some $\ell \ge 0$,
    where $(\widetilde{C}_n, \widetilde{\Lambda}_n) = (C, \Lambda)$. The proof is complete.
\end{proof}

We now apply Theorem \ref{thm.Y.nscale} to the equation of the multiscale correctors:
\begin{equation}\label{eq.AppCorrector}
    -\widehat{\nabla}_n\cdot A \widehat{\nabla}_n \mathcal{X} + \tau^2 \mathcal{X} = \widehat{\nabla}_n\cdot (A v),
\end{equation}
where $v$ is a unit vector in $\R^d$  (particularly $v = e_j$). Recall that in our setting, $\delta_i = \e_i/\e_1$, $i=1, \cdots, n.$
\begin{theorem}\label{thm.AppCorrector}
    Assume that $A$ satisfies \eqref{as.ellipticity}, \eqref{as.periodicity}, and \eqref{cond.A2}. Then there exist constants $c_j>0, j=2,3,\cdots, n$ such that the following statement holds:
    If $(\e_1,\e_2,\cdots, \e_n)$ satisfies the scale-separation condition 
    \begin{equation}\label{cond.ej.sepcond}
    \e_j \le \frac{c_j \e_{j-1}}{1+\log(\e_1/\e_{j-1})},
    \end{equation}
    then with some $\tau^2 \simeq \max_j \{ e^{-c\e_{j-1}/\e_j} \}$, \eqref{eq.AppCorrector} has a unique bounded solution satisfying
    \begin{equation}\label{inftychi}
        \| \mathcal{X} \|_{L^\infty} \le C,
    \end{equation}
    and for any $\ell \ge 0$,
    \begin{equation}\label{est.X.energy}
        \| \bfD_n^\ell \widehat{\nabla}_n \mathcal{X} \|_{L^2} + \tau \| \bfD_n^\ell \mathcal{X} \|_{L^2} \le C_* \Lambda^\ell \ell!,
    \end{equation}
    where $\Lambda = C_* \Lambda_0$ and $C_*$ is given in Proposition \ref{prop.Ytau.F}.
    The constant $c_j, c$ and $C$ depend only on the characters of $A$.
\end{theorem}
\begin{proof}
    The estimate \eqref{est.X.energy} is exactly the energy estimate in Proposition \ref{prop.Ytau.F} with $F = Av$, and it is independent of $\delta = (\delta_1, \cdots, \delta_n)$  and the scale-separation condition.

    Due to Theorem \ref{thm.Y.nscale}, the above equation has a bounded solution if the scale-separation condition holds: there exists $k_j>0$ for $2\le j\le n$ such that
    \begin{equation}\label{cond.separation.onchi}
        \left\{ \begin{aligned}
            & (\delta_j/\delta_{j-1}) M_j (\ell + k_j) \le e^{-1}, \\
            & \tau^2 \ge M_j \delta_{j-1}^{-1} e^{-k_j+1},
        \end{aligned}
        \right. 
        \text{ which is equivalent to }
        \left\{ \begin{aligned}
            & (\e_j/\e_{j-1}) M_j (\ell + k_j) \le e^{-1}, \\
            & \tau^2 \ge M_j (\e_1/\e_{j-1}) e^{-k_j+1},
        \end{aligned}
        \right. 
    \end{equation}
for some constant $M_j$ depending on the characters of $A$. To make sure that \eqref{eq.AppCorrector} is indeed a sufficiently accurate equation for multiscale correctors, we need $\tau$ be small enough. Particularly, if we assume
\begin{equation}\label{cond.deltaj.kj}
    \delta_{j-1}^{-1} e^{-k_j/2} \le 1,
\end{equation}
then $\tau^2$ is allowed to be as small as $e^{-k_j/2}$. The condition \eqref{cond.deltaj.kj} is equivalent to $k_j \ge 2\log(\e_1/\e_{j-1}) $. For our application, we can pick $\ell = \ell_0 := [nd/2]+1$ in \eqref{cond.separation.onchi}. Therefore  such $k_j$ exists and $k_j \simeq  \e_{j-1}/\e_j$, provided
\begin{equation*}
    \frac{\e_j}{\e_{j-1}}M_j ([nd/2]+1 + 2\log(\e_1/\e_{j-1}) ) \le e^{-1}.
\end{equation*}
The above condition can be reduced to \eqref{cond.ej.sepcond},
for some small constant $c_j$. Thus if \eqref{cond.ej.sepcond} holds with $2 \le j\le n$ and
\begin{equation*}
    \tau^2  \simeq \max_j \{ e^{-c\e_{j-1}/\e_j} \} \gtrsim \max_j \{ M_j e^{-k_j/2+1} \},
\end{equation*}
then \eqref{cond.sep.n} is satisfied for any $\ell \le \ell_0$. It follows from Theorem \ref{thm.Y.nscale} that for any $0\le \ell \le \ell_0 = [nd/2] + 1$,
\begin{equation}\label{est.DlX.unif}
    \| \bfD_n^\ell \mathcal{X} \|_{L^2} \le C_n\Lambda_n^{\ell} (\ell + 1)!.
\end{equation}
Finally, the Sobolev embedding theorem yields \eqref{inftychi}.
\end{proof}

    For our application in the next section, we will face coefficient matrices in a form of
    $$A_x(y_1,\cdots, y_n) = A(x, y_1, y_2,\cdots, y_n)$$ relying on an additional variable $x$.
    This variable can be viewed as a parameter in the equations, whose differential operators only act on variables $y_i$ (i.e., $x$ does not interact with the differential operators). In fact, we will consider
\begin{equation}\label{eq.Xx}
    -\widehat{\nabla}_n \cdot A(x, y_1,\cdots, y_n) \widehat{\nabla}_n \mathcal{X} + \tau^2 \mathcal{X} = \widehat{\nabla}_n \cdot A(x,y_1,\cdots, y_n) v, 
\end{equation}
with a unit vector $v$. We are particularly interested in how $\mathcal{X}$ depends on the regularity of $A(x,\cdot)$ with respect to $x$.

\begin{theorem}  \label{thm.AppCorrector.X}
    Assume that $A = A(x,y_1,\cdots, y_n)$ satisfies the ellipticity condition \eqref{as.ellipticity} and is 1-periodic in each $y_i$. Moreover, assume  for all $\ell\ge 0$,
    \begin{equation}\label{ass.Ax}
        \sup_{x\in \R^N} | \bfD_n^\ell A(x,\cdot) | \le C_0 \Lambda_0^\ell \ell!,
    \end{equation}
    and
    \begin{equation}\label{ass.DxAx}
        \sup_{x\in \R^N} | \bfD_n^\ell \nabla_x A(x,\cdot) | \le L_0 \Lambda_0^\ell \ell!.
    \end{equation}
    Then under the same scale-separation conditions on $\e$ and $\tau$ as in Theorem \ref{thm.AppCorrector} with possibly different constants $c_j$ and $c$, there exists a unique solution to \eqref{eq.Xx} with $\Ag{\mathcal{X}(x,\cdot)} = 0$ such that
    \begin{equation}\label{est.XDxX}
       \sup_{x\in \R^N} \| \mathcal{X}(x,\cdot) \|_{L^\infty}\le C, \qquad \sup_{x \in \R^N} \| \nabla_x \mathcal{X}(x,\cdot) \|_{L^\infty} \le CL_0,
    \end{equation}
    and for all $\ell \ge 0$
    \begin{equation}\label{est.Ax.energy}
        \sup_{x\in \R^N} \| \bfD_n^\ell \widehat{\nabla}_n \mathcal{X}(x,\cdot) \|_{L^2} \le C_* \Lambda^\ell \ell!, \qquad  \sup_{x\in \R^N} \| \bfD_n^\ell \widehat{\nabla}_n \nabla_x \mathcal{X}(x,\cdot) \|_{L^2}\le L_0 C_\sharp \Lambda_\sharp^\ell \ell!.
    \end{equation}
    The constants $c,c_j, C, C_*, C_\sharp$ depend at most on $\lambda, C_0, d, n$ and $\Lambda_0$. Particularly, these constants are all independent of $L_0$.
\end{theorem}

\begin{proof}
    The first part of \eqref{est.XDxX} and the first part of \eqref{est.Ax.energy} follow from Theorem \ref{thm.AppCorrector} under the assumption \eqref{ass.Ax} for each fixed $x\in \R^d$ and then taking the supremum over $x$. To see the second part of \eqref{est.XDxX},
    we take derivative in $x$ to the equation \eqref{eq.Xx} to get
    \begin{equation}\label{eq.Ax-Xx}
        -\widehat{\nabla}_n \cdot A(x, y_1,\cdots, y_n) \widehat{\nabla}_n \nabla_x \mathcal{X} + \tau^2 \nabla_x \mathcal{X} = \widehat{\nabla}_n \cdot F_+(x,y_1,\cdots,y_n),
    \end{equation}
    where
    \begin{equation*}
        F_+(x,y_1,\cdots, y_n) = \nabla_x A(x, y_1,\cdots, y_n)v + \nabla_x A(x,y_1,\cdots, y_n) \widehat{\nabla}_n \mathcal{X}(x,y_1,\cdots, y_n).
    \end{equation*}
    Then by the assumption \eqref{ass.DxAx} and the first part of \eqref{est.Ax.energy},  using the generalized Leibniz rule, we have
    \begin{equation}\label{est.F+}
        \sup_{x\in \R^N} \| \bfD_n^\ell F_+(x,\cdot) \|_{L^2} \le L_0 C_+ \Lambda_+^\ell \ell!,
    \end{equation}
    with some $C_+$ and $\Lambda_+$ independent of $L_0$. In view of the linearity of the equation \eqref{eq.Ax-Xx} and \eqref{est.F+}, the estimates of $\nabla_x \mathcal{X}$ rely linearly on $L_0$.
    Hence, under the scale-separation condition \eqref{cond.sep.n} (with possibly different constants), we can apply Theorem \ref{thm.Y.nscale} to obtain
    \begin{equation*}
        \sup_{x\in \R^N} \| \bfD_n^\ell \nabla_x \X(x,\cdot)  \|_{L^2} \le L_0 C_n \Lambda_n^\ell (\ell+1)!,
    \end{equation*}
    for all $\ell \le [nd/2]+1$. Thus the Sobolev embedding theorem implies the second part of \eqref{est.XDxX}. Finally, the second part of \eqref{est.Ax.energy} follows from   Proposition \ref{prop.Ytau.F} applied to \eqref{eq.Ax-Xx}.
\end{proof}

Let $\mathcal{X}^i = \mathcal{X}^i(x,y_1,\cdots, y_n)$ be the solution of \eqref{eq.Xx} with $v = e_i$. Set
\begin{equation} \label{chi-de}
    \mathcal{X}_\delta^i (x,y) = \mathcal{X}^i(x, y/\delta_1, y/\delta_2,\cdots, y/\delta_n),
\end{equation}
and
\begin{equation}\label{b-de}
    B_\delta(x,y) = A(x, y/\delta_1, y/\delta_2, \cdots, y/\delta_n).
\end{equation}
Thus we have 
\begin{equation}\label{eq.Xdelta-X}
    \nabla_y (\mathcal{X}^i_\delta(x,y)) = (\widehat{\nabla}_n \mathcal{X}^i)(x,y/\delta_1,y/\delta_2,\cdots, y/\delta_n).
\end{equation}
Then \eqref{eq.Xx} yields
\begin{equation}\label{eq.X}
    -\nabla_y \cdot B_\delta(x,\cdot) \nabla_y \mathcal{X}_\delta^i  + \tau^2 \mathcal{X}_\delta^i = \nabla_y \cdot ( B_\delta(x,\cdot) e_i) \quad \text{for } y\in \R^d.
\end{equation}
This means that $\mathcal{X}^i_\delta$ are the multiscale correctors (with an additional parameter $x$).
Write $\mathcal{X}_\delta = (\mathcal{X}_\delta^1 ,\mathcal{X}_\delta^2,\cdots, \mathcal{X}_\delta^d)$. By Theorem \ref{thm.AppCorrector.X} and \eqref{eq.Xdelta-X}, under the corresponding assumptions, we have
\begin{equation}\label{chi.infty}
    \sup_{x\in \R^N} \| \mathcal{X}_\delta(x,\cdot) \|_{L^\infty} + \sup_{x\in \R^N} \| \nabla_y \mathcal{X}_\delta(x,\cdot) \|_{L^\infty} \le C,
\end{equation}
and
\begin{equation}\label{naxchi.infty}
    \sup_{x\in \R^N} \| \nabla_x \mathcal{X}_\delta(x,\cdot) \|_{L^\infty} + \sup_{x\in \R^N} \| \nabla_x \nabla_y \mathcal{X}_\delta(x,\cdot) \|_{L^\infty} \le CL_0.
\end{equation}
In the above estimate, the explicit dependence on the parameter $L_0$ will be crucial for us.

\section{Multiscale flux correctors}
\label{sec_flux-correctors}

In this section, we construct the flux correctors and establish their estimates. 

\subsection{Effective coefficient matrix}
Let $\X_\delta$  be the multiscale correctors given by \eqref{chi-de} and \eqref{eq.X}, with $B_\de$ given by \eqref{b-de}.  We consider at first the case where $B_\delta$ and $\X_\delta$ are independent of $x$, i.e., $\X_\delta(y) = \X(y/\delta_1,y/\delta_2,\cdots, y/\delta_n)$ and $B_\delta(y) = A(y/\delta_1,y/\delta_2,\cdots, y/\delta_n)$. Ideally, we shall solve the equation for constructing flux correctors 
\begin{align*}
  -\Delta_y \U_{\delta, ij}=\overline{A}_{ij}-B_{\delta,ij}-B_{\delta, i\ell}\partial_\ell \mathcal{X}^j_\delta \qquad \text{in } \R^d,
\end{align*}
in the quasi-periodic  setting, where $\overline{A} = (\overline{A}_{ij})$ is a constant matrix selected to guarantee the well-posedness of the equation. Since this equation may not be solvable, we instead consider the approximate (matrix-valued) equation
\begin{align}\label{flux-correctors_eq_approximateflux}
  -\Delta_y  \U_{\delta}+\tau^2 \U_{\delta}=\overline{A}-B_{\delta}-B_{\delta}\nabla_y \mathcal{X}_\delta \qquad \text{in } \R^d.
\end{align}
Due to the multiscale periodic structure of $B_\delta$ and $\X_\delta$,
we expect the solution has also the form
\begin{align*}
  \U_{\delta}(y)=\U (y/\delta_1, y/\delta_2, \cdots, y/\delta_n),
\end{align*}
where $\U = \U(y_1,y_2,\cdots, y_n)$ is 1-periodic in each $y_i\in \R^d$.
Therefore, the approximate equation \eqref{flux-correctors_eq_approximateflux} can be lifted to be a degenerate equation
\begin{align}\label{flux-correctors_eq_approximateflux_lift}
  -\widehat{\Delta}_n \U+\tau^2 \U=\overline{A}-A-A\widehat{\nabla}_n\mathcal{X} \qquad \text{in } \T^{d\times n},
\end{align}
where $\widehat{\Delta}_{n-1}=\widehat{\nabla}_{n-1}\cdot \widehat{\nabla}_{n-1}$. Note that \eqref{flux-correctors_eq_approximateflux_lift} is degenerate. If we require $\U$ to be uniformly bounded in $\tau$, it must hold $\langle \overline{A}-A-A\widehat{\nabla}_n\mathcal{X}\rangle=0$, i.e., 
\begin{align}\label{hatA}
  \overline{A}=\langle A+A\widehat{\nabla}_n\mathcal{X}\rangle,
\end{align} 
where $\Ag{\cdot}$ is the average over $\T^{d\times n}$ given in \eqref{def.average}. 

The constant matrix $\overline{A}$ will be called the effective matrix, which is more accurate than the classical homogenized matrix, denoted by $A_0$, derived by reiterated homogenization; see e.g. \cite{NSX20} for the derivation of $A_0$. We emphasize that $\overline{A}$ depends on $\tau$ and $\delta_i$'s, while $A_0$ is independent of $\tau$ and $\delta_i$. They have the following relation.

\begin{proposition}\label{prop.barA-A0}
Under the assumptions of Theorem \ref{thm.AppCorrector}, we have
  \begin{align}\label{error.A0-lineA}
    |A_0-\overline{A}| \le C(\tau^2+\delta_2+\delta_3\delta_2^{-1}+\cdots+\delta_n\delta_{n-1}^{-1}). 
  \end{align}
In particular,
  \begin{align}\label{def.A0} A_0=\lim_{\tau\rightarrow0}\lim_{\delta_2\rightarrow0}\cdots\lim_{\delta_n\rightarrow0} \overline{A}.
  \end{align}
\end{proposition}

\begin{proof}
  First of all, due to Theorem \ref{thm.AppCorrector}, the solution $\mathcal{X}$ to equation \eqref{eq.AppCorrector} exists during taking limits in \eqref{def.A0}, and thereby the limit in \eqref{def.A0} makes sense. 

Recall that in the process of reiterated homogenization, we homogenize first $y_n$ to get an interim matrix $A_{n-1}(y_1, \cdots, y_{n-1})$, and next homogenize $y_{n-1}$ to get $A_{n-2}(y_1, \cdots, y_{n-2})$, until every scale is homogenized successively. The last step is to homogenize $y_1$ from $A_1(y_1)$ to obtain the constant matrix $A_0$. Now we introduce a new matrix based on $A_1$. Set
\begin{align*}
  -\nabla_1\cdot A_1\nabla_1\psi_\tau+\tau^2\psi_\tau=\nabla_1\cdot A_1\quad \text{ in }\T^d,
\end{align*}
and define
\begin{align*}
  A_\tau=\langle A_1+A_1\nabla_1\psi_\tau\rangle.
\end{align*}
We claim that under the assumptions of Theorem \ref{thm.AppCorrector}
\begin{align}\label{claim.Atau}
  |A_\tau-\overline{A}|\leq C(\delta_2+\delta_3\delta_2^{-1}+\cdots+\delta_n\delta_{n-1}^{-1}).
\end{align}
As a byproduct, \eqref{claim.Atau} implies that
\begin{align*}
  A_\tau=\lim_{\delta_2\rightarrow0}\cdots\lim_{\delta_n\rightarrow0} \overline{A}.
\end{align*}
We will prove the claim by induction. 

  In the case $n=2$, according to the expansion for $Y = \X$ in Theorem \ref{thm.2S.Y}, it holds
  \begin{align*}
    \widehat{\nabla}_2 \mathcal{X} = \nabla_1Y_0 + \nabla_2Y_1^o+Z \quad \text{with }\|Z\|_{L^2}\leq C\delta_2,
  \end{align*}
 where $Y_0$ and $Y_1^o$ are given in Theorem \ref{thm.2S.Y} with $F=A$. Especially, $Y_0$ is identically $\psi_\tau$. Therefore, we have
\begin{align*}
  |\overline{A}-\langle A+A\nabla_1Y_0+A\nabla_2 Y_1^o\rangle|\leq \|Z\|_{L^2}=C\delta_2.
\end{align*}
Noticing that
\begin{align*}
  \langle A+A\nabla_1Y_0+A\nabla_2 Y_1^o\rangle=\langle A_1+A_1\nabla_1Y_0\rangle_{y_1}=A_\tau,
\end{align*}
we conclude $$|A_\tau-\overline{A}|\leq C\delta_2.$$

Assume now the claim is true for the case of $n-1$ scales. For the general case of $n$ scales, by the procedure of reiterated homogenization, we have
  \begin{align*}
      A_{n-1}(y_1, \cdots, y_{n-1})=\langle A+A\nabla_{n}\chi\rangle_{y_n},
  \end{align*}
    which coincides with $\widehat{A}$ in \eqref{sum.s12}. Note that $Y_0$ given in Theorem \ref{thm.Y.nscale} is exactly the corresponding multiscale corrector for $\widehat{A}$. By the inductive assumption, this leads to
  \begin{align}\label{es.Atau.n-1}
    |A_\tau-\langle A_{n-1}+A_{n-1}\widehat{\nabla}_{n-1}Y_0\rangle|\leq C(\delta_2+\cdots+\delta_{n-1}\delta_{n-2}^{-1}),
  \end{align}
  and
  \begin{align*}
A_\tau&=\lim_{\delta_2\rightarrow 0}\cdots \lim_{\delta_{n-1}\rightarrow 0}\langle A_{n-1}+A_{n-1}\widehat{\nabla}_{n-1}Y_0\rangle\\
      &=\lim_{\delta_2\rightarrow 0}\cdots \lim_{\delta_{n-1}\rightarrow 0}\langle A+A\widehat{\nabla}_{n-1}Y_0+A\nabla_n Y_1^o\rangle,
  \end{align*}
  under the assumptions of Theorem \ref{thm.AppCorrector}. On the other hand, taking into account the expansion for $Y = \X$ given in Theorem \ref{thm.Y.nscale}
\begin{align*}
  \mathcal{X}=\sum_{j=0}^{k-1}\delta_n^jY_j+\delta_n^kY_k^o+Z_k,  
\end{align*}
we have
\begin{align*}
  \widehat{\nabla}_n\mathcal{X}=\widehat{\nabla}_{n-1}Y_0+\nabla_nY_1^o+\sum_{j=1}^{k-1}\delta_n^j\widehat{\nabla}_{n-1}Y_j+\sum_{j=2}^{k}\delta_n^{j-1}\nabla_nY_j^o+\widehat{\nabla}_nZ_k,
\end{align*}
where $Z_k$ is given by \eqref{eq.Zk}. Therefore,
\begin{align}\label{eq-DnX}
  \widehat{\nabla}_n\mathcal{X}=\widehat{\nabla}_{n-1}Y_0+\nabla_nY_1^o+Z,
\end{align}
and $Z$ satisfies
\begin{align*}
  \|Z\|_{L^2}&\leq \sum_{j=1}^{k-1}\delta_n^j\delta_{n-1}^{-j}C^{3j+1}\Lambda^jj!+\sum_{j=2}^k\delta_n^{j-1}\delta_{n-1}^{-j+1}C^{3j-1}\Lambda^{j-1}(j-1)!\\&\qquad+\frac12 \delta_n^k \delta_{n-1}^{-k}  C^{3k} \Lambda^{k} k! + \frac12 \delta_n^{k-1} \delta_{n-1}^{-k} \tau^{-1} C^{3k} \Lambda^{k} k!\\&\leq \delta_n\delta_{n-1}^{-1}[2C^5\Lambda +\tau^{-1}\delta_{n-1}^{-1}e^{-k+2}C^6\Lambda^2],
\end{align*}
provided that $\delta_n\delta_{n-1}^{-1}C^3\Lambda k\leq 1/e$. Here the estimates for $Y_j$ and $\widehat{\nabla}_n Z_k$ have been used. Further, if $\tau^{-1}\delta_{n-1}^{-1}e^{-k+2}C\Lambda\leq 1$, it holds
\begin{align}\label{eq-Z-L2}
  \|Z\|_{L^2}\leq 3C^5\Lambda\delta_n\delta_{n-1}^{-1}.
\end{align}
Hence, taking \eqref{eq-DnX} and \eqref{eq-Z-L2} into \eqref{hatA}, we have
  \begin{align*}
    |\overline{A}-\langle A+A\widehat{\nabla}_{n-1}Y_0+A\nabla_n Y_1^o\rangle|=C\delta_n\delta_{n-1}^{-1},
  \end{align*}
  which, together with \eqref{es.Atau.n-1}, implies \eqref{claim.Atau}. Note that the conditions
  \begin{align*}
      \delta_n\delta_{n-1}^{-1}C^3\Lambda k\leq 1/e\quad\text{and}\quad
      \tau^{-1}\delta_{n-1}^{-1}e^{-k+2}C\Lambda\leq 1
  \end{align*}
are fulfilled under the assumptions of Theorem \ref{thm.AppCorrector}.

To end the proof of the proposition, it remains to show
\begin{align*}
  |A_0-A_\tau|\leq C\tau^2. 
\end{align*}
We introduce the exact corrector $\psi_0$ for $A_1$ by
  \begin{align*}
    -\nabla_1 \cdot A_1\nabla_1\psi_0=\nabla_1\cdot A_1\quad\text{ in }\T^d
  \end{align*}
  and get
  \begin{align*}
    A_0=\langle A_1+A_1\nabla_1\psi_0\rangle_{y_1}.
  \end{align*}
  Comparing the equations of $\psi_\tau$ and $\psi_0$, we deduce from the energy estimate that
  \begin{align*}
    |A_0-A_\tau|\leq \|\nabla_1\psi_\tau-\nabla_1\psi_0\|_{L^2}\leq C\tau^2,
  \end{align*}
  where $C$ depends only on the characters of $A$. This combined with \eqref{claim.Atau} completes the proof.
\end{proof}

\begin{remark}\label{rmk.barA-A0}
    Recall that $\delta_i = \e_i/\e_1$ and $\tau^2 \simeq \max_j \{ e^{-c\e_{j-1}/\e_j} \} $ in our application to $A_\e(x) = A(x/\e_1,\cdots, x/\e_n)$. Thus \eqref{error.A0-lineA} yields
    \begin{equation*}
        |A_0 - \overline{A}| \le C( \max_j \{ e^{-c\e_{j-1}/\e_j} \} + \max_j \{ \e_j/\e_{j-1} \}).
    \end{equation*}
    This perfectly explains the difference between the errors in \eqref{est.MainRate} and \eqref{rate.LipA}.
\end{remark}

\subsection{Multiscale ansatz}
\label{sec_multiscale-flux}

We consider the following general equation
\begin{equation}
  \label{flux-correctors_eq_energy1}
  -\widehat{\Delta}_{n} U+\tau^2 U=G \qquad \text{in } \T^{d\times n},
\end{equation}
where $G=G(y_1,\cdots,y_n)$ is $1$-periodic in each $y_i$ with $\langle G\rangle=0$, and satisfies \eqref{cond.F2}. By a regularization argument, this degenerate equation admits a unique periodic solution $U$ with $\langle U\rangle=0$. 

To use the structure of equation \eqref{flux-correctors_eq_energy1}, we apply the ansatz
\begin{align*}
  U(y_1, \cdots, y_n)=\sum_{j=0}^\infty \delta_n^j U_j(y_1, \cdots, y_n).
\end{align*}
Then taking the ansatz of $U$ into equation \eqref{flux-correctors_eq_energy1}, using 
$$\widehat{\Delta}_n = \widehat{\Delta}_{n-1} + \delta_n^{-1} 2 \widehat{\nabla}_{n-1} \cdot \nabla_n + \delta_n^{-2} \Delta_n$$ 
and comparing the order of $\delta_n$, we have the following recursive system for $U_j$,
\begin{subequations}\label{flux-correctors_eq_U_systems}
    \begin{empheq}[left=\empheqlbrace]{align}
    &-\Delta_{n}U_0=0, \label{flux-U0}\\
    &-\Delta_{n}U_1-2\widehat{\nabla}_{n-1}\cdot\nabla_{n}U_0=0, \label{flux-U1}\\
    &-\Delta_{n}U_{2}-2\widehat{\nabla}_{n-1}\cdot\nabla_{n}U_{1}-\widehat{\Delta}_{n-1}U_0+\tau^2 U_0=G, \label{flux-U2}\\
    &-\Delta_{n}U_{k+2}-2\widehat{\nabla}_{n-1}\cdot\nabla_{n}U_{k+1}-\widehat{\Delta}_{n-1}U_k+\tau^2U_k=0,\quad k\geq 1. \label{flux-Uk}
    \end{empheq}
\end{subequations}

Now we solve \eqref{flux-correctors_eq_U_systems} one by one. By \eqref{flux-U0}, $U_0$ is independent of $y_n$ and therefore $U_0 = U_0(y_1,\cdots, y_{n-1})$. Taking average in $y_n$ to \eqref{flux-U2}, we get
\begin{align*}
    -\widehat{\Delta}_{n-1}U_0+\tau^2 U_0=\langle G\rangle_{y_n}.
\end{align*}
This equation has the same structure as \eqref{flux-correctors_eq_energy1} with one less scale and suggests that we need to solve it inductively.

Furthermore, since $U_0$ is independent of $y_n$, \eqref{flux-U1} becomes
\begin{align*}
  -\Delta_nU_1=0,
\end{align*}
which means $U_1$ is also independent of $y_n$ and $U_1=U_1(y_1, \cdots, y_{n-1})$. By taking the average in $y_n$ to the equation \eqref{flux-Uk} with $k=1$, we have
\begin{align*}
    -\widehat{\Delta}_{n-1}U_1+\tau^2U_1=0,
\end{align*}
which implies that $U_1=0$ by a simple integration by parts (or energy estimate). 

Similarly, for general $k\geq 2$, by taking average in $y_n$ to the last equation in \eqref{flux-correctors_eq_U_systems}, we get
\begin{equation*}
    -\widehat{\Delta}_{n-1}U_{k}^r+\tau^2U_{k}^r=0,
\end{equation*}
where $U_{k}^r(y_1, \cdots, y_{n-1})=\langle U_{k}\rangle_{y_n}$. The last equation only has a trivial solution, which means that $U_k$ has no regular component and $\langle U_{k}\rangle_{y_n}=0$ for each $k\geq 2$. Then we can solve for $U_k$ with $k\ge 2$ directly from \eqref{flux-U2} and \eqref{flux-Uk}. In fact, from \eqref{flux-U2}, we deduce
\begin{align}\label{flux-correctors_eq_U2}
  -\Delta_nU_2=\widehat{\Delta}_{n-1} U_0-\tau^2U_0+G = G - \Ag{G}_{y_n}.
\end{align}
Obviously, the right-hand side has mean value $0$ with respect to $y_n$. Therefore,  \eqref{flux-correctors_eq_U2} admits a unique solution in $y_n$ with mean value $0$. 
Moreover, for $k\geq 1$, if $U_{k+1}$ and $U_{k}$ have been solved, one obtains from \eqref{flux-Uk}
\begin{equation*}
  -\Delta_{n}U_{k+2}=2\widehat{\nabla}_{n-1}\cdot\nabla_nU_{k+1}+\widehat{\Delta}_{n-1}U_k-\tau^2U_k.
\end{equation*}
Note that this equation admits a unique solution in $y_n$ with mean value $0$, since $\langle U_k\rangle_{y_n}=0$. In this way, $U_k$ can be derived recursively.

We summarize the recursive equations to find all $U_k$ with $\Ag{U_k}_{y_n} = 0$:
\begin{align}
    -\widehat{\Delta}_{n-1}U_0+\tau^2 U_0 & =\langle G\rangle_{y_n}, \label{eq.U0}\\
    U_1 & = 0, \label{eq.U1}\\
    -\Delta_nU_2& =G - \Ag{G}_{y_n} , \label{eq.U2}\\
    -\Delta_{n}U_{k+2}& =2\widehat{\nabla}_{n-1}\cdot\nabla_nU_{k+1}+\widehat{\Delta}_{n-1}U_k-\tau^2U_k, \quad \txt{for } k \ge 1. \label{eq.Uk}
\end{align}
Note that only the equation \eqref{eq.U0} for $U_0$ needs to be solved by an inductive argument.

\subsection{Energy estimates}
\label{sec_energy-estimates}

The solution $U$ to the equation \eqref{flux-correctors_eq_energy1} satisfies the following energy estimate.

\begin{proposition}\label{flux-correctors_prop_energy1}
  Let $G$  satisfy \eqref{cond.F2} and $\langle G\rangle=0$. Then there exists a constant $C_d$, depending only on $d$, such that for $\ell\geq 0$, the solution $U$ to equation \eqref{flux-correctors_eq_energy1} satisfies
  \begin{align*}
    \|\bfD^\ell_n \widehat{\nabla}_{n}^2U\|_{L^2}+\tau\|\bfD^\ell_n \widehat{\nabla}_{n}U\|_{L^2}+\tau^2\|\bfD^\ell_n U\|_{L^2}\leq C_d\Lambda_1^\ell(\ell+q)!.
  \end{align*}
\end{proposition}
\begin{proof}
  The estimate for $\ell=0$ follows directly by applying the test functions $-\widehat{\Delta}_{n} U$ and $U$ to \eqref{flux-correctors_eq_energy1}. Alternatively, one can also take $L^2$ norm of the both sides of \eqref{flux-correctors_eq_energy1} and apply the integration by parts to get
  \begin{equation*}
      \| \widehat{\nabla}_{n}^2U\|_{L^2} +\tau \| \widehat{\nabla}_{n}U\|_{L^2} +\tau^2\| U\|_{L^2}  \le C_d \| G \|_{L^2}.
  \end{equation*}
  For general $\ell>0$, applying $\bfD_n^\ell$ to \eqref{flux-correctors_eq_energy1} and the last inequality to the corresponding tensor-valued equation, we have
    \begin{equation*}
        \| \bfD_n^\ell \widehat{\nabla}_{n}^2U\|_{L^2} +\tau \| \bfD_n^\ell \widehat{\nabla}_{n}U\|_{L^2} +\tau^2\| \bfD^\ell U\|_{L^2}  \le C_d \| \bfD_n^\ell G \|_{L^2} \le C_d \Lambda_1^\ell (\ell+ q)!.
    \end{equation*}
    This is the desired estimate.
\end{proof}

By applying this proposition to equation \eqref{flux-correctors_eq_approximateflux_lift} and using \eqref{est.X.energy} in Theorem \ref{thm.AppCorrector}, we have for all $\ell\geq 0$
\begin{align}
  \label{flux-correctors_es_energy}
  \|\bfD^\ell_n \widehat{\nabla}_{n}^2 \U\|_{L^2}+\tau\|\bfD^\ell_n \widehat{\nabla}_{n} \U\|_{L^2}+\tau^2\|\bfD^\ell_n \U\|_{L^2}\leq C_d\Big(C_0+\frac{C_0C_*^2}{C_*-1}\Big)\Lambda^\ell\ell!,
\end{align}
with $\Lambda = C_* \Lambda_0$.
However, this estimate is insufficient to achieve the desired convergence rate. To proceed, we need an estimate of $\| \bfD_n^\ell \U \|_{L^2}$ uniform in $\tau$. 
The following standard $H^2$ estimate plays an essential role in establishing this uniform bound. 
\begin{proposition}\label{flux-correctors_prop_energy2}
Let  $G$ satisfy \eqref{cond.F2} and $\langle G\rangle_{y_n}=0$. Let $U$ be the  periodic solution of
  \begin{align*}
    -\Delta_n U=G \quad \text{ in }\T^d
  \end{align*}
with $\langle U\rangle_{y_n}=0$. Then there exists $C_d$, depending only on $d$, such that for all $\ell\geq 0$, 
  \begin{align}\label{flux-correctors_es_U_energy}
    \|\bfD^\ell_n U\|_{L^2_{\mathbf{x}}H^2_{\mathbf{y}}}\leq C_d\Lambda_1^\ell(\ell+q)!.
  \end{align}
  Furthermore, if instead of \eqref{cond.F2},  $G$ satisfies
  \begin{align*}
    \|\bfD^\ell_n G\|_{L^\infty_{\mathbf{x}}L^2_{\mathbf{y}}}\leq \Lambda_1^\ell(\ell+q)!,
  \end{align*}
  then the norm of $U$ in \eqref{flux-correctors_es_U_energy} can also be replaced by $\|\bfD^\ell_n U\|_{L^\infty_{\mathbf{x}}H^2_{\mathbf{y}}}$. 
\end{proposition}

\subsection{Uniform estimates}
\label{sec_structure-U}

To obtain refined control for $U$, we establish the following estimates of $U_k$ based on the recursive equations in Section \ref{sec_multiscale-flux}. 

\begin{proposition}\label{m_prop_Uk}
  Suppose that $\langle G\rangle=0$ and $G$ satisfies \eqref{cond.F2}. Then there exists a constant $C$, depending only on $d$, such that for $\ell\geq 0$
  \begin{align*}
    \|\bfD^\ell_n \widehat{\nabla}_{n-1}^2U_0\|_{L^2}+\tau\|\bfD^\ell_n \widehat{\nabla}_{n-1}U_0\|_{L^2}+\tau^2\|\bfD^\ell_n U_0\|_{L^2}\leq C\Lambda_1^\ell (\ell+q)!,
  \end{align*}
  and for $k\geq 2$
  \begin{align*}
    \|\bfD^\ell_n U_k\|_{L^2_\x H^2_\y}\leq 
      \delta_{n-1}^{-k+2}C^{k-1}\Lambda_1^{\ell+k-2}(\ell+k-2+q)!.
  \end{align*}
\end{proposition}

\begin{proof}  
  Firstly, by applying Proposition \ref{flux-correctors_prop_energy1} in $n-1$ scales to the equation \eqref{eq.U0}, we obtain directly for $\ell\geq 0$
  \begin{align*}
    \|\bfD^\ell_n \widehat{\nabla}_{n-1}^2U_0\|_{L^2}+\tau\|\bfD^\ell_n \widehat{\nabla}_{n-1}U_0\|_{L^2}+\tau^2\|\bfD^\ell_n U_0\|_{L^2}\leq C_d\Lambda_1^\ell (\ell+q)!.
  \end{align*}

  Setting $G_2 = G - \Ag{G}_{y_n}$, we have
  \begin{align*}
    \|\bfD^\ell_n G_2\|_{L^2}\leq 2\Lambda_1^\ell(\ell+q)!.
  \end{align*}
By Proposition \ref{flux-correctors_prop_energy2} applied to equation \eqref{eq.U2}, this yields
\begin{align*}
  \|\bfD^\ell_n U_2\|_{L^2_\x H^2_\y}\leq 2C_d \Lambda_1^\ell(\ell+q)!\leq C\Lambda_1^\ell(\ell+q)!,
\end{align*}
where we choose $C\geq 2C_d$. 
  
Next we establish by induction that for all $k \ge 1, \ell\geq 0$,
  \begin{align*}
    \|\bfD^\ell_n U_{k+2}\|_{L^2_\x H^2_\y}\leq 
      \delta_{n-1}^{-k}C^{(k+2)-1}\Lambda_1^{\ell+(k+2)-2}(\ell+(k+2)-2+q)!.
  \end{align*}
Suppose now the estimates of $U_j$ have been known for $1\leq j\leq k+1$. Let $G_{k+2} = 2\widehat{\nabla}_{n-1}\cdot\nabla_nU_{k+1}+\widehat{\Delta}_{n-1}U_k-\tau^2U_k$. Then by the inductive assumption,
\begin{equation}\label{est.Gk+2}
\begin{aligned}
    \| \bfD_n^\ell G_{k+2} \|_{L^2} & \le 2\delta_{n-1}^{-1} \| \bfD_n^{\ell+1} U_{k+1} \|_{L_\x^2 H_\y^1} + \delta_{n-1}^{-2} \| \bfD_n^{\ell+2} U_{k} \|_{L^2} + \tau^2 \| \bfD_n^\ell U  \|_{L^2} \\
    & \le 2\delta_{n-1}^{-1} \delta_{n-1}^{k-1} C^{(k+1)-1} \Lambda_1^{\ell+1+(k+1)-2} (\ell+1 + (k+1) -2+q)! \\
    & \quad + \delta_{n-1}^{-2} \delta_{n-1}^{k-2} C^{k-1} \Lambda_1^{\ell+2+k-2} (\ell+2 + k -2+q)! \\
    & \quad + \tau^2 \delta_{n-1}^{k-2} C^{k-1} \Lambda_1^{\ell+k-2} (\ell + k -2+q)! \\
    & \le (3+\tau^2)\delta_{n-1}^{-k} C^{k} \Lambda_1^{\ell+(k+2)-2} (\ell+(k+2)-2+q)!.
    \end{aligned}
\end{equation}
By applying Proposition \ref{flux-correctors_prop_energy2} to the equation \eqref{eq.Uk} and using \eqref{est.Gk+2}, we have
    \begin{align*}
    \|\bfD^\ell_n U_{k+2}\|_{L^2_\x H^2_\y} & \leq (3+\tau^2)C_d\delta_{n-1}^{-k}C^{k}\Lambda_1^{\ell+(k+2)-2}(\ell+(k+2)-2+q)!\\&\leq \delta_{n-1}^{-k}C^{(k+2)-1}\Lambda_1^{\ell+(k+2)-2}(\ell+(k+2)-2+q)!,
    \end{align*}
    where we have assumed that $C\geq (3+\tau^2)C_d$. This completes the proof.
\end{proof}

Now we combine the structure estimates in Proposition \ref{m_prop_Uk} and the energy estimates to get a uniform control for the solutions of \eqref{flux-correctors_eq_energy1} under a scale-separation condition. Note that the energy estimate \eqref{flux-correctors_es_energy} only provides effective control for $\|\bfD^\ell_n \widehat{\nabla}_{n}^2 U\|_{L^2}$ uniform in $\tau \in (0,1)$, which does not imply the uniform boundedness of $U$.
\begin{theorem}\label{flux-correctors_thm_U_truncation}
Under the same assumption of Proposition \ref{m_prop_Uk}, there exist $C, \widetilde{C}$ and $C_n$, depending only on $d $ and $ n$, such that if for some $\ell \ge 0$ and any $2\le j\le n$, there exists $k_j \ge 2$ such that
\begin{align}\label{sep-con2}
      \begin{cases}
        (\delta_j/\delta_{j-1})C\Lambda_1 (\ell+k_j+q)\leq e^{-1},\\
        \tau^2\geq \widetilde{C} e^{-k_j},
      \end{cases}
    \end{align}
    then \eqref{flux-correctors_eq_energy1} admits a unique solution $U$ satisfying   
    \begin{align} \label{est.DlU.unif}
      \|\bfD^\ell_n U\|_{L^2}\leq C_n\Lambda_1^\ell (\ell+q)!.
    \end{align}
  \end{theorem}

\begin{proof}
Due to the factorial growth in $j$ of the estimates of $U_j$ in Proposition \ref{m_prop_Uk}, we turn to the truncated partial sum
\begin{align*}
    V_k(y_1, \cdots,  y_n)=\sum_{j=0}^{k}\delta_n^jU_j(y_1, \cdots, y_n),
\end{align*}
which satisfies
    \begin{align}
      \label{flux-correctors_eq_flux_error}
      \begin{split}
        &-\widehat{\Delta}_nV_k+\tau^2 V_k-G\\
        & =\mathcal{G}_k:=\delta_n^k(-\widehat{\Delta}_{n-1}U_k+\tau^2U_k) +\delta_n^{k-1}(-2\widehat{\nabla}_{n-1}\cdot\nabla_n U_k-\widehat{\Delta}_{n-1}U_{k-1}+\tau^2U_{k-1}).
      \end{split}
    \end{align}
The choice of $k = k_n \ge 2$ will be assumed to satisfy the scale-separation condition \eqref{sep-con2}.
Let $W_k$ be the solution to
\begin{equation*}
      -\widehat{\Delta}_nW_k+\tau^2 W_k=-\mathcal{G}_k.
\end{equation*}
Then we have
\begin{equation}\label{eq.U=V+W}
    U = V_k + W_k = U_0 + \sum_{j=2}^k \delta_n^j U_j + W_k.
\end{equation}
    
Note that the uniform estimates of $U_j$ with $j \ge 2$ can be derived directly from Proposition \ref{m_prop_Uk}. In fact,
    \begin{equation*}
        \Big\| \bfD_n^\ell \sum_{j=2}^k \delta_n^j U_j \Big\|_{L^2_\x H^2_\y} \le \sum_{j=2}^{k}\delta_n^{j}\delta_{n-1}^{-j+2}C^{j-1}\Lambda_1^{\ell+j-2}(\ell+j-2+q)! \leq \frac{Ce \delta_n^2}{e-1} \Lambda_1^{\ell}(\ell+q)!,
    \end{equation*}
    provided that $(\delta_n/\delta_{n-1}) C\Lambda_1 (\ell+k-2+q)\leq 1/e$. Since $\Ag{U_j}_{y_n} = 0$, the Poincar\'{e} inequality implies
    \begin{equation}\label{est.U-sumterm}
        \Big\| \bfD_n^\ell \sum_{j=2}^k \delta_n^j U_j \Big\|_{L^2} \le C_d \Big\| \bfD_n^\ell \sum_{j=2}^k \delta_n^j U_j \Big\|_{L^2_\x H^2_\y} \le \frac{C_d C e \delta_n^2}{e-1}  \Lambda_1^{\ell}(\ell+q)! \le \Lambda_1^{\ell}(\ell+q)!.
    \end{equation}
    In the last inequality we have used the fact $\delta_n \le 1/(Ce)$, which is due to the scale-separation assumption.

Next, we estimate $W_k$, by applying the energy estimate to \eqref{flux-correctors_eq_flux_error}. Observe from Proposition \ref{m_prop_Uk} that for $\ell\geq 0$
\begin{align*}
    \|\bfD^\ell_n \mathcal{G}_k\|_{L^2}  & \le \delta_n^k ( \delta_{n-1}^{-2}\| \bfD_n^{\ell+2} U_k \|_{L^2} + \tau^2 \| \bfD_n^\ell U_k \|_{L^2} ) \\
    & \quad + \delta_n^{k-1} ( 2\delta_{n-1}^{-1} \| \bfD_n^{\ell+1} U_k \|_{L^2_\x H^1_\y} + \delta_{n-1}^{-2} \| \bfD_n^{\ell+2} U_{k-1} \|_{L^2} + \tau^2 \| \bfD_n^\ell U_{k-1} \|_{L^2} )\\
    &\leq \delta_n^k\delta_{n-1}^{-k}C^{k-1}\Lambda_1^{\ell+k}(\ell+k+q)!+3\delta_n^{k-1}\delta_{n-1}^{-k+1}C^{k-1}\Lambda_1^{\ell+k-1}(\ell+k-1+q)!\\&\quad+\tau^2\delta_n^k\delta_{n-1}^{-k+2}C^{k-1}\Lambda_1^{\ell+k-2}(\ell+k-2+q)!\\
    &\quad+\tau^2\delta_n^{k-1}\delta_{n-1}^{-k+3}C^{k-2}\Lambda_1^{\ell+k-3}(\ell+k-3+q)! \\
      & \le (4+2\tau^2) e^{-k+1} \Lambda_1^\ell (\ell+q)!,
    \end{align*}
    provided that $(\delta_n/\delta_{n-1}) C\Lambda_1 (\ell+k+q)\leq 1/e$.
    Then Proposition \ref{flux-correctors_prop_energy1} applied to \eqref{flux-correctors_eq_flux_error} yields
    \begin{align*}
      \tau^2 \|  \bfD^\ell_n W_k \|_{L^2} \le C_d \| \bfD_n^\ell \mathcal{G}_k \|_{L^2}
        \leq 6C_d e^{-k+1}\Lambda_1^\ell(\ell+q)!.
    \end{align*}
    Let $\widetilde{C} = 6e C_d$ and assume $\tau$ satisfies \eqref{sep-con2}. Then the last inequality implies
    \begin{equation}\label{est.U-Wterm}
        \|  \bfD^\ell_n W_k \|_{L^2} \le \Lambda_1^\ell(\ell+q)!.
    \end{equation}
It is crucial to notice that the estimates \eqref{est.U-sumterm} and \eqref{est.U-Wterm} are derived directly for any scale $n$ without using induction. The scale-separation condition \eqref{sep-con2} is required for these estimates in the case of $j$ scales for each $2\le j\le n$. The constant $C$ and $\widetilde{C}$ may be different for different $j$'s; but we choose a unified one depending  only on $n$ and $d$.

Finally, in view of \eqref{eq.U=V+W}, to obtain the uniform bound of $U$, it suffices to estimate $U_0$. We will obtain the estimate of the pair $(U_0, U)$ via  an inductive argument on the number of scales. We begin with the case $n=2$. In this case, $\widehat{\nabla}_1 = \nabla_1 = \bfD_1$ and the equation \eqref{eq.U0} for $U_0$  is not degenerate.
Due to Proposition \ref{m_prop_Uk},  we know for $\ell\geq 0$
\begin{align*}
  \|\bfD_1^{\ell+2} U_0\|_{L^2_\x}=\|\bfD_1^\ell \widehat{\nabla}^2_1 U_0\|_{L^2}\leq C\Lambda_1^{\ell}(\ell+q)!,
\end{align*}
which, by the Poincar\'{e} inequality, yields that for $\ell=0, 1$,
\begin{align*}
  \|U_0\|_{L^2_\x} + \|\bfD_1 U_0\|_{L^2_\x} \leq C_d C q!.
\end{align*}
For simplicity, we can summarize for $\ell\geq 0$
\begin{align}\label{flux-correctors_es_U0}
  \|\bfD_1^{\ell} U_0\|_{L^2_\x}\leq C_1\Lambda_1^{(\ell-2)_+}((\ell-2)_++q)!,
\end{align}
where $C_1=\max\{C, C_dC\}$ and $(a)_+ = \max\{ a, 0 \}$. Combining \eqref{est.U-sumterm} and \eqref{est.U-Wterm} with \eqref{flux-correctors_es_U0}, we have
\begin{equation*}
    \| \bfD_2^\ell U \|_{L^2} \le \|\bfD_1^{\ell} U_0\|_{L^2_\x} + 2 \Lambda_1^\ell (\ell + q)! \le (C_1 + 2) \Lambda_1^\ell (\ell + q)!.
\end{equation*}
Without loss of generality, we have assumed $\Lambda_1 \ge 1$. Hence, we arrived \eqref{est.DlU.unif} for $n=2$ with $C_2 = C_1 + 2$.

Now, we assume \eqref{est.DlU.unif} holds for $n-1$ scales and prove it for $n$ scales. Indeed, if $U$ is the solution of \eqref{flux-correctors_eq_energy1} with $n$ scales, then $U_0$ is the solution of \eqref{eq.U0}, which is an equation of the same type with $n-1$ scales. Clearly $\Ag{G}_{y_n}$ satisfies the same assumption as $G$. Thus by the inductive assumption,
\begin{equation*}
    \| \bfD_{n-1}^\ell U_0 \|_{L^2} \le C_{n-1} \Lambda_1^\ell \ell!.
\end{equation*}
Consequently, by \eqref{est.U-sumterm}, \eqref{est.U-Wterm} and \eqref{eq.U=V+W}, we have
\begin{equation*}
    \| \bfD_n^\ell U \|_{L^2} \le \|\bfD_{n-1}^{\ell} U_0\|_{L^2_\x} + 2 \Lambda_1^\ell (\ell + q)! \le (C_{n-1} + 2) \Lambda_1^\ell (\ell + q)!.
\end{equation*}
This proves \eqref{est.DlU.unif} for general $n$ scales with $C_n = C_{n-1} + 2$. Note that $C_1$ (a constant arising for $n=2$) depends only on $d$ and thus $C_n$ can be given simply by $C_n = C_1 + 2(n-1)$.
\end{proof}

 \begin{theorem}\label{flux-correctors_coro_Ubound}
 Suppose that $A$ satisfies \eqref{as.ellipticity}, \eqref{as.periodicity} and \eqref{as.analyticity}. Let $\U$ be the solution of \eqref{flux-correctors_eq_approximateflux_lift} with \eqref{hatA}. Then there exist constants $\widetilde{C}, c>0$, depending only on $d, n, C_0$ and $\Lambda_0$, such that if
 \begin{equation}\label{cond.ei.fluxsep}
     \e_{j+1} \le \widetilde{C}^{-1}\e_j, \qquad \text{for each } j = 1,2,\cdots, n-1,
 \end{equation}
 then for some $\tau \simeq \max_i \{ e^{-c\e_i/\e_{i+1}} \}$,  \eqref{flux-correctors_eq_approximateflux_lift} has a unique bounded solution satisfying
 \begin{equation*}
     \| \widehat{\nabla}_n^2 \U \|_{L^\infty} + \| \widehat{\nabla}_n \U \|_{L^\infty} + \| \U \|_{L^\infty} \le C,
 \end{equation*}
 where  $C$ depends only on the characters of $A$. 
\end{theorem}

 \begin{proof}
   Note that by \eqref{est.X.energy} in Theorem \ref{thm.AppCorrector}, $G=\overline{A}-A-A\cdot\widehat{\nabla}_n\mathcal{X}$ satisfies for $\ell\geq 0$,
   \begin{align*}
     \|\bfD^\ell_n G\|_{L^2}\leq C_0\Lambda_0^\ell \ell!+\sum_{m=0}^\ell \binom{\ell}{m}C_0\Lambda_0^mm!C_*\Lambda^{\ell-m}(\ell-m)!\leq C_0\Big(1+\frac{C_*^2}{C_*-1}\Big)\Lambda^\ell\ell!,
   \end{align*}
   where $\Lambda, C_*$ are as given in Theorem \ref{thm.AppCorrector}. Thanks to Theorem \ref{flux-correctors_thm_U_truncation}, under \eqref{sep-con2} we have
   \begin{align*}
     \|\bfD^\ell_n \U\|_{L^2}\leq C_nC_0\Big(1+\frac{C_*^2}{C_*-1}\Big)\Lambda^\ell\ell!.
   \end{align*}
   In particular, we let $0\le \ell \le \ell_0 := [nd]+1$. Then \eqref{sep-con2} holds for some $k_j$ if \eqref{cond.ei.fluxsep} holds for sufficiently large $\widetilde{C}$ and $\tau \simeq \max_i \{ e^{-c\e_{i+1}/\e_{i}} \}$ for sufficiently small $c$. Then by the Sobolev embedding theorem, we have
   \begin{equation*}
     \|\U \|_{L^\infty}\leq C.
   \end{equation*}

   On the other hand, Proposition \ref{flux-correctors_prop_energy1} and the Sobolev embedding theorem imply that 
    \begin{equation*}
     \| \widehat{\nabla}_{n}^2 \U \|_{L^\infty} \le C_d \sum_{\ell=0}^{\ell_0} \| \bfD_n^{\ell} \widehat{\nabla}_{n}^2 \U \|_{L^2} \leq C.
   \end{equation*}
   Now, observe that for $0\le \ell \le \ell_0$, by the integration by parts, we have
   \begin{equation*}
       \int_{\T^{d\times n}} |\widehat{\nabla}_n  \bfD_n^\ell \U|^2 = -\int_{\T^{d\times n}} \bfD_n^\ell \U \cdot \widehat{\Delta}_n \bfD^\ell_n \U \le \| \bfD_n^\ell \U \|_{L^2} \| \bfD_n^\ell \widehat{\nabla}_n^2 \U \|_{L^2} \le C.
   \end{equation*}
   Again, by the Sobolev embedding theorem, we obtain
   \begin{equation*}
       \| \widehat{\nabla}_{n} \U \|_{L^\infty} \le C,
   \end{equation*}
   for some $C$ depending only on $d, n, C_0$ and $\Lambda_0$. The proof is complete.
\end{proof}

\begin{remark}
In contrast to the strong scale-separation condition \eqref{cond.ej.sepcond} in Theorem \ref{thm.AppCorrector} , we only need a weak scale-separation condition \eqref{cond.ei.fluxsep} in Theorem \ref{flux-correctors_thm_U_truncation}. It seems that this distinction comes from the structure of the equation \eqref{flux-correctors_eq_energy1} for which the $H^2$ estimate is available. This structure makes $U_k$ have one more power of $\delta_{n-1}$ than $Y_k$.
\end{remark}

 For later applications, we will also consider the multiscale flux correctors for the coefficient matrix $A(x,y_1,\cdots, y_n)$ depending on an additional variable $x \in \R^d$. Let $\U = \U(x,y_1,\cdots, y_n)$ be the solution of
 \begin{equation}\label{eq.Uxy}
     -\widehat{\Delta}_n \U(x,\cdot)+\tau^2 \U(x,\cdot)=\overline{A}(x)-A(x,\cdot)-A(x,\cdot)\widehat{\nabla}_n\mathcal{X}(x,\cdot).
 \end{equation}
 
 Analogous to Theorem \ref{thm.AppCorrector.X}, we also have the following theorem as a corollary of Theorem \ref{flux-correctors_coro_Ubound}. 

 \begin{theorem}\label{thm.Ux-unif}
     Let $A$ satisfy the same assumptions as Theorem \ref{thm.AppCorrector.X}. Then under the same scale-separation conditions for $\e$ and $\tau$ as Theorem \ref{flux-correctors_coro_Ubound} (with possibly different constant $\widetilde{C}$ and $c$), we have
     \begin{equation}\label{est.Uxy}
         \sup_{x\in \R^d}\| \widehat{\nabla}_n^2 \U(x,\cdot) \|_{L^\infty} + \sup_{x\in \R^d}\| \widehat{\nabla}_n \U(x,\cdot) \|_{L^\infty} + \sup_{x\in \R^d} \| \U(x,\cdot) \|_{L^\infty} \le C,
     \end{equation}
     and
     \begin{equation}\label{est.DxUxy}
         \sup_{x\in \R^d}\| \widehat{\nabla}_n^2 \nabla_x \U(x,\cdot) \|_{L^\infty} + \sup_{x\in \R^d}\| \widehat{\nabla}_n \nabla_x \U(x,\cdot) \|_{L^\infty} + \sup_{x\in \R^d} \| \nabla_x \U(x,\cdot) \|_{L^\infty} \le C L_0,
     \end{equation}
     where $c$, $\widetilde{C}$ and $C$ depend only on $\lambda, d, n, C_0$ and $\Lambda_0$. In particular, they are independent of $L_0$.
 \end{theorem}
\begin{proof}
    Let $G(x,\cdot) = \overline{A}(x)-A(x,\cdot)-A(x,\cdot)\widehat{\nabla}_n\mathcal{X}(x,\cdot)$. The estimate \eqref{est.Uxy} follows straightforward from Theorem \ref{flux-correctors_coro_Ubound}. To obtain \eqref{est.DxUxy}, we apply $\nabla_x$ to the equation \eqref{eq.Uxy} to obtain
    \begin{equation}\label{eq.DxU-DxG}
        -\widehat{\Delta}_n \nabla_x \U(x,\cdot)+\tau^2  \nabla_x \U(x,\cdot) = \nabla_x G(x,\cdot).
    \end{equation}
    By \eqref{est.Ax.energy}, we can show for any $\ell \ge 0$,
    \begin{equation*}
    \| \bfD_n^\ell \nabla_x G(x,\cdot) \|_{L^2} \le L_0 C_\sharp \Lambda_\sharp^\ell \ell!.
    \end{equation*}
    Then applying Theorem \ref{flux-correctors_coro_Ubound} under the scale-separation condition to the equation \eqref{eq.DxU-DxG}, we obtain \eqref{est.DxUxy}.
\end{proof}

Let $\U = \U(x,y_1,\cdots, y_n)$ be given by \eqref{eq.Uxy}. Define the multiscale flux corrector $\Phi=\Phi(x,y_1,y_2,\cdots,y_n)$ as follows,
\begin{align}\label{flux-co-yn}
    \Phi_{\ell ij}=(\widehat{\nabla}_n)_\ell \U_{ij}-(\widehat{\nabla}_n)_i \U_{\ell j},
\end{align}
where $(\widehat{\nabla}_n)_i= \sum_{k=1}^n \delta_k^{-1} \pa_{y_k^i},$ and $y^i_k, 1\le i\le d,$ denotes the $i$-th component of $y_k\in\mathbb{T}^d$. Obviously, $\Phi$ is antisymmetric in the sense that $\Phi_{\ell ij} = - \Phi_{ i \ell j}$.
Define $\U_\de(x,y)=\U(x,y,y/\de_2, \cdots, y/{\de_n})$ and $ \Phi_\de(x,y)= \Phi(x,y,y/\de_2, \cdots, y/{\de_n})$.  We can write \eqref{flux-co-yn} as 
  \begin{align}\label{flux-co}
    \Phi_{\de,\ell ij}=\partial_\ell \U_{\de,ij}-\partial_i \U_{\de,\ell j} 
     = \partial_{y^\ell} \U_{\de,ij}-\partial_{y^i} \U_{\de,\ell j} , 
      \end{align} 
where $y^i, 1\le i\le d,$ denotes the $i$-th component of $y\in\mathbb{T}^d$. This together with  \eqref{flux-correctors_eq_approximateflux}  implies that 
\begin{align*}
  \begin{split}
    \partial_{y^\ell} \Phi_{\de,\ell ij}&=\Delta_y \U_{\de,ij}-\partial_{y^\ell}\partial_{y^i} \U_{\de,\ell j}\\&=\tau^2 \U_{\de,ij}+B_{\de,ij}+B_{\de,ik}\partial_{y^k} \mathcal{X}^j_\delta-\overline{A}-\partial_{y^i}\partial_{y^\ell} \U_{\de,\ell j}.
    \end{split}
    \end{align*}
    Alternatively, we can write this equation as
    \begin{equation}\label{eq-Bdelta}
        B_\delta + B_\delta \nabla_y \X_\delta - \overline{A} = \nabla_y \cdot \Phi_\delta + (\nabla_y (\nabla_y \cdot \U_\delta) - \tau^2 U_\delta ).
    \end{equation}
We would like to show that the last two terms of the above equation are small, under the scale-separation condition. Clearly, by Theorem \ref{thm.Ux-unif}, we have
\begin{equation}\label{eq-tau2U}
    \sup_{x\in \R^d} \| \tau^2 \U_\delta(x,\cdot) \|_{L^\infty} \le C\tau^2. 
\end{equation}
It remains to estimate $\nabla_y (\nabla_y \cdot \U_\delta)$.
  
By equations \eqref{eq.X} and \eqref{flux-correctors_eq_approximateflux}  we have
  \begin{align*}
    -\Delta_y \partial_{y^i} \U_{\de,ij}(x,y)+\tau^2 \partial_{y^i}\U_{\de,ij}(x,y)&=\partial_{y^i}(\overline{A}_{ij}(x)-B_{\de,ij}(x,y)-B_{\de,i\ell}\partial_{y^\ell} \mathcal{X}^{j}_\delta (x,y))\\
    &=-\tau^2 \mathcal{X}^j_\delta(x,y) \qquad \text{for } y\in \R^d,
  \end{align*}
whose lifted version can be written as
  \begin{align*}
    -\widehat{\Delta}_n(\widehat{\nabla}_n\cdot \U(x,\cdot))_j+\tau^2 (\widehat{\nabla}_n\cdot \U(x,\cdot))_j=-\tau^2 \mathcal{X}^j(x,\cdot) \qquad \text{in } \T^{d\times n}.
  \end{align*}
By the energy estimate in Proposition \ref{flux-correctors_prop_energy1}, it holds that
  \begin{align*}\|\bfD_n^\ell\widehat{\nabla}_n(\widehat{\nabla}_n\cdot \U)\|_{L^2}\leq \tau C_d\|\bfD_n^\ell \mathcal{X}  \|_{L^2} \le \tau C_d C_n\Lambda_n^{\ell}(\ell+1)!
  \end{align*}
 for all $0\le \ell \le \ell_0 = [nd/2]+1$, where in the last step we have used the uniform estimate \eqref{est.DlX.unif} under the scale-separation condition in Theorem \ref{thm.AppCorrector} or \ref{thm.AppCorrector.X}. Hence, by the Sobolev embedding theorem, we obtain that 
  \begin{align*}
    \|\nabla_y (\nabla_y\cdot \U_{\de}) \|_{L^\infty} \le \|\widehat{\nabla}_n(\widehat{\nabla}_n\cdot \U) \|_{L^\infty} \leq C \tau.
  \end{align*}
The above estimate actually holds for each $x\in \R^d$, and thereby we have
\begin{equation}\label{eq-DyDyU}
    \sup_{x\in \R^d} \|\nabla_y (\nabla_y\cdot \U_{\de})(x,\cdot)\|_{L^\infty} \le C\tau.
\end{equation}

Taking \eqref{eq-tau2U} and \eqref{eq-DyDyU} into \eqref{eq-Bdelta}, we obtain
\begin{corollary}\label{coro.fluxCorr}
    Under the assumption of Theorem \ref{thm.AppCorrector.X}, we have
    \begin{equation*}
        \sup_{x\in \R^d} \| B_\delta(x,\cdot) + B_\delta(x,\cdot) \nabla_y \X_\delta(x,\cdot) - \overline{A}(x) - \nabla_y \cdot \Phi_\delta(x,\cdot) \|_{L^\infty} \le C\tau.
    \end{equation*}
\end{corollary}

\section{Optimal convergence rates}
\label{sec.rate}
In this section, we will prove Theorem \ref{thm.MainRate}. The proof will be carried out in two steps. In the first step, we consider $n$ scales $\e_1 > \e_2 > \cdots > \e_n$ under the scale-separation condition \eqref{cond.ej.sepcond}. In this case, we can find the multiscale correctors and flux correctors as in the previous sections, and establish the convergence rate to an effective equation in one step that simultaneously homogenizes all the scales. 
Therefore, we will call this process simultaneous homogenization, which is in contrast to the reiterated homogenization that homogenizes one scale at a time and repeat the process until all scales are homogenized.

In the second step, we will consider the general situation that the scales may not be all separated, for which Theorem \ref{thm.AppCorrector} does not apply and we cannot find bounded correctors with full scales.
To handle this tricky situation, we need a new strategy to combine the simultaneous homogenization with reiterated homogenization.

\subsection{Under scale-separation conditions}
In this subsection, we consider the convergence rate for the equation
    \begin{equation}\label{eq51}
        \left\{ \begin{aligned}
            -\nabla\cdot A_\e(x) \nabla u_\e &= f \quad \text{in } \Omega,\\
            u_\e &= g \quad \text{on } \partial \Omega,
        \end{aligned}
        \right.
    \end{equation} under the scale-separation condition \eqref{cond.ej.sepcond}. Here $A_\e(x) =A(x, x/\e_1 , x/\e_2, \dots, x/\e_n )$, and $A=A(x,y_1,\cdots,y_n)$ satisfies \eqref{as.ellipticity}, \eqref{as.periodicity}, \eqref{ass.Ax}, and \eqref{ass.DxAx}. 
Under the above conditions, Theorem \ref{thm.AppCorrector} ensures that for some $\tau^2 \simeq \max_j \{ e^{-c\e_{j-1}/\e_j} \}$, the equation \eqref{eq.X}
admits a unique solution $\X^i_\de, 1 \le i\le d,$ in the form of
\begin{equation*}
    \mathcal{X}^i_\de(x,y)= \mathcal{X}^i(x, y, y/\delta_2,\cdots, y/\delta_n),
\end{equation*}
where $\delta_j = \e_j/\e_1$ and $\X^i$ is the solution to \eqref{eq.Xx} with the unit vector $v=e_i$.

Let $\overline{A}(x)$ be given by \eqref{hatA}, i.e.,
	\begin{align*}
		\overline{A}(x)=\int_{\mathbb{T}^d}\cdots \int_{\mathbb{T}^d}\big[A(x, y_1, \cdots, y_n) +A(x, y_1, \cdots, y_n) \widehat{\na}_n \mathcal{X} (x, y_1, \cdots, y_n) \big]dy_1\cdots dy_n.
	\end{align*}
Thanks to \eqref{ass.Ax}, \eqref{ass.DxAx} and \eqref{est.Ax.energy} with $\ell = 0$,
	we obtain that 
	\begin{align}\label{lip-Abar}
    \begin{split}
	 \| \overline{A} \|_{L^\infty} \le C\quad \text{ and	} \quad |\nabla_x \overline{A} \|_{L^\infty }\le CL_0.
    \end{split}
\end{align}
    Moreover, by a classical argument, $\overline{A}$ satisfies the ellipticity condition with the same constant $\lambda$ as $A$.

Let $u_0$ be the solution to the following problem
\begin{equation}\label{limit-eq}
    \left\{ \begin{aligned}
        -\na \cdot  \overline{A}(x)\nabla u_0 &= f \quad \text{in } \Omega,\\
        u_0 &= g \quad \text{on } \partial \Omega.
    \end{aligned}
    \right.
\end{equation}
Define 
\begin{equation}\label{def.w}
w_\e
=u_\e -u_0
-\e_1  \eta_{\e_1} \mathcal{X}_\de(x,x/\e_1) \nabla u_0 ,
\end{equation}
where  the cut-off function $\eta_{\e_1}$  in \eqref{def.w} is chosen so that
$\eta_{\e_1} \in C_0^\infty(\Omega)$, $0\le \eta_{\e_1}\le 1$,
$$
\aligned
& \eta_{\e_1}(x)=1\quad \text{  if  } x\in \Omega \text{  and dist} (x,\partial\Omega) \ge  4\e_1,\\
& \eta_{\e_1} (x)=0 \quad \text{  if dist} (x, \partial\Omega)\le 3\e_1,
\endaligned
$$
and $|\nabla \eta_{\e_1}| \le C \e_1^{-1}$.
Denote 
\begin{equation*}
\Omega_t =\big\{ x\in \Omega: \ \text{ dist}(x, \partial\Omega)< t \big\}.
\end{equation*}

\begin{lemma}\label{lemma-convergcen-sep}
Let $\Omega$ be a bounded Lipschitz domain in $\mathbb{R}^d$. Suppose that $A=A(x,y_1,\cdots,y_n)$ satisfies \eqref{as.ellipticity}, \eqref{as.periodicity}, \eqref{ass.Ax} and \eqref{ass.DxAx}. Assume $(\e_1,\e_2,\cdots, \e_n)$ satisfies the scale-separation condition \eqref{cond.ej.sepcond} and $\tau \simeq \max_j \{ e^{-c\e_{j-1}/\e_j} \}$ with $c_j$ and $c$ determined by Theorem \ref{thm.AppCorrector.X} (depending only on $d, n, \lambda, C_0 $ and $ \Lambda_0$).
	 Let $w_\e$ be defined by \eqref{def.w}.
   Then for any $\psi\in H^1_0(\Omega)$,
\begin{equation}\label{rel-c-s}
\aligned
& \Big|\int_\Omega A_\va
    \nabla w_\varepsilon\cdot \nabla \psi dx\Big| \\
&    \leq
C\e_1 \|\nabla\psi\|_{L^2(\Omega)}
\Big\{
L_0
\|\nabla u_0\|_{L^2(\Omega)}
+ \|\nabla^2 u_0\|_{L^2(\Omega\setminus\Omega_{3\e_1})} \Big\} \\
&\qquad +  C
    \|\nabla \psi \|_{L^2(\Omega_{5\e_1})}
    \|\nabla u_0\|_{L^2(\Omega_{4\e_1})} +  C\tau  \|\nabla\psi\|_{L^2(\Omega)}
\|\nabla u_0\|_{L^2(\Omega)},
\endaligned
\end{equation} 
where $L_0$ is given in \eqref{ass.DxAx} and the constant $C$ depends only on $d$, $n$, $\lm,$ $C_0$, $\Lm_0$ and $\Omega$. 
\end{lemma}
\begin{proof}
Note that  
\begin{equation*}
\aligned
-\nabla\cdot A_\e \nabla w_\e
&= \text{\rm div} \big[ ( A_\e +A_\e   (\nabla_y \mathcal{X}_\de)(x,x/\e_1)  -\overline{A}) \nabla u_0  \eta_{\e_1}\big]
+\text{\rm div} \big[
(A_\e   -\overline{A}) \nabla u_0  (1-\eta_{\e_1})\big]\\
&
\quad + \e_1\, \text{\rm div}
\big[ A_\e  \mathcal{X}_\de(x,x/\e_1) \nabla u_0\nabla \eta_{\e_1}  \big]
+\e_1\,  \text{\rm div} \big[
A_{\e}   (\nabla_x \mathcal{X}_\de)(x,x/\e_1)  \nabla u_0 \eta_{\e_1} \big]\\
&\quad +
\e_1\,  \text{\rm div}
\big[ A_\e   \mathcal{X}_\de(x,x/\e_1) \nabla^2 u_0 \eta_{\e_1} \big].
\endaligned
\end{equation*}
This implies that 
\begin{align}\label{pl-c-s-2}
\begin{split}
  \Big|\int_\Omega A_\e
    \nabla w_\varepsilon\cdot \nabla \psi dx\Big| 
    &
    \le \Big|\int_\Omega
   ( A_\e +A_\e   (\nabla_y \mathcal{X}_\de)(x,x/\e_1)  -\overline{A})  \eta_{\e_1}  \nabla u_0 \nabla \psi \, dx \Big|  \\
   &\qquad  + C  \int_\Omega
|  (1- \eta_{\e_1} ) (A_\e   -\overline{A})  \nabla u_0  |\, |\nabla \psi|\, dx \\
&\qquad +  C \e_1 \int_\Omega
|  (\nabla \eta_{\e_1} ) \mathcal{X}_\de(x,x/\e_1) \nabla u_0  |\, |\nabla \psi|\, dx\\
& \qquad + C \e_1 \int_\Omega
| \eta_{\e_1} (\na_x \mathcal{X}_\de)(x,x/\e_1)  \nabla u_0 |\, |\nabla \psi|\, dx\\
& \qquad + C \e_1 \int_\Omega
|\eta_{\e_1} \mathcal{X}_\de(x,x/\e_1) \nabla^2 u_0 | \, |\nabla \psi|\, dx\\
&\doteq I_1+\dots + I_5.
\end{split}
\end{align}
By the definition of $\eta_{\e_1}$, and the estimates for $\mathcal{X}_\de  $ and $ \na_x\mathcal{X}_\de $  in \eqref{chi.infty}, we have  
\begin{align*}
\begin{split}
I_2+\cdots+I_5  &\leq  C \|\na u_0\|_{L^2(\Omega_{4\e_1})} \|\na \psi\|_{L^2(\Omega_{4\e_1})} + C \e_1 \|\nabla^2 u_0\|_{L^2(\Omega\setminus\Omega_{3\e_1})}   \|\na \psi\|_{L^2(\Omega)}\\
&\quad+   C \e_1  L_0 \|\na u_0\|_{L^2(\Omega )}   \|\na \psi\|_{L^2(\Omega)}.
\end{split}
\end{align*}
It remains to  deal with $I_1$. In view of $A_\e(x) = B_\de(x,x/\e)$ and Corollary \ref{coro.fluxCorr}, we have
\begin{equation*}
    \| A_\e +A_\e   (\na_y \mathcal{X}_\de)(x,x/\e_1)  -\overline{A} - (\nabla_y\cdot \Phi_\delta) (x,x/\e_1) \|_{L^\infty} \le C\tau,
\end{equation*}
where $\Phi_\delta = (\Phi_{\delta, \ell i j})$ is given by \eqref{flux-co}.
Observe that
\begin{equation*}
    (\nabla_y\cdot \Phi_\delta) (x,x/\e_1) = \e_1 \nabla\cdot ( \Phi_\delta (x,x/\e_1) ) - \e_1 (\nabla_x \cdot \Phi_\delta) (x,x/\e_1).
\end{equation*}
By the integration by parts, we deduce that 
\begin{align*}
\begin{split}
I_1 
&\leq \e_1 \Big| \int_\Omega     \Phi_{\de,\ell ij}(x,x/\e_1) \pa_{x^j}u_0   \partial_{x^\ell} \pa_{x^i} \psi \eta_\e \Big|  \\
&\quad+ \e_1 \Big| \int_\Omega    \Phi_{\de,\ell ij}(x,x/\e_1)   \partial_{x^\ell}\pa_{x^j}u_0   \pa_{x^i} \psi  \eta_\e\Big|  \\
&\quad+ \e_1 \Big| \int_\Omega     \Phi_{\de,\ell ij}(x,x/\e_1) \pa_{x^j}u_0   \pa_{x^i} \psi  \partial_{x^\ell} \eta_\e \Big| \\ 
&\quad+ \e_1 \int_\Omega  \big|( \partial_{x^\ell} \Phi_{\de,\ell ij})(x,x/\e_1)\big| |\pa_{x^j}u_0 |\, |\pa_{x^i} \psi| \\
    &\quad+ \int_\Omega  C\tau  \, |\nabla u_0|\, | \nabla \psi| \\
   &\doteq I_{11}+\cdots+I_{15}.
\end{split}
\end{align*}
Since $\Phi_\delta$ is antisymmetric (i.e., $\Phi_{\delta, \ell ij} = - \Phi_{\delta, i \ell j}$), we have $I_{11}=0$.  By the definition of $\Phi_\delta$ \eqref{flux-co} and Theorem \ref{thm.Ux-unif}, 
$$\|\Phi_\delta \|_{L^\infty}\leq C \|\widehat{\na}_n \U\|_{L^\infty} \leq C  \quad \text{and} \quad  \|\na_x \Phi_\delta \|_{L^\infty}\leq C \|\na_x \widehat{\na}_n \U\|_{L^\infty} \leq C L_0. $$
It follows that 
\begin{align*}
\begin{split}
I_1 & \le I_{12} +  I_{13} + I_{14} + I_{15}  \\
& \leq  C\e_1 \big\{
L_0 
\|\nabla u_0\|_{L^2(\Omega)}
+ \|\nabla^2 u_0\|_{L^2(\Omega\setminus\Omega_{3\e_1})} \big\} \|\na \psi\|_{L^2(\Omega)}\\
& \qquad +  C
    \|\nabla \psi \|_{L^2(\Omega_{5\e_1})}
    \|\nabla u_0\|_{L^2(\Omega_{4\e_1})} + C \tau  \|\na u_0\|_{L^2(\Omega )}   \|\na \psi\|_{L^2(\Omega)}.
\end{split}
\end{align*}
By taking the estimates of $I_1$ -- $I_5$ into \eqref{pl-c-s-2}, we derive \eqref{rel-c-s} immediately. 
\end{proof}

\begin{remark} \label{remark5.01}
As a consequence of the lemma above, we have the error estimate in $H^1(\Omega)$, 
\begin{align}\label{H-1-est}
\begin{split}
\|  w_\e\|_{H^1(\Omega)}
&\le C   \e_1^{1/2} \| u_0\|^{1/2}_{H^2(\Omega)} \| \nabla u_0\|^{1/2}_{L^2(\Omega)} +C \tau \|\nabla u_0\|_{L^2(\Omega)} \\
&\quad  +C\e_1  \big\{\| u_0\|_{H^2(\Omega)}
+ L_0 \|\nabla u_0\|_{L^2(\Omega)}\big\}.
\end{split}
\end{align} 
Indeed, by taking  $\psi =w_\e$ in \eqref{rel-c-s} we deduce  that  
\begin{align}\label{L-2-bl}
\begin{split}
\|  w_\e\|_{H^1(\Omega)}
&\le C \e_1 \big\{
L_0  \|\nabla u_0\|_{L^2(\Omega)}
+ \|\nabla^2 u_0\|_{L^2(\Omega\setminus \Omega_{3\e_1} )} \big\}\\
&\quad + C \| \nabla u_0\|_{L^2(\Omega_{4\e_1})}  +  C\tau  \|\nabla u_0\|_{L^2(\Omega)} .
\end{split}
\end{align}
Recall that for any $u\in H^1(\Omega)$ and $t>0$,
\begin{equation}\label{bl-est}
\| u\|_{L^2(\Omega_t)}
\le C t^{1/2} \| u\|_{L^2(\Omega)}^{1/2}  \| u\|^{1/2}_{H^1(\Omega)}.
\end{equation}
This together with \eqref{L-2-bl} gives  \eqref{H-1-est}.
\end{remark}

\begin{remark} \label{remark5.2}
{\rm In view of \eqref{lip-Abar}, if $\Omega$ is $C^{1, 1}$ the energy estimate and 
the $H^2$ estimate for $u_0$ imply that 
\begin{align*}
&\| u_0\|_{H^1(\Omega)}
\le C
\big(
\| f\|_{L^2  (\Omega)}
+ \|g\|_{H^{1/2}(\partial\Omega)}
\big),\\
 &\| u_0\|_{H^2(\Omega)}
\le C  (L_0 +1)
\big(
\| f\|_{L^2(\Omega)}
+ \| g\|_{H^{3/2}(\partial\Omega)}
\big).
\end{align*}
 This combined with \eqref{H-1-est} gives  
\begin{equation}\label{remark5.2.re}
 \|  w_\e\|_{H^1(\Omega)}\\
 \le
C \Big\{\e_1^{1/2}  \big(
 1  +  L_0 ^{1/2}
 +\e_1^{1/2}   L_0 \big) +\tau \Big\}
\big(
\|f\|_{L^2(\Omega)}
+ \|g\|_{H^{3/2}(\partial\Omega)}
\big),
\end{equation}
where $C$ depends only on $d$, $n$, $\lm,$ $C_0$, $\Lm_0$ and $\Omega$.}
\end{remark}
	
The following theorem establishes the optimal convergence rate in $L^2(\Omega)$ under the scale-separation condition.
\begin{theorem}\label{thm.con.sep}
In addition to the assumption of Lemma \ref{lemma-convergcen-sep}, 
 let $\Omega$ be a bounded $C^{1,1}$ domain in $\mathbb{R}^d$. 
 Let $u_\e$ be a weak solution of \eqref{eq51}, 
 and $u_0$ the solution of \eqref{limit-eq} with  $f\in L^2(\Omega)$ and $g\in H^{3/2}(\partial\Omega)$.
  Then
  \begin{equation}\label{thm.con.sep.re}
  \|u_\e-u_0\|_{L^{2}(\Omega)}\\
 \leq C
  \big\{\e_1 \big (
  1 +
  L_0 +\e_1 L_0^2\big) +\tau \big(  
  1 +
  \e_1  L_0\big ) 
 \big\}
 \big( \| f\|_{L^2(\Omega)}
 + \|g\|_{H^{3/2}(\partial\Omega)}\big),
  \end{equation}
 where $\tau^2 \simeq \max_j \{ e^{-c\e_{j-1}/\e_j} \}$ with $c$ depending only on $d, n, \lambda, C_0 $ and $ \Lambda_0$, and the constant $C$ depends additionally on $\Omega$. 
\end{theorem}

\begin{proof}
By the definition of  $w_\e$, it suffices to 
 show that $\|w_\e\|_{L^2(\Omega)}$ is bounded by the right-hand side of \eqref{thm.con.sep.re}. 
We perform a duality argument 
 as in  \cite{suslinaD2013} (also see \cite{SZ17, shenbook1}). 
Let $ A^*(x,y_1,\cdots,y_n)$ be the adjoint of ${A}(x,y_1,\cdots,y_n)$.
Note that $A^*$ satisfies the same conditions as $A$.
Let $ \mathcal{X}^*$ and $\Phi^*$ be the corresponding correctors and flux correctors respectively, and 
    $\overline{ A^*}$  be defined by \eqref{hatA},  with $A$ and $ \mathcal{X}$  replaced by $A^*$ and $ \mathcal{X}^*$, respectively.   Note that $ \overline{A^*}= \overline{A}^*$.

 For $G\in C_0^\infty(\Omega)$, let $v_\e$ be the weak solution of the following Dirichlet problem,
\begin{equation*} 
  -  \na \cdot   {A}^*(x, x/\e_1, \cdots,x/\e_n) \na v_{\e}    =G  \quad\text{in } \Omega \,\quad \text{ and }  \,\quad 
 v_\e=0  \quad \text{ on } \partial\Omega,
 \end{equation*}
and $v_0$ the homogenized solution to the problem \eqref{limit-eq} with $\overline{A}$ replaced by $\overline{ A^*}.$ 
Define
 \begin{align*}
\widetilde{ w}_\e (x)=&v_\e-v_0-\e_1  \mathcal{X}^*_\de(x, x/\e_1)   \nabla v_0  \widetilde {\eta}_{\e_1} ,
\end{align*}
where  $\widetilde{\eta}_{\e_1} \in C_0^\infty(\Omega)$ is a cut-off function
such that $0\le \widetilde{\eta}_{\e_1} \le 1$,
$$
\widetilde{\eta}_{\e_1} (x)=1   \text{ in }  \Omega\setminus\Omega_{10\e_1},\quad
\widetilde{\eta}_{\e_1} (x)=0   \text{ in } \Omega_{ 8 \e_1 },
$$
and $|\nabla \widetilde{\eta}_{\e_1} | \le C \e_1^{-1}$. By Remark \ref{remark5.2}, we have
\begin{equation}\label{est.v0H1H2}
    \| v_0 \|_{H^1(\Omega)} \le C\| G\|_{L^2(\Omega)}, \qquad \| v_0 \|_{H^2(\Omega)} \le C(1+L_0) \| G\|_{L^2(\Omega)},
\end{equation}
and
\begin{align}\label{4.12}
 \| \widetilde{w}_\va\|_{H^1 (\Omega)} \leq 
 C \big\{\e_1^{1/2}  \big(
 1  +  L_0^{1/2}
 +\e_1^{1/2}   L_0 \big) +\tau  \big\}
  \| G\|_{L^2(\Omega)}.
  \end{align}
The proof of (\ref{4.12}) is the same as that of \eqref{remark5.2.re}.

Now, note that
\begin{align*}
  \Big|\int_\Omega w_\varepsilon\cdot G\,  dx\Big|
  &=\Big|\int_\Omega {A}_\e (x)\nabla w_\varepsilon\cdot\nabla v_\varepsilon \, dx\Big|\nonumber\\
  &\leq \Big|\int_\Omega {A}_\e(x)\nabla w_\varepsilon\cdot\nabla \widetilde{w}_\varepsilon \, dx\Big|
  +\Big|\int_\Omega {A}_\e(x)\nabla w_\varepsilon\cdot\nabla v_0 \, dx\Big|\nonumber\\
  &\quad
  +\e_1 \Big|\int_\Omega {A}_\varepsilon(x)\nabla w_\varepsilon\cdot
  \nabla\big[
   \mathcal{X}^*_\de(x, x/\e_1) \nabla v_0 \widetilde{\eta}_{\e_1}  \big] dx\Big|\nonumber\\
  &\doteq J_1+J_2+J_3.
\end{align*}

By using the Cauchy inequality, \eqref{remark5.2.re} and \eqref{4.12}, we obtain
\begin{equation*} 
\aligned
J_1
& \le C \|\nabla w_\e \|_{L^2(\Omega)}
\|\nabla \widetilde{w}_\e \|_{L^2(\Omega)}\\
& \le
C \big\{\e_1 \big(1  +    L_0 + \e_1 L_0^2\big) +\tau^2
 \big\}
\big( \|f\|_{L^2(\Omega)}
+ \|g\|_{H^{3/2}(\partial\Omega)}\big)
\| G \|_{L^2(\Omega)}.
\endaligned
\end{equation*}

Next, we use \eqref{rel-c-s} to obtain 
\begin{equation}\label{4.13}
 \aligned
 J_2
 &\le C\e_1 \|\nabla v_0\|_{L^2(\Omega)}
\big\{L_0 \|\nabla u_0\|_{L^2(\Omega)}
+ \|\nabla^2 u_0\|_{L^2(\Omega\setminus\Omega_{3\e_1})} \big\} \\
&\quad +  C
    \|\nabla v_0\|_{L^2
    (\Omega_{5\e_1})}
    \|\nabla u_0\|_{L^2(\Omega_{4\e_1})} +  C\tau  \|\nabla v_0\|_{L^2(\Omega)}
\|\nabla u_0\|_{L^2(\Omega)}.
 \endaligned
 \end{equation}
 Note that by (\ref{bl-est}),
 $$
  \|\nabla v_0\|_{L^2(\Omega_{5\e_1})}
 \|\nabla u_0\|_{L^2(\Omega_{4\e_1})}
 \le C \e_1 \| \nabla v_0\|_{L^2(\Omega)}^{1/2} \| v_0\|_{H^2(\Omega)}^{1/2}
 \| \nabla u_0\|_{L^2(\Omega)}^{1/2}
 \| u_0\|_{H^2(\Omega)}^{1/2}.
 $$
 This, together with (\ref{4.13}), \eqref{est.v0H1H2} and Remark \ref{remark5.2},
 gives
 \begin{equation*} 
 J_2 \le C \big\{\e_1 (1+L_0) +\tau  \big \}
 \big( \|f\|_{L^2(\Omega)}
+ \|g\|_{H^{3/2}(\partial\Omega)}\big)
\| G \|_{L^2(\Omega)}.
 \end{equation*}

The estimate of $J_3$ is similar to that of $J_2$.
By \eqref{rel-c-s} we see that
\begin{align*} 
\begin{split}
J_3
&\le C\e_1^2  \| \nabla   \big( \mathcal{X}^*_\de 
    \nabla v_0 \widetilde{\eta}_{\e_1} \big)  \|_{L^2(\Omega)}
  \big\{
  L_0 \|\nabla u_0\|_{L^2(\Omega)}
  +\|\nabla^2 u_0\|_{L^2(\Omega)} \big\}\\
  &\quad  + C\tau \e_1 \| \nabla \big( \mathcal{X}^*_\de
    \nabla v_0 \widetilde{\eta}_{\e_1}  \big) \|_{L^2(\Omega)} \|\na u_0\|_{L^2(\Omega)}. \end{split}
 \end{align*}
   Note that by the estimates of $\mathcal{X}^*_\de $ as in \eqref{chi.infty} and \eqref{naxchi.infty},
  $$
  \aligned
   &  \| \nabla\big[
  \mathcal{X}^*_\de \nabla v_0 \widetilde{\eta}_{\e_1} \big]\|_{L^2(\Omega)}\\
  &\le \| 
   (\nabla \widetilde{\eta}_{\e_1})  \mathcal{X}^*_\de  \nabla v_0 \|_{L^2(\Omega)}
  + \|  \widetilde{\eta}_{\e_1} (\na_x \mathcal{X}^*_\de)(x,x/\e_1) \nabla v_0 \big] \|_{L^2(\Omega)}\\
&\quad
  + \e_1^{-1} \|  \widetilde{\eta}_{\e_1}  (\na_y \mathcal{X}^*_\de)(x,x/\e_1) \nabla v_0 \|_{L^2(\Omega)}
  + \|  \widetilde{\eta}_{\e_1} \mathcal{X}^*_\de \nabla^2 v_0 \|_{L^2(\Omega)}\\
  &\le C \e_1^{-1} \|\nabla v_0\|_{L^2(\Omega)} +  C   L_0 \|\nabla v_0\|_{L^2(\Omega)} 
  + C \|\nabla^2 v_0\|_{L^2(\Omega)} \\
  & \le C(\e_1^{-1} + 1+L_0) \| G \|_{L^2(\Omega)}.
  \endaligned
  $$
  It follows that
$$
\aligned
J_3  & \le C \e_1 \big\{ \big (1+ \e_1 L_0 \big)\|G \|_{L^2(\Omega)} \big\} \cdot\big\{ L_0 \| \nabla u_0\|_{L^2(\Omega)}
+ \|\nabla^2 u_0\|_{L^2(\Omega)} \big\}\\
&\quad+ C \tau \big\{ \big (1+ \e_1L_0 \big)\|G \|_{L^2(\Omega)} \big\}  \|\nabla u_0\|_{L^2(\Omega)} \\
&\le C \big\{  \e_1 \big(1+L_0 \big) + \tau \big\}\big(1+\e_1 L_0\big) 
 \big( \|f\|_{L^2(\Omega)}
+ \|g\|_{H^{3/2}(\partial\Omega)}\big)
\| G \|_{L^2(\Omega)}.
\endaligned
$$
By combining  the estimates of $J_1, J_2$ and $J_3$, we obtain for any $G \in C_0^\infty(\Omega)$,
$$
\aligned
   \Big|\int_{\Omega}w_\varepsilon\cdot G  \, dx\Big| 
&   \le C
  \big\{ \e_1 \big(
  1 +  L_0  +\e_1 L_0^2    \big) +\tau \big(
  1 +\e_1  L_0\big) \big\}\\
  &\quad \times \big\{ \|f\|_{L^2(\Omega)} +  \|g\|_{H^{3/2}(\partial\Omega)} \big\}  \|G\|_{L^2(\Omega)},
\endaligned
$$
from which the desired estimate for $w_\e$  follows by duality.
\end{proof}

\begin{remark}
    In particular, if $A = A(y_1,\cdots, y_n)$ is independent of $x$, then \eqref{thm.con.sep.re} holds with $L_0 = 0$. This is exactly \eqref{est.MainRate}. Thus, we have proved Theorem \ref{thm.MainRate} under the extra scale-separation condition \eqref{cond.ej.sepcond}.
\end{remark}

\subsection{Remove scale-separation conditions}\label{sec.Remove-Scale-Separation}
In this subsection, we will remove the scale-separation conditions and fully prove Theorem \ref{thm.MainRate}. Note that in the case of two scales, the scale-separation condition \eqref{cond.ej.sepcond} reads $\e_2/\e_1 \le c$ for some $c>0$. This condition can be assumed automatically since otherwise the convergence rate is trivial. Thus we focus on the cases of more than two scales.

Assume that the scales $(\e_1, \cdots, \e_n)$ do not satisfy the scale-separation condition \eqref{cond.ej.sepcond}.
Starting from the smallest scale $\e_n$ and with at least three scales, i.e., $(\e_{n-2}, \e_{n-1}, \e_n)$, we examine if the scale-separation condition is satisfied for these scales. If the scale-separation condition is satisfied, then we add one more scale $\e_{n-3}$ and examine if the scales $(\e_{n-3}, \e_{n-2}, \e_{n-1}, \e_n)$ satisfy the scale-separation condition. Generally, we can keep adding scales until the scale-separation condition is not satisfied. This process helps us locate the scale gap where the scale-separation condition fails and therefore we can apply the reiterated homogenization at this scale gap.
In order to make this idea working for obtaining the optimal convergence rate, we need another subtle modification. In the initial case of three scales, we will examine a slightly stronger separation condition by replacing $c_j$ by $2^{3-n} c_j$ in \eqref{cond.ej.sepcond}. Then each time we add one more scale in, we increase the constant by multiplying a factor $2$. Since we will add at most $n-3$ scales, the constant will increase at most to $ c_j$. Now suppose the above process stops at $\e_{n-m}$ for some $m$ with $2\le m\le n$, namely, $(\e_{n-m+1}, \cdots, \e_n)$ satisfies the scale-separation condition with constant $2^{m-n} c_j$ and $(\e_{n-m}, \e_{n-m+1}, \cdots, \e_n)$ does not satisfy the scale-separation condition with constant $2^{m+1-n}c_j$. 
Precisely, this means that $m \ge 2$ is the largest integer such that for all $n-m+2 \le j\le n$
\begin{equation}\label{est.msej}
    \e_j \le \frac{2^{m-n} c_j \e_{j-1}}{1+ \log(\e_{n-m+1}/\e_{j-1})}.
\end{equation}
Moreover, there exists some $n-m+1 \le p \le n$ such that
\begin{equation}\label{est.msep}
    \e_{p} > \frac{ 2^{m+1-n} c_{p} \e_{p-1} }{1+ \log(\e_{n-m} /\e_{p-1} )}.
\end{equation}
Note that if \eqref{est.msej} holds for $m = n$, then $(\e_1,\cdots, \e_n)$ satisfies the scale-separation condition \eqref{cond.ej.sepcond}.

The following lemma is crucial in removing the scale-separation condition in our main theorem.
\begin{lemma}\label{lem.sep.group} 
    Let $2\le m< n$ be selected as above.  Then we 
    have either $$\e_{n-m+1} > 2^{m+1-n} c_{n-m+1} \e_{n-m},\quad \text{i.e., } p = n-m+1\text{ in \eqref{est.msep}},$$ or  \begin{equation}\label{est.ep-2}
    p \ge n-m+2 \quad\text{and}\quad  
        \e_p > \frac{2^{m-n} c_p \e_{p-1}}{1+ \log(\e_{n-m}/\e_{n-m+1})}.
    \end{equation}
\end{lemma}

\begin{proof}
    It suffices to consider $p \ge n-m+2$. Then by \eqref{est.msej} and \eqref{est.msep}, we have
    \begin{equation*}
        \frac{ 2^{m+1-n} c_{p} \e_{p-1} }{1+ \log(\e_{n-m} /\e_{p-1} )} \le \e_p \le \frac{2^{m-n} c_p \e_{p-1}}{1+ \log(\e_{n-m+1}/\e_{p-1})}.
    \end{equation*}
    This yields
    \begin{equation*}
        1+ \log(\e_{n-m}/\e_{p-1}) \le  2\log(\e_{n-m}/\e_{n-m+1}).
    \end{equation*}
    Inserting this into \eqref{est.msep}, we arrive at \eqref{est.ep-2}.
\end{proof}

We are now ready to provide the proof of Theorem \ref{thm.MainRate}. 

\begin{proof} [\textbf{Proof of Theorem \ref{thm.MainRate}}]
The case of $n=1$ is well-known \cite{BLP78,klsal2,shenbook1}.  
For the case $n=2$,  the desired estimate follows directly from Theorem \ref{thm.con.sep}. Indeed, when $n=2$, i.e, $A_\e=A(x/\e_1,x/\e_2)$, the scale-separation condition \eqref{cond.ej.sepcond} reads $\e_2 \le c\e_1$, which can be assumed automatically for otherwise the convergence rate is trivial.  Hence for $n=2$ we directly obtain the optimal convergence rate from Theorem \ref{thm.con.sep} with $L_0 = 0$ and $\tau \simeq e^{-c\e_1/\e_2}$,
\begin{equation*}
  \|u_\e-u_0\|_{L^{2}(\Omega)}
 \leq C
  \big(\e_1 + e^{-c\e_1/\e_2} \big)
 \big( \| f\|_{L^2(\Omega)}
 + \|g\|_{H^{3/2}(\partial\Omega)}\big). 
\end{equation*}

Let us consider the case $n\geq 3$ next. 
The proof combines the ideas of reiterated homogenization and simultaneous homogenization mentioned above. 
Assume the theorem is true if the number of scales is smaller than $n$. We need to show the theorem holds for $n$ scales.  Let $m \ge 2$ be the largest integer such that \eqref{est.msej} holds for all $n-m+2\le j \le n$. If $m = n$, as mentioned earlier, the scale-separation condition \eqref{cond.ej.sepcond} is satisfied. Then the desired convergence rate follows from Theorem \ref{thm.con.sep} with $L_0 = 0$.

Suppose now $m < n$. 
We use the idea of reiterated homogenization and consider the matrix
$$
\mathcal{A} (x, y_{n-m+1},\cdots, y_n) := A(x/\e_1, \dots, x/\e_{n-m},  y_{n-m+1}, \cdots, y_n).
$$
It is obvious that $\mathcal{A}$ satisfies the ellipticity condition \eqref{as.ellipticity} and is
1-periodic in $y_j$ for $n-m+1 \le j\le n$.

With an abuse of notation, let $\bfD_m = \bfD_{n,m} = (\nabla_{y_{n-m+1}}, \cdots, \nabla_{y_n})$. Then 
 \begin{equation*}
     \bfD_{m}^\ell \nabla_x \mathcal{A}(x,\cdot) = \sum_{i=1}^{n-m} \e_i^{-1} (\nabla_{y_i}  \bfD_{m}^\ell A)(x/\e_1, \cdots, x/\e_{n-m}, \cdot).
 \end{equation*}
Consequently, we have
\begin{equation*}
    \sup_x \|  \bfD_{m}^\ell \mathcal{A}(x,\cdot) \|_{L^\infty} \le C_0 \Lambda_0^\ell \ell!,
\end{equation*}
and
\begin{equation*}  
    \sup_x \|  \bfD_{m}^\ell \nabla_x \mathcal{A}(x,\cdot) \|_{L^\infty} \le \frac{C_0}{\e_{n-m}} \Lambda_0^{\ell+1} (\ell+1)!\le   \frac{C_1}{\e_{n-m}} (2\Lambda_0)^{\ell} \ell!,
\end{equation*}
where $C_1 = C_0 \Lambda_0$ and
we have used $\ell+1 \le 2^\ell$ in the last inequality.
Hence the matrix $\mathcal{A}$  satisfies the analyticity conditions \eqref{ass.Ax}  and \eqref{ass.DxAx} with $L_0 = C_1 \e_{n-m}^{-1}$ and $\Lambda_0$ replaced by $2\Lambda_0$. 
We now homogenize the scales $\e_{n-m+1},\cdots, \e_n$  simultaneously. Let $$\overline{\mathcal{A}}_{\e'} (x)=\overline{\mathcal{A}} (x/\e_1,\cdots, x/\e_{n-m})$$ be the matrix of effective coefficients given by Theorem \ref{thm.con.sep}, where $\e' = (\e_1,\cdots, \e_{n-m})$, and $v_{\e'}$ be the solution to
\begin{equation*}
    \left\{ \begin{aligned}
        -\na \cdot \overline{\mathcal{A}}_{\e'} \na v_{\e'} & =f  \quad\text{in  }  \Omega,\\
        v_{\e'} & =g \quad\text{on }  \pa\Omega.
    \end{aligned}
    \right.    
\end{equation*}
As $(\e_{n-m+1}, \cdots, \e_n)$ satisfies the scale-separation condition \eqref{est.msej}, which is a version of \eqref{cond.ej.sepcond} with $m$ scales, we apply Theorem \ref{thm.con.sep}, with $\e_1$ replaced by $\e_{n-m+1}$, to conclude that
 \begin{align}\label{pro.thm.MainRate1}
 \begin{split}
 &\|u_\e -v_{\e'} \|_{L^2(\Omega)}  
  \leq  C\bigg\{\frac{\e_{n-m+1}}{\e_{n-m}}
   + \max_{n-m+1 \le i\le n-1} \{ e^{-c\e_{i}/\e_{i+1}} \}
 \bigg\}
 \big\{ \|f\|_{L^2(\Omega)} +  \|g\|_{H^{3/2}(\partial\Omega)} \big\} .
 \end{split}
 \end{align}
 In view of Lemma \ref{lem.sep.group}, by our choice of $m$,
 we have either
 $$\e_{n-m+1} > c_0  \e_{n-m},$$ for some $c_0>0$,  or there exists $ p \ge n-m+2$ such that
    \begin{equation*}
        \e_p > \frac{2^{m-n} c_p \e_{p-1}}{1+ \log(\e_{n-m}/\e_{n-m+1})}.
    \end{equation*}
 In the first case, $\e_{n-m}$ and $\e_{n-m+1}$ are not separated well and therefore the convergence rate in \eqref{est.MainRate} is trivial as $\e_{n-m}/\e_{n-m+1}\simeq 1$.  In the second case,  $\e_{n-m+1}$ and $\e_{n-m}$ are separated better comparing to the separation between  $\e_{p}$ and $\e_{p-1}$,
 and we have
\begin{equation}\label{est.breakpoint}
    \frac{\e_{n-m+1}}{\e_{n-m}} \le C_0 e^{-c_p\e_{p-1}/\e_p},
\end{equation}
which together with \eqref{pro.thm.MainRate1} gives
\begin{align}\label{pro.thm.MainRate2}
 \begin{split}
  \|u_\e -v_{\e'}\|_{L^2(\Omega)}  \le  C
   \max_{n-m+1 \le i\le n-1} \{ e^{-c\e_{i}/\e_{i+1}} \}
  \big\{ \|f\|_{L^2(\Omega)} +  \|g\|_{H^{3/2}(\partial\Omega)} \big\}.
 \end{split}
 \end{align}
We note that the application of reiterated homogenization on the break point between $\e_{n-m+1}$ and $\e_{n-m}$ does not cause an essential loss on the convergence rate due to \eqref{est.breakpoint}.

On the other hand, since $\overline{\mathcal{A}}(y_1,\cdots,y_{n-m})$ satisfies the assumptions \eqref{as.ellipticity}--\eqref{as.analyticity}. By the inductive assumption, there exists a constant (effective) matrix $\overline{A}$,  and a unique  weak solution $u_0$ to the problem
\begin{equation*}
    \left\{ \begin{aligned}
        -\nabla\cdot \overline{A} \nabla u_0 & =f  \quad\text{in  }  \Omega,\\
        u_0 & =g \quad\text{on }  \pa\Omega,
    \end{aligned}
    \right.      
\end{equation*}
such that
    \begin{equation*}
        \| v_{\e'} - u_0 \|_{L^2(\Omega)} \le C\big( \e_1 + \max_{1\le i\le n-m-1} \{ e^{-c\e_{i}/\e_{i+1}} \} \big) \big\{ \|f\|_{L^2(\Omega)} +  \|g\|_{H^{3/2}(\partial\Omega)} \big\}, 
    \end{equation*}
    where $C$ and $c$ are constants depending only on characters of $A$ and $\Omega$.
This together with \eqref{pro.thm.MainRate2}  gives  \eqref{est.MainRate}. The proof is complete.
\end{proof}

\begin{remark} \label{re-on-thm1} 
The main result of Theorem \ref{thm.MainRate} could be extended to the more general case $A_\e = A(x,x/\e_1,\cdots, x/\e_n)$ satisfying  conditions \eqref{as.ellipticity}, \eqref{as.periodicity}, \eqref{ass.Ax} and \eqref{ass.DxAx}.  In this case, the homogenized coefficient matrix $\overline{A}(x)$ depends on $x$ smoothly (Lipschitz), and the same optimal convergence rate \eqref{est.MainRate} can be established analogously with slight modifications. 
\end{remark}

\section{Uniform Lipschitz estimates}
\label{sec.Lip}
This section is devoted to the proof of Theorem  \ref{thm.lip.est} by the quantitative approach formulated in \cite{shenan2017}, originating from \cite{AS16}.

 \begin{lemma} \label{lem.appro}  Assume that $A=A(x,y_1,\cdots,y_n)$  satisfies  conditions \eqref{as.ellipticity}, \eqref{as.periodicity}, \eqref{ass.Ax} and \eqref{ass.DxAx}.  
 Let $u_\e$ be a weak solution to  
 \begin{equation*}
     -\nabla\cdot A(x,x/\e_1,\cdots, x/\e_n)\nabla u_\e = f \quad \text{ in } B_1.
 \end{equation*}
Then there exists a  matrix $\overline{A}(x)$ such that for each $r\in (\e_1, 1/2)$, there is a solution $u^r$ to 
\begin{equation}\label{eq.ur.inBr}
    -\nabla\cdot \overline{A}(x) \na u^r = f \quad \text{ in } B_r,
\end{equation}
such that
\begin{align} \label{lem.appro.re}
\begin{split}
    \bigg( \fint_{B_r} |u_\e - u^r|^2 \bigg)^{1/2} \le &C\bigg\{ \Big( \frac{\e_1}{r} \Big)^{1/2} + \max_{1\le i\le n-1} e^{-c\e_i/\e_{i+1}} \bigg\}\\&\quad\times \bigg\{ \bigg( \fint_{B_{2r}} |u_\e|^2 \bigg)^{1/2} + r^2\bigg( \fint_{B_{2r}} |f|^2 \bigg)^{1/2} \bigg\},
\end{split}
\end{align}
where the constants $C, c$ depend only on the characters of $A$ (including $L_0$). 
\end{lemma}

\begin{proof}
Fix $r\in (\e_1, 1/2)$ and let  $v (x)= u_\e (rx)$. It follows that $v$ satisfies the equation
  $$-\na \cdot  \widetilde{A} (x, x/\e'_1,\cdots, x/\e'_n) \nabla v = \widetilde{f} \quad\text{ in }  B_2 ,$$
where $\widetilde{A} (x, y_1,\cdots,y_n)=A(rx, y_1,\cdots, y_n)$, $\e'_i=\e_i/r$, and $\widetilde{f}(x)=r^2 f(rx)$. Also, observe that 
$$
\sup_{x\in \R^d} \| \bfD^\ell_n \widetilde{A}(x,\cdot) \|_{L^\infty}=   \sup_{x\in \R^d} \| \bfD^\ell_n A(x,\cdot) \|_{L^\infty}, 
$$ 
$$
\sup_{x\in \R^d} \| \bfD^\ell_n \na_x \widetilde{A}(x,\cdot) \|_{L^\infty}=  r \sup_{x\in \R^d} \| \bfD^\ell_n \na_x A(x,\cdot) \|_{L^\infty}.
$$ 
This shows that \eqref{lem.appro.re} can be rescaled to the case $r = 1$.
We therefore assume $r=1$ hereafter. 

Now, suppose that $-\na \cdot A (x, x/\e_1,\cdots, x/\e_n)\nabla  u_\e  =f$ in $B_{2}$. For $1<t<3/2$, by Remark \ref{re-on-thm1}, there exist a  matrix $\overline{A }(x) $ and the weak solution $u_0^t$ to 
$$
-\na \cdot \overline{A }(x) \na u^t_0= f \quad \text{in  }  B_{t}
\quad \text{ and }
\quad
u^t_0=u_\e \quad \text{on } \partial B_{t}, 
$$
such that 
\begin{equation*}
    \| u_\e - u_0^t \|_{L^2(B_1)} \le C (\e_1 + \max_{1\le i\le n-1} e^{-c\e_i/\e_{i+1}}) (\| u_\e\|_{H^{3/2}(\partial B_t)} + \| f\|_{L^2(B_t)} ).
\end{equation*}
The trivial energy estimates for both $u_\e$ and $u^t_0$ imply
\begin{equation*}
    \| u_\e - u_0^t \|_{L^2(B_1)} \le C \big( \| u_\e \|_{H^{1/2}(\partial B_t)} + \| f\|_{L^2(B_t)} \big).
\end{equation*}
An interpolation argument yields
\begin{equation*}
    \| u_\e - u_0^t \|_{L^2(B_1)} \le C \big(\e_1^{1/2} + \max_{1\le i\le n-1} e^{-\frac12 c\e_i/\e_{i+1}}\big)\big (\| u_\e \|_{H^{1}(\partial B_t)} + \| f\|_{L^2(B_t)} \big).
\end{equation*}

Let
\begin{equation*}
    u_0 = 2\int_1^{3/2} u_0^t dt.
\end{equation*}
Note that $u_0$ is still a solution of the homogenized equation \eqref{eq.ur.inBr} in $B_{1}$.
Then
\begin{equation*}
\begin{aligned}
    \| u_\e - u_0 \|_{L^2(B_1)}^2 & \le \int_1^{3/2} \| u_\e - u_0^t \|_{L^2(B_1)}^2 dt \\
    & \le C (\e_1 + \max_{1\le i\le n-1} e^{- c\e_i/\e_{i+1}}) \int_1^{3/2}( \| u_\e \|_{H^{1}(\partial B_t)}^2 + \| f\|_{L^2(B_t)}^2 )  dt \\
    & \le C (\e_1 + \max_{1\le i\le n-1} e^{- c\e_i/\e_{i+1}}) ( \| u_\e \|_{H^1(B_{3/2})}^2 + \| f \|_{L^2(B_{3/2})}^2 ) \\
    & \le C (\e_1 + \max_{1\le i\le n-1} e^{- c\e_i/\e_{i+1}}) ( \| u_\e \|_{L^2(B_{2})}^2 + \| f \|_{L^2(B_{2})}^2 ),
    \end{aligned}
\end{equation*}
where we have used the co-area formula and the Caccioppoli inequality.
\end{proof}

 The proof of Theorem \ref{thm.lip.est} relies on the following iteration lemma, which is a slight modification of \cite[Lemma 8.5]{shenan2017}.  
\begin{lemma}\label{shen.lem}
  Let $H(r)$ and $h(r)$ be two nonnegative continuous functions on the interval $(0,1]$ and let $\rho \in (0,1/4)$. Assume that
\begin{align}\label{Lip_cond_1}
\max_{r\leq t\leq 2r} H(t)\leq C_0H(2r), ~~~~~\max_{r\leq t, s\leq 2r} | h(t)-h(s)|\leq C_0H(2r),
\end{align}
for any $r\in [\rho, 1/2]$,  and  for some fixed $\theta \in (0,1/8)$,
\begin{align}\label{Lip_cond_2}
H( \theta  r) \leq \frac{1}{2} H(r) + C_0 (\omega (\rho/r)+\ga)\{ H(2r)+h(2r)\},
\end{align}
for any $r\in [\rho, 1/2]$, where 
$\omega$ is a nonnegative increasing function on $[0, 1]$, such that $\omega(0)=0$, and 
\begin{align}\label{Lip_cond_3}
\int_0^1 \frac{\omega(s)}{s} ds < \infty. \end{align}
Then exists $c_0$, depending only on $C_0$ and $\theta$,  such that if $\ga \le c_0 |\log \rho|^{-1}$, we have  \begin{align}\label{Lip_es_H}
\max_{\rho\leq r\leq 1} \{H(r)+h(r)\}\leq C \{H(1) +h(1)\} ,
\end{align}
where $C$ depends only on $C_0,  \theta,$ and $\omega$. 
\end{lemma}

 We now prove Theorem \ref{thm.lip.est}, deferring the proof of Lemma \ref{shen.lem} until the end of this section.

 \begin{proof}[Proof of Theorem \ref{thm.lip.est}]   
We prove the theorem in a more general setting, where the coefficient matrix $A= A(x,y_1,\cdots,y_n)$ satisfies  the assumptions \eqref{as.ellipticity}, \eqref{as.periodicity}, \eqref{ass.Ax} and \eqref{ass.DxAx} with some $L_0 \ge 0$. The constants throughout the proof are allowed to depend on $L_0$. 

Let $v_\e \in H^1(B_1)$ be a weak solution to \begin{align} \label{pro-thm12-1}
 -\nabla\cdot A(x,x/\e_1,\cdots,x/\e_n)\nabla v_\e = f   \quad \text{ in } B_1 
 \end{align} 
  with $f\in L^p(B_1), p>d$. Under the scale-separation condition \eqref{cond.SS4Lip}, we shall prove that 
 \begin{equation}\label{lip-vx}
        \| \nabla v_\e \|_{L^\infty(B_{1/2})} \le C\big( \| v_\e \|_{L^2(B_1)} + \| f \|_{L^p(B_1)} \big).
    \end{equation}
  Note that if   $A$ is independent of $x$, the conditions \eqref{ass.Ax} and \eqref{ass.DxAx} are reduced to \eqref{as.analyticity}. Thus the Lipschitz estimate \eqref{lip-vx} in the general setting implies \eqref{mian-re-thm2} in Theorem \ref{thm.lip.est}.

To verify \eqref{lip-vx}, by translation and dilation it suffices to verify  
 \begin{equation*} 
    |\nabla v_\e (0)| \le C\bigg( \fint_{B_1}  |\nabla v_\e|^2 \bigg)^{1/2} +\bigg( \fint_{B_1}  |f|^p \bigg)^{1/p},
\end{equation*}
which, by a simple blow-up argument, is a direct consequence of the large-scale estimate
\begin{align}\label{lip,largescale.en}
\bigg(\fint_{B_{\e_n}} |\na v_\e|^2\bigg)^{1/2}\leq C \bigg\{ \bigg(\fint_{B_1}
|  \nabla v_{\va}|^2\bigg)^{1/2} +
\bigg(\fint_{B_1} |f |^p\bigg)^{1/p}\bigg\}.
\end{align}
Let us next prove \eqref{lip,largescale.en} by induction. For $n=1$, the above estimate has been derived in \cite{NSX20}.   
 Assume that \eqref{lip,largescale.en} holds for $n-1$ scales, we show it is true for $n$ scales.

Let $v_\e$ be a weak solution of \eqref{pro-thm12-1}. Thanks to Lemma \ref{lem.appro}, there exists a matrix $\overline{A}(x)$ such that for any $\e_1 < r\le 1$, there exists a solution $v_0$ to $-\nabla\cdot \overline{A}(x) \na v_0 = f$ in $B_r$, such that \eqref{lem.appro.re} holds.     
For $0<r\leq 1$, we define  
  \begin{align}
  \label{defh}
  \begin{split}
  H(r;v_\e)=\frac{1}{r}\inf_{P\in \mathcal{P} }\bigg\{\bigg(\fint_{B_r}|v_\e-P|^2\bigg)^{1/2}
  +r^{2-d/p} |\na P| \bigg\}
   +r\bigg(\fint_{B_r}|f|^p\bigg)^{1/p},
  \end{split}
  \end{align} 
where   $\mathcal{P}$  denotes the linear space of affine functions. 
Note that  $\overline{A}(x)$ is Lipschitz with a constant independent of $\e$. By the $C^{1,\alpha}$  regularity of $v_0$, there exists some $\theta \in (0,1/8)$ such that (see \cite[Lemma 6.1]{NSX20})
$$ 
H(\theta r;v_0) \le  \frac{1}{2}H(r;v_0). 
$$
Fix such $\theta$. We deduce that  
      \begin{align}\label{vc-Lip_cond_1}
      \begin{split}
      H(\theta r;v_\e)&\leq \frac{1}{\theta r}\left(\fint_{B_{\theta r}}|v_{\va}-v_0|^2\right)^{1/2}+ H(\theta r; v_0)\\
       &\leq \frac{C}{ r}\left(\fint_{B_{r}}|v_{\va}-v_0|^2\right)^{1/2} + \frac{1}{2} H(r;v_0)\\
      &\leq C\bigg\{ \Big( \frac{\e_1}{r} \Big)^{1/2} + \max_{1\le i\le n-1} e^{-c\e_i/\e_{i+1}} \bigg\} \bigg\{\bigg( \frac{1}{r} \fint_{B_{2 r}}|v_{\va}-b|^2\bigg)^{1/2} +r \bigg(\fint_{B_{2r}} | f |^2\bigg)^{1/2}\bigg\}
      \\
      & \qquad +\frac{1}{2}H(r;v_0)
     \end{split}
    \end{align}
for any $b\in \mathbb{R}$, where we have used Lemma \ref{lem.appro}  in the last inequality above. 
 Let $P_r$  be the affine function  achieving the infimum in \eqref{defh}, and $h(r)=|\nabla P_r|$. 
By the Poincar\'{e}  inequality, we deduce that  
\begin{align*}
\begin{split}
    \inf_{b\in \mathbb{R }}  \frac{1}{r} \left(\fint_{B_{2 r}}|v_{\va}-b|^2\right)^{1/2} & \leq H(2r;v_\e)+\inf_{b\in \mathbb{R }} \frac{1}{r} \left( \fint_{B_{2r}}|P_{2r}-b|^2\right)^{1/2} \\
    &\leq H(2r;v_\e)+Ch(2r).
    \end{split}
\end{align*}
 This, combined with \eqref{vc-Lip_cond_1}, implies that $H(r)=H(r;v_\e)$ and $h(r)$ satisfy the assumption \eqref{Lip_cond_2}  of Lemma \ref{shen.lem} with 
$\omega(s) = s^{1/2} $ and  $\gamma=\max_{1\le i\le n-1} e^{-c\e_i/\e_{i+1}}.$ 
 
On the other hand, for $t\in[r, 2r]$, it is obvious that $H(t;v_\e)\leq CH(2r;v_\e)$. Moreover,  
\begin{align*}
|h(t)-h(s)| &\le |\nabla (P_t-P_s)|\leq \frac{C}{r} \left( \fint_{B_r}|P_t-P_s|^2\right)^{1/2}\\
&\leq \frac{C}{t} \left( \fint_{B_r}|v_{\va}-P_t|^2\right)^{1/2}+\frac{C}{s} \left( \fint_{B_s}|v_{\va}-P_s|^2\right)^{1/2}\\
&\leq C\{H(t)+H(s)\} \\
&\leq CH(2r)
\end{align*} 
 for all $t, s\in[r, 2r]$.
Therefore,  $H(r)$ and $h(r)$ satisfy the condition \eqref{Lip_cond_1}.

Finally, by the assumption \eqref{cond.SS4Lip},  we have 
$\frac{\e_i}{\e_{i+1}} \ge M \log \log \e_1^{-1}$  and therefore 
$$\gamma   \le e^{-c M \log\log \e_1^{-1}}  = |\log \e_1|^{-cM} \le c_0  |\log \e_1|^{-1},$$ 
where the last inequality holds provided that $M$ is large enough and $\e_1 < 1/10$.
Thanks to Lemma \ref{shen.lem}, we get the uniform Lipschitz estimate for $v_\e$ down to the scale $\e_1$,  
\begin{align}\label{lip,largescale.e1}
\begin{aligned}
\bigg(\fint_{B_{\e_1}} |\na v_\e|^2\bigg)^{1/2} & \leq C(H(2\e_1) + h(2\e_1)) \\
& \le C(H(1) + h(1)) \\
& \le C \bigg\{ \bigg(\fint_{B_1}
|  \nabla v_\e|^2\bigg)^{1/2} +
\bigg(\fint_{B_1} |f |^p\bigg)^{1/p} \bigg\},
\end{aligned}
\end{align}
where we have used the Caccioppoli and Poincar\'{e} inequalities.

To prove \eqref{lip,largescale.en}, we perform a rescaling argument and use the inductive assumption.  Let $\widetilde{u}_\e=u_\e(\e_1 x)$ and $ \widetilde{f} =\e_1^2 f(\e_1 x)$. It follows that 
$$-\nabla\cdot E(x,  x/\e'_1,\cdots,  x/\e'_{n-1})\nabla \widetilde{u}_\e = \widetilde{f},$$ where  $\e'_i=\e_{i+1}/\e_1, i=1,2,\cdots,n-1,$ and  $ E(x,  y_2,\cdots,  y_{n})= A(\e_1x, x, y_2,\cdots,   y_{n})$. Since $$ \sup_{x\in \R^d}|\bfD_{n-1}^\ell E(x,\cdot)|
 \le \sup_{x\in \R^d} |\bfD_{n}^\ell A(x,\cdot)| \le C_0\Lambda_0^\ell \ell!,
 $$ and 
\begin{align*}
\sup_{x\in \R^d}|\bfD_{n-1}^\ell \na_x E(x,\cdot)|  
&\le \sup_{x\in \R^d}    \e_1|\bfD^\ell_{n,n-1}  \na_x A(x,\cdot)| +|   \bfD^\ell_{n,n-1}\na_{y_1} A(x,\cdot)|  \\
&  \le\e_1\sup_{x\in \R^d} |\bfD_{n}^\ell \na_xA(x,\cdot)| + \sup_{x\in \R^d}|\bfD_{n}^{\ell+1} A(x,\cdot)|\\
&\le ( \e_1 L_0+C_0) \Lambda_0^{\ell+1} (\ell+1)!\\
&\le ( \e_1 L_0+C_0) \Lm_0  (2\Lambda_0)^\ell \ell!,
\end{align*}
where $ \bfD_{n,m} = (\nabla_{y_{n-m+1}}, \cdots, \nabla_{y_n}).$
We know that $E$  satisfies  \eqref{as.ellipticity}, \eqref{as.periodicity}, \eqref{ass.Ax} and  \eqref{ass.DxAx} (the constants are slightly different). 
Moreover,  by \eqref{cond.SS4Lip} for $(\e_1,\cdots, \e_n)$, 
\begin{align*}
 \frac{\e'_i}{\e'_{i+1}} =  \frac{\e_{i+1}}{\e_{i+2}} \ge M \log \log \e_{i+1}^{-1} \ge M \log \log \e_1 \e_{i+1}^{-1} = M \log \log {\e'_i}^{-1}.
\end{align*}
Thus the new scales $(\e'_1,\cdots,\e'_{n-1})$ also satisfy the scale-separation condition \eqref{cond.SS4Lip}.   
The inductive assumption implies that   
\begin{equation*}
    \bigg( \fint_{B_{\e'_{n-1}}}  |\nabla \widetilde{u}_\e|^2 \bigg)^{1/2} \le C \bigg\{\bigg( \fint_{B_1}  |\nabla \widetilde{u}_\e|^2 \bigg)^{1/2} + 
\bigg(\fint_{B_1} |\widetilde{f}|^p\bigg)^{1/p}\bigg\}.
\end{equation*}
By rescaling back to $v_\e$ it follows that
\begin{equation*}
    \bigg( \fint_{B_{\e_n}}  |\nabla v_\e|^2 \bigg)^{1/2} \le C \bigg\{\bigg( \fint_{B_{\e_1}}  |\nabla v_\e|^2 \bigg)^{1/2} +  \e_1 
\bigg(\fint_{B_{\e_1}}| f|^p \bigg)^{1/p}\bigg\},
\end{equation*}
which, combined with \eqref{lip,largescale.e1}, gives \eqref{lip,largescale.en}. The proof is  complete. 
  \end{proof}

We finally provide the proof of Lemma \ref{shen.lem}.
 \begin{proof}[Proof of Lemma \ref{shen.lem}]
By the second inequality of \eqref{Lip_cond_1}, we have 
\begin{align*}
h(r) \le h(2r) +C_0 H (2r) 
\end{align*} for any $ \rho\le r\le 1/2,$ from which we deduce that  
\begin{align*}
 \int_a^{1/2} \frac{h(r)}{r} dr \le \int_{2a}^{1} \frac{h(r)}{r} dr + C_0\int_{2a}^{1} \frac{H(r)}{r} dr,
\end{align*} 
for any $ \rho \le a \le 1/4.$ This implies that 
\begin{align*}
 \int_a^{2a} \frac{h(r)}{r} dr &\le \int_{1/2}^{1} \frac{h(r)}{r} dr + C_0\int_{2a}^{1} \frac{H(r)}{r} dr,\\
 & \le  C_0 (\log 2)  (h(1)+H(1))  + C_0\int_{2a}^{1} \frac{H(r)}{r} dr.
\end{align*} 
By this and the second inequality of \eqref{Lip_cond_1}, 
\begin{align*}
    h(a) &\le (\log 2)^{-1} \int^{2a}_{a} \frac{|h(r)-h(a)| +h(r)} {r}  dr\\
 & \le  C_0 (\log 2)^{-1}\bigg\{ H(2a) +    h(1)+H(1)   +  \int_{2a}^{1} \frac{H(r)}{r} dr\bigg\}.
 \end{align*}
Moreover, by the first inequality of \eqref{Lip_cond_1},
\begin{align}\label{pro-shen-lem-1}
\begin{split}
H(2a) & \le  C_0 (\log 2)^{-1} \int^1_{2a} \frac{H(r)}{r} dr,  \quad\text{ if }   \rho \le  a\le 1/8,\\
H(2a) & \le  C_0^2  H(1), \quad\text{ if }   1/8 \le  a\le 1/4.
\end{split}
\end{align}
Therefore we obtain that 
 \begin{align}  \label{pro-shen-lem-2}
 h(a)  \le  C_1 \bigg\{  h(1)+H(1)   +  \int_{a}^{1} \frac{H(r)}{r} dr\bigg\},
\end{align} where $C_1$ depends only on $C_0.$
 This together with \eqref{Lip_cond_2} and \eqref{pro-shen-lem-1} implies that 
 \begin{align*}  
 H(\theta r)  \le  \frac{1}{2} H(r) + C_2(\omega(\rho/r) +\ga)  \{ h(1)+H(1)\}   +  C_2 (\omega(\rho/r) +\ga)  \int_r^{1} \frac{H(t)}{t} dt,
\end{align*} 
where $C_2$ depends only on $C_0$ and $\theta.$
By using \eqref{Lip_cond_3}, dividing the last inequality by $r$ and integrating over $r\in (\nu \rho,1)$ for some $\nu > 1$, we  deduce that
\begin{align*} 
\begin{split}
 \int_{\nu \theta \rho} ^\theta \frac{ H( r)} {r} dr&\le  \frac{1}{2} \int_{\nu\rho} ^1 \frac{ H( r)} {r} dr + \{C_\nu +C_2 \ga \log \rho^{-1}\}  \{ h(1)+H(1)\}  \\
 & \quad+  C_2\int_{\nu \rho} ^1  \frac{\omega(\rho/r) +\ga}{r}  \bigg\{\int_r^{1} \frac{H(t)}{t} dt\bigg\} dr. 
 \end{split}
\end{align*}  
Note that 
\begin{align*}  
    \int_{\nu \rho} ^1  \frac{\omega(\rho/r) +\ga }{r}  \bigg\{\int_r^{1} \frac{H(t)}{t} dt\bigg\} dr
   &= \int_{\nu \rho} ^1 \frac{ H(t) }{t}  \bigg\{\int_{\rho/t}^{1/\nu} \frac{\omega(s) +\ga}{s}    ds\bigg\}   dt\\
   & \le \big\{(8C_2)^{-1}+   \ga \log \rho^{-1} \big\} \int_{\nu \rho} ^1 \frac{ H(t) }{t}  dt, 
\end{align*} 
for sufficiently large $\nu = \nu(\omega)$, where we have used the assumption \eqref{Lip_cond_3}. Here the value of $\nu$ depends on $\omega$ and $C_2$, and will be fixed hereafter. Setting $c_0 = (8C_2)^{-1}$ and using the assumption $\gamma < c_0 |\log \rho|^{-1}$, 
we obtain from the last two displayed inequalities that 
\begin{align*}  
\begin{split}
 \int_{\nu \theta \rho} ^\theta \frac{ H( r)} {r} dr&\le  \frac{3}{4} \int_{\nu\rho} ^1 \frac{ H( r)} {r} dr + \{C_\nu +C_2\}  \{ h(1)+H(1)\} . 
 \end{split}
\end{align*} 
Then it follows that 
\begin{equation*}
    \int_{\nu \theta \rho}^1 \frac{H(r)}{r} \le 4\int_\theta^1 \frac{H(r)}{r} dr + 4\{C_\nu +C_2\}  \{ h(1)+H(1)\}.
\end{equation*}
In view of \eqref{pro-shen-lem-1} and \eqref{pro-shen-lem-2}, this implies \eqref{Lip_es_H} for $\nu \theta \rho \le r \le 1$. Finally, for $\rho \le r\le \nu \theta \rho$, the desired estimate follows from the case $r = \nu \theta \rho$ and  \eqref{Lip_cond_1}.
\end{proof}

\section{Counterexamples}\label{sec.egs}

\subsection{Necessity of scale separation}\label{sec.eg1}
In our main result, we have obtained a fast convergence rate even under a mild (logarithmic) scale-separation condition.
Then one may naturally ask if it is possible to find a suitable constant effective matrix and establish a quantitative convergence rate or qualitative homogenization theorem for the multiscale problem without any separation condition. In the following, we will give a counterexample showing that a multiscale problem with non-separated scales may not converge to an effective problem with constant coefficients.

Let $d = 1$ and $n = 2$. Consider the second-order elliptic equation in one dimension (ODE)
\begin{equation}\label{eq.1d.eg}
    -\frac{d}{dx} a_\e(x) \frac{d u_\e}{dx} = f,
\end{equation}
where
\begin{equation*}
    a_\e(x) = \frac{1}{1 + 2\alpha \sin(2\pi x/\e_1) \sin(2\pi x/\e_2)},
\end{equation*}
and $0< \alpha < 1/2$. We let $\e_1 = \e \ll 1$ and $\e_2 = \e/(1+\e) < \e_1$. Thus, $a_\e$ has two oscillating 1-periodic scales which are not separated and actually are almost resonant. In the following, we will study the effect of multiplying two almost resonant waves (with high frequency) in homogenization.

By the trigonometric identity
\begin{equation*}
    \sin \theta \sin \gamma = \frac12 (\cos(\theta -\gamma) - \cos(\theta + \gamma)),
\end{equation*}
we see that
\begin{equation*}
    \frac{1}{a_\e(x)} = 1 +  \alpha ( \cos(2\pi x) + \cos(2\pi x/\e_3)), \quad \text{where } \e_3 = \frac{\e}{2+\e}.
\end{equation*}
Note that $\e_3 \le \e$. Thus $a_\e$ is converted into a function with one rapidly oscillating scale and the other component $\cos(2 \pi x)$ does not oscillate rapidly.
As $\e \to 0$, it is well known that the homogenized coefficient for the one-dimensional equation is
\begin{equation*}
    \bar{a} = ( \lim_{\e \to 0} a_\e(x)^{-1} )^{-1} = \frac{1}{ 1 +  \alpha  \cos(2\pi x) },
\end{equation*}
where the limit is understood as the weak limit in $L^2$.
This homogenized coefficient $\bar{a} = \bar{a}(x)$ relies on the  variable $x$. In fact, by the standard homogenization theory (or a straightforward computation for ODE), one can show that the solution $u_\e$ of \eqref{eq.1d.eg} converges to the solution of
\begin{equation}\label{eq.1d.effective}
    -\frac{d}{dx} \bar{a}(x) \frac{d \bar{u}}{dx} = f,
\end{equation}
with a convergence rate of $O(\e)$; see e.g. \cite{NSX20}. Clearly, the equation \eqref{eq.1d.effective} generally cannot be further approximated by any equations with constant coefficients.

The above example illustrates the necessity of the scale-separation condition for the existence of an algebraic convergence rate in general multiscale homogenization problems. But recall that in some special scenarios without the scale-separation conditions, it is still possible to obtain algebraic convergence rates , such as under a Diophantine condition \cite{Koz78, Shen15} (also see \cite{AGK16,SZ18}) or with a reperiodization technique \cite{NZ23}. Nevertheless, the method for these special scenarios cannot be extended to the general cases.

\begin{remark}
    Note that in the previous example if we consider initially $\e_2 = \e/(1+\beta \e) < \e_1 = \e$ with a fixed $\beta > 1$, eventually we will find $\bar{a}(x) = (1 + \alpha \cos(2\pi \beta x))^{-1}$, which is still rapidly oscillating if $\beta$ is relatively large. This provides an example showing that a two-scale problem \eqref{eq.1d.eg} quantitatively converges to a one-scale problem which is still rapidly oscillating. This phenomenon turns out to be generic. In fact, in \cite{NZ24}, we have shown that an elliptic equation with $n$ oscillating periodic scales can always be quantitatively approximated by another elliptic equation with at most $n-1$ oscillating scales. The above example shows that the result in \cite{NZ24} is the best we can expect in general.
\end{remark}

\subsection{Optimality of exponential rate}\label{sec.eg2}

We show that the exponential parts of the convergence rate $\exp(-c\e_i/\e_{i+1})$ in our main result are optimal in the sense that it cannot be true if any of them is replaced by $\exp(-C \e_i/\e_{i+1} )$ for some large $C > 0$.
Indeed, for any $\beta_0 \in \N$, we can find $\beta \in (\beta_0, \beta_0+1)$ and two scales $\e_1 = \e$ and $\e_2 = \e /\beta<\e_1$, and construct an analytic coefficient such that the solution to some elliptic equation (see \eqref{eq.1d.BVP} below) satisfies
\begin{equation*}
    \| u_\e - \bar{u} \|_{L^2} \ge c e^{-C\beta},
\end{equation*}
for any solution $\bar{u}$ of the equations with constant coefficient, where the constant $c, C>0$ are independent of $\e$ and $\beta$.

Let $\beta_0 $ be some large integer and $\beta = \beta_0 + \e \in (\beta_0, \beta_0+1)$. This choice of $\beta$ is to avoid the possible reperiodization argument introduced in \cite{NZ23}, which leads to a fast convergence rate. Consider a coefficient in the form of
\begin{equation*}
    a(y_1, y_2) = \frac{1}{1 + b_1(y_1)  b_2(y_2)}.
\end{equation*}
Here, we let
\begin{equation*}
    b_1(y_1) =  \frac{\beta_0 !}{ \beta_0^{\beta_0}} \sin(\beta_0 2\pi y_1), \qquad b_2(y_2) = \sin(2\pi y_2).
\end{equation*}
It is easy to verify that $b_1$ is analytic with $|\frac{d^k b_1}{d y_1^k}| \le (2\pi)^k k!$ for any $k \ge 0$. The same estimate holds also for $b_2$. Therefore, $a(y_1,y_2)$ is analytic in both $y_1$ and $y_2$ with quantitative estimates on all the derivatives. Now consider
\begin{equation*}
    a_\e(x) = a(x/\e, \beta x/\e) = \frac{1}{1 + b_1(x/\e)  b_2(\beta x/\e)},
\end{equation*}
and let $u_\e$ be the solution of
\begin{equation}\label{eq.1d.BVP}
    \left\{
    \begin{aligned}
        & -\frac{d}{dx} a_\e(x) \frac{d u_\e}{dx} = 1 \quad \text{in } (0,1),\\
        & u_\e(0) = u_\e(1) = 0.
    \end{aligned}
    \right.
\end{equation}
The explicit solution of the above ODE is given by
\begin{equation*}
    u_\e(x) = -\int_0^x \frac{t + p_\e}{a_\e(t)} dt,
\end{equation*}
for some constant $p_\e$ determined by $u_\e(1) = 0$. Using the trigonometric identity
\begin{equation*}
    \sin \theta \sin \gamma = \frac12 (\cos(\theta -\gamma) - \cos(\theta + \gamma)),
\end{equation*}
and $\beta - \beta_0 = \e$, we have
\begin{equation*}
    \frac{1}{\alpha_\e(t)} = 1 + \frac{\beta_0 !}{ \beta_0^{\beta_0}} \sin(\beta_0 2\pi t/\e) \sin(\beta 2\pi t/\e) = 1 + \frac12 \frac{\beta_0 !}{ \beta_0^{\beta_0}} 
 ( \cos(2\pi t) - \cos((\beta+\beta_0) 2\pi t/\e)).
\end{equation*}
Thus
\begin{equation*}
\begin{aligned}
    u_\e(x) & = -\int_0^x (t+p_\e)(1  + \frac12 \frac{\beta_0 !}{\beta_0^{\beta_0}} 
 ( \cos(2\pi t) - \cos((\beta+\beta_0) 2\pi t/\e) ) )dt  \\
 & = -\int_0^x (t+p_\e)(1  + \frac12 \frac{\beta_0 !}{ \beta_0^{\beta_0}} 
  \cos(2\pi t) )dt + O(\e/\beta).
\end{aligned}  
\end{equation*}
Moreover, using $u_\e(1) = 0$, one observes that $p_\e = -1/2 + O(\e/\beta)$. Thus, we get
\begin{equation*}
    u_\e(x) = \frac12 x(1-x) -  \frac{\beta_0 !}{ \beta_0^{\beta_0}} \frac{1}{8\pi^2} ( \pi (2x+1) \sin(2\pi x) + \cos(2\pi x) - 1 ) + O(\e/\beta).
\end{equation*}

Now note that any solution $\bar{u}$ of
\begin{equation*}
    \left\{
    \begin{aligned}
        & -\frac{d}{dx} \bar{a} \frac{d \bar{u}}{dx} = 1 \quad \text{in } (0,1), \\
        & \bar{u}(0) = \bar{u}(1) = 0,
    \end{aligned}
    \right.
\end{equation*}
with constant coefficient $\bar{a}>0$ is given by
\begin{equation*}
    \bar{u}(x) = \frac{1}{2\bar{a}} x(1-x).
\end{equation*}
Thus
\begin{equation*}
\begin{aligned}
    \inf_{\bar{a}>0} \| u_\e - \bar{u} \|_{L^2(0,1)} & \ge  \frac{\beta_0 !}{ \beta_0^{\beta_0}} \frac{1}{8\pi^2} \inf_{k\in \R} \| \pi (2x+1) \sin(2\pi x) + \cos(2\pi x) - 1 - kx(1-x) \|_{L^2(0,1)} - \frac{C\e}{\beta} \\
    & \ge c_1 \frac{\beta_0 !}{ \beta_0^{\beta_0}} - \frac{C\e}{\beta_0},
\end{aligned}
\end{equation*}
where $c_1>0$ and $C>0$ are absolute constants independent of $\e$ and $\beta$. The lower bound $c_1>0$ is guaranteed by the fact that the trigonometric function $\pi (2x+1) \sin(2\pi x) + \cos(2\pi x) - 1$ has a positive distance from the linear span of $x(1-x)$ in the Hilbert space $L^2(0,1)$. This lower bound can even be computed explicitly with a careful calculation.

Now 
recall that the Stirling's formula implies
\begin{equation*}
    \frac{\beta_0 !}{ \beta_0^{\beta_0}} \ge e^{-\beta_0}.
\end{equation*}
We assume $\e$ is small and $\beta_0 \le M |\log \e|$ for sufficiently large $M$. This is the typical case that $\exp(-c\e_1/\e_2)$ dominates the convergence rate and in this case
\begin{equation*}
    c_1 e^{-\beta_0} \ge \frac{2C\e}{\beta_0}.
\end{equation*}
As a result, we obtain
\begin{equation*}
    \inf_{\bar{a}>0} \| u_\e - \bar{u} \|_{L^2(0,1)} \ge \frac12 c_1 e^{-\beta_0} \ge ce^{-\beta},
\end{equation*}
for some $c>0$ independent of $\e$ and $\beta$. This proves the optimality of the exponential ratio terms in the convergence rate.

\subsection{Analyticity helps in a toy problem}\label{sec.eg3}
We use a toy example showing how the analyticity of coefficients helps improve the convergence rates. In view of the 1-D example in the previous subsections, it is reasonable to consider the averaging process of the product of two scalar functions oscillating at different scales. Let $f:\R\to \R$ be a 1-periodic analytic function and $g: \R \to \R$ be a 1-periodic locally $L^2$ function.
For $\beta \in (1,\infty)$, define
\begin{equation*}
    h_\beta(x/\e) = f(\frac{x}{\beta \e}) g(\frac{x}{\e}).
\end{equation*}
Clearly, $h_\beta$ is quasiperiodic. Let $\bar{h}_\beta$ be the mean of $h_\beta$. If we fix $\beta$, then as $\e \to 0$, $h_\beta$ converges to $\bar{h}_\beta$ in the weak sense. However, the convergence rate can be arbitrarily slow depending on the rationality/irrationality and magnitude of $\beta$. On the other hand, if $\beta$ is relatively large (say $\beta = |\log \e|$), we may use the idea of reiterated averaging/homogenization to show that $h_\beta$ is close to $\bar{f} \bar{g}$ in the weak sense, namely, we can quantify
\begin{equation*}
    \rho(\e,\beta) := \bigg|\int_0^1 h_\beta(x/\e) dx - \bar{f} \bar{g} \bigg|.
\end{equation*}
The natural question is to estimate this error. A typical idea of reiterated homogenization is to average $g$ first as it oscillates much faster than $f$, due to $\beta \gg 1$, and then average $f$. That is, we write
\begin{equation*}
    \rho(\e,\beta) \le \bigg|\int_0^1 f(\frac{x}{\beta \e}) g(\frac{x}{\e}) dx - \int_0^1 f(\frac{x}{\beta \e}) \bar{g} dx \bigg| + \bigg| \int_0^1 f(\frac{x}{\beta \e}) \bar{g} dx -  \bar{f} \bar{g} \bigg|,
\end{equation*}
and estimate these two errors correspondingly.
This will give an error of $ \beta^{-1} + \beta \e$ in general. 

In the following, we will use the analyticity of $f$ to sharpen this error to $\beta \e + e^{-c\beta}$, which is optimal in general.
Consider the Fourier series of $f$ truncated at $k > 0$ (to be determined):
\begin{equation*}
    f_k(y) = \sum_{|j| \le k } \hat{f}_j e^{-2\pi i jy},
\end{equation*}
where $\hat{f}_j = \int_{0}^1 f(y) e^{-2\pi i jy} dy$. Due to the real analyticity of $f$, we have
\begin{equation*}
    \| f - f_k \|_\infty \le Ce^{-ck},
\end{equation*}
for some $C$ and $c>0$ depending on $f$. Thus,
\begin{equation*}
\begin{aligned}
    \rho(\e,\beta) & \le \bigg| \int_0^1 \sum_{|j| \le k } \hat{f}_j e^{-2\pi i jx/(\beta \e)} \sum_{m\in \Z} \hat{g}_m e^{-2\pi i m x/\e}  - \bar{f} \bar{g} \bigg| + \bigg| \int_0^1 (f-f_k)(x/(\beta \e)) g(x/\e)  \bigg| \\
    & \le \bigg| \sum_{0\le |j|\le k, m\in \Z\setminus \{0\} } \hat{f}_j \hat{g}_m \int_0^1 e^{-2\pi i (\frac{j}{\beta} + m) x/\e } dx \bigg| \\
    & \qquad + \bigg| \sum_{0<  |j|\le k } \hat{f}_j \hat{g}_0 \int_0^1 e^{-2\pi i (\frac{j}{\beta}) x/\e } dx \bigg| + Ce^{-ck} \| g \|_2.
\end{aligned}
\end{equation*}
We notice that the worst term in the summation above is the one when $\frac{j}{\beta} + m$ is close to zero, which means those oscillating frequencies of $f$ and $g$ are almost resonant and canceled, leading to slow convergence. To avoid this situation, we can restrict $j$ such that $|j/\beta| \le 1/2$. Since $m$ is always a nonzero integer, we must have $|j/\beta + m| \ge |m| - 1/2 \approx |m|$. Hence we can handle all the nonresonant estimates of the summation if $k = [\beta/2]$. On the other hand, the real analyticity guarantees the tail error (possibly resonant terms) is exponentially small.

Precisely, we choose $k = [\beta/2]$. Then $|\frac{j}{\beta} + m| \simeq |m| $ for $|j| \le k$ and  $m \neq 0$. Hence, it is easy to verify that
\begin{equation*}
    \bigg| \int_0^1 e^{-2\pi i (\frac{j}{\beta} + m) x/\e } dx \bigg| \le \frac{C\e}{|m|} \quad \text{for any } m \neq 0.
\end{equation*}
Similarly, for $0<|j| \le k$,
\begin{equation*}
    \bigg| \int_0^1 e^{-2\pi i (\frac{j}{\beta} ) x/\e } dx \bigg| \le \frac{C\beta \e}{|j|}.
\end{equation*}
It follows that
\begin{equation*}
\begin{aligned}
    \rho(\e,\beta) & \le C\e \sum_{|j|\le [\beta/2]} |\hat{f}_j| \sum_{m\neq 0} |m|^{-1} |\hat{g}_m| +  C\beta \e \sum_{0<|j|\le [\beta/2]} |j|^{-1} |\hat{f}_j| |\hat{g}_0| + Ce^{-c\beta/2} \| g \|_2 \\
    & \le C(\beta \e + e^{-c\beta/2}) \| g\|_2.
\end{aligned}
\end{equation*}
The above argument explains how the real analyticity helps improve the error bound in a simple averaging process. In homogenization of elliptic equations, the mapping from coefficients to the solutions is much more complicated and how the analyticity plays a role is far from obvious.

\bibliographystyle{abbrv}
\bibliography{ref}
\end{document}